\documentclass{article}

\usepackage[all]{xy}
\usepackage{amssymb}
\usepackage{amsmath}
\usepackage{amsthm}
\newtheorem{thm}{Theorem}
\newtheorem{prop}{Proposition}
\newtheorem{lem}{Lemma}
\newtheorem{rem}{Remark}
\newtheorem{cor}{Corollary}
\newtheorem{defi}{Definition}
\newtheorem{defiprop}{Definition-Proposition}

\def\Top{\mathop{\rm Top}\nolimits}
\def\CH{\mathop{\rm CH}\nolimits}

\def\hom{\mathop{\rm hom}\nolimits}

\def\codim{\mathop{\rm codim}\nolimits}
\def\dim{\mathop{\rm dim}\nolimits}

\def\SmVar{\mathop{\rm SmVar}\nolimits}
\def\Gr{\mathop{\rm Gr}\nolimits}
\def\reg{\mathop{\rm reg}\nolimits}

\def\PSmVar{\mathop{\rm PSmVar}\nolimits}
\def\Var{\mathop{\rm Var}\nolimits}
\def\Hom{\mathop{\rm Hom}\nolimits}

\def\supp{\mathop{\rm supp}\nolimits}

\def\PVar{\mathop{\rm PVar}\nolimits}

\def\Tot{\mathop{\rm Tot}\nolimits}
\def\sing{\mathop{\rm sing}\nolimits}

\def\Im{\mathop{\rm Im}\nolimits}

\def\diff{\mathop{\rm diff}\nolimits}
\def\Cone{\mathop{\rm Cone}\nolimits}
\def\ad{\mathop{\rm ad}\nolimits}
\def\card{\mathop{\rm card}\nolimits}
\def\nul{\mathop{\rm nul}\nolimits}
\def\log{\mathop{\rm log}\nolimits}
\def\Diff{\mathop{\rm Diff}\nolimits}

\title{Some results on the higher Abel Jacobi map for open varieties}
\author{Johann Bouali}
 
\begin{document}

\maketitle

\begin{abstract}
In this article, we study the infinitesimal invariant of the relative higher Abel Jacobi map of a smooth open morphism.
We give a generalization of a theorem of Voisin to open algebraic varieties and higher Chow groups and 
as a corollary a non vanishing criterion for the higher Abel Jacobi map 
of an general open smooth hypersurface section of high degree of a smooth projective variety $Y$.
On the other side by Nori connectness theorem, the image of the primitive part of the higher Abel Jacobi map 
of a general open smooth hypersurface of high degree of $Y$ is, modulo torsion, generated
by the restriction to this open smooth hypersurface of a closed Bloch cycle in the corresponding affine subset of $Y$ whose
cohomology class in $Y$ is primitive.
\end{abstract}

%%%%%%%%%%%%%%%%%%%%%%%%%%%%%%%%%%%%%%%%%%%%%%%%%%%%%%%%%%%%%%%%
\section{Introduction}
%%%%%%%%%%%%%%%%%%%%%%%%%%%%%%%%%%%%%%%%%%%%%%%%%%%%%%%%%%%%%%%%

\textbf{Notations}: 
\begin{itemize}
\item We denote by $\Var(\mathbb C)$ the category of algebraic varieties over $\mathbb C$, 
$\SmVar(\mathbb C)\subset\Var(\mathbb C)$ the full subcategory of smooth algebraic varieties, 
$\PVar(\mathbb C)\subset\Var(\mathbb C)$, the full subcategory of projective varieties,
$\PSmVar(\mathbb C)\subset\Var(\mathbb C)$ the full subcategory of smooth projective varieties.
\item For $V\in\Var(\mathbb C)$, we denote by $V^{an}$ 
the complex analytic space associated to $V$ with the usual topology induced by $\mathbb C^N$.
By $V'\subset V$ an open subset, we mean an open subset of $V^{an}$ (i.e. an open subset for the usual topology).
\item For a sheaf $\mathcal F$ of abelian group on a locally compact Hausdorf topological space $V$, we denote by
$D^{\vee}(\mathcal F)$ the (Verdier dual) sheaf : 
for $V'\subset V$ an open subset $\Gamma(V',D^{\vee}(\mathcal F))=\Gamma_c(V',\mathcal F)^{\vee}$.
\item For $V\in\SmVar(\mathbb C)$, we denote by $O_V$ the sheaf of holomorphic function on $V^{an}$ 
and by $(\Omega^{\bullet}_V,\partial)$ the complex of sheaf of holomorphic forms on $V^{an}$. 
We denote by $(\mathcal A^{\bullet,\bullet}_V,\partial,\bar\partial)$ the bicomplex of sheaf of differential forms on $V^{an}$. 
The filtration $F$ associated to its total complex $(\mathcal A_V^{\bullet},d)$ is the Fr\"olicher filtration. 
We denote by $\mathcal D^{\bullet}_V=D^{\vee}(\mathcal A_V^{\bullet})$ the complex of sheaf of currents
on $V^{an}$ which is filetered by the Fr\"olicher filtration $F$.
\item For $V\in\Var(\mathbb C)$ and  $\mathcal F$ a sheaf of $O_V$ module on $V^{an}$ , we denote by
$D_{O_V}^{\vee}(\mathcal F)=\mathcal{H}om_{O_V}(\mathcal F,O_V)$ the dual sheaf of $O_V$ module on $V^{an}$ : 
for $V'\subset V$ an open subset, $\Gamma(V',D_{O_V}^{\vee}(\mathcal F))=\Hom(\mathcal F_{|V'},O_{V'})$.
\item For a complex $A^{\bullet}$ in an abelian category, we denote by $F_b$ the filtration b\^ete on it:
$F_b^pA^{\bullet}=A^{\bullet\geq p}$.
\item We denote by $\square^n=(\mathbb P^1\backslash \left\{1\right\})^n\subset(\mathbb P^1)^n$ and by
$\mathcal Z^p(X,n)\subset \mathcal Z^p(X\times\square^n)$ the subgroup of $p$ codimentional cycle in $X\times\square^n$ meeting
all faces of $\square^n$ properly.
We denote by $\pi_X:X\times(\mathbb P^1)^n\to X$ and $\pi_{(\mathbb P^1)^n}:X\times(\mathbb P^1)^n\to(\mathbb P^1)^n$ the projections.
\item For $V\in\Top$ a topological space, we denote by $C_{\bullet}^{\sing}(V,\mathbb Z)=\mathbb Z\Hom_{\Top}(\Delta^{\bullet},V)$ the complex
of singular chains, $\Delta^{p}\subset\mathbb R^p$ being the standard simplex.
For $V\in\Diff(\mathbb R)$ a differential manifold, we have an the inclusion of complexes 
$C_{\bullet}^{\diff}(V,\mathbb Z)=\mathbb Z\Hom_{\Diff(\mathbb R)}(\Delta^{\bullet},V)\subset C_{\bullet}^{\sing}(V,\mathbb Z)$
which is a quasi-isomorphism.
\end{itemize}

The Abel Jacobi map and normal functions associated to a family of algebraic cycles has been studied a lot for projective
varieties, but few appears in the literature for open varieties 
By an open variety, we mean a non complete algebraic variety, 
or most specificaly in our case a non projective quasi-projective variety. 
In this article we give generalization
of classical result for projective varieties to the case of open varieties. 

Every smooth open variety is the complementary subset of a normal crossing divisor in a smooth projective variety. 
For an open variety $U=X\backslash D$, with $X\in\PSmVar(\mathbb C)$ and $D\subset X$ a normal crossing, we have
$(\mathcal D_X^{\bullet}(\log D)),F)=(D^{\vee}(\mathcal A_X^{\bullet}(\nul D)),F)$ the complex  
of sheaves of $\log D$ currents on $X^{an}$ defined by King \cite{King} and $F$ is the Fr\"olicher filtration. 
A  $\log D$ current on an open subset $V\subset X$ is a linear form on the $\nul D$ differential forms with compact support on $V$.
The complex sheaves of $\nul D$ differential forms on $X^{an}$ is the subcomplex $\mathcal A^{\bullet}_X(\nul D)\subset\mathcal A_X^{\bullet}$ 
of differential forms on $X^{an}$ consisting of those which vanishes holomorphically on $D$.

The main goal of the first section is to note the $E_1$ degenerescence of the filtered complex $(\Gamma(X,\mathcal A_X^{\bullet}(\nul D)),F)$
where $F$ is the Fr\"olicher filtation and to reinterpret the Poincare duality paring 
\begin{equation}
<\cdot,\cdot>_{ev_X}:(H^k(U,\mathbb C),F)\otimes (H^{2d_X-k}(X,D,\mathbb C),F)\to\mathbb C \; \; 
\lambda\otimes\mu\mapsto (\lambda.\mu)([X])
\end{equation}
which is a morphism of mixed Hodge structure, as the one induced in cohomology by the pairing
\begin{equation}
<\cdot,\cdot>_{ev_X}:(\Gamma(X,\mathcal D_X^k(\log D)),F)\otimes (\Gamma(X,\mathcal A_X^{2d_X-k}(\nul D)),F)\to\mathbb C \; \; 
T\otimes\eta\mapsto T(\eta)
\end{equation}
To see the $E_1$ degenerescence of $(\Gamma(X,\mathcal A_X^{\bullet}(\nul D)),F)$, we 
prove (c.f. proposition \ref{IDAXD}) that the inclusion map of filtered complexes
\begin{equation}
\tau:(\Gamma(X,\mathcal A_X^{\bullet}(\nul D)),F)\hookrightarrow(\Gamma(X,\mathcal A^{\bullet}_{X,D}),F),
 \; \; \tau(\omega)=(\omega,0,\cdots, 0)
\end{equation}
is a filtered quasi-isomorphism and use the $E_1$ degenerescence of $(\Gamma(X,\mathcal A^{\bullet}_{X,D}),F)$,
where $\mathcal A^{\bullet}_{X,D}=\Cone(i_{D_{\bullet}}^*:\mathcal A_X^{\bullet}\to i_{D_{\bullet}*}\mathcal A^{\bullet}_{D_{\bullet}})[-1]$,
$D_{\bullet}$ is the simplicial variety associated to $D$ together with the canonical morphism
$i_{D_{\bullet}}:D_{\bullet}\to D\hookrightarrow X$. 

For an open variety $U=X\backslash D$, with $X\in\PSmVar(\mathbb C)$ and $D\subset X$ a normal crossing
divisor, we have (c.f.\cite{MKerr}) the classical realization map
\begin{eqnarray*}
\mathcal{R}^p(X,D):\mathcal{Z}^p(U,\bullet)^{pr/X}\to C^{\mathcal D}_{\bullet}(X,D,\mathbb Z), \; \; \;
Z\mapsto (T_Z,\Omega_Z, R_Z):=r^{\mathcal D}_{X,D}(T_{\bar Z},\Omega_{\bar Z},R_{\bar Z}),  
\end{eqnarray*}
where $\bar{Z}\in\mathcal Z^p(X,n)$ is the closure of $Z$ in $X\times(\mathbb P^1)^n$, 
which take naturally value in the relative Deligne homology complex,
\begin{equation*}
C^{\mathcal D}_{\bullet}(X,D,\mathbb Z)=\Cone(C_{2d_X-2p+\bullet}^{\diff}(X,D,\mathbb Z)\oplus\Gamma(X,F^p\mathcal{D}^{2p+\bullet}_X(\log D)) 
\to\Gamma(X,\mathcal{D}^{2p+\bullet-1}_X(\log D))).
\end{equation*}
This leads to the higher Abel Jacobi map for $U=X\backslash D$ : 
\begin{equation*}
AJ_U:\mathcal{Z}^p(U,n)_{\partial=0}^{pr/X,\hom}\to\CH^p(U,n)^{\hom}\to J^{p,2p-n-1}(U), \; \; Z\mapsto AJ(Z)=[R'_Z],
\end{equation*}
where the abelian group $\mathcal{Z}^p(U,n)_{\partial=0}^{pr/X,\hom}$ consist of the closed Bloch cycle on $U$, 
whose closure $\bar Z\in\mathcal Z^p(X\times\square^n)$ in $X\times\mathbb\square^n$ 
is still a Bloch cycle, i.e. meet all the faces of $X\in\square^n$ properly 
($\partial Z=0$ is then equivalent to $\partial\bar Z\in i_{D*}\mathcal Z^{p-1}(D,n)$),
and whose cohomology class $[\Omega_Z]=0\in H^{2p-n}(U,\mathbb C)$ vanishes,
and the complex variety
\begin{equation}
J^{p,k}(U)=H^{k}(U,\mathbb C)/(F^pH^{k}(U,\mathbb C)\oplus H^{k}(U,\mathbb Z))\simeq
(F^{d_X-p+1}H^{2d_X-k}(X,D,\mathbb C))^{\vee}/H_{2d_X-k}(X,D,\mathbb Z)
\end{equation}
is the intermediate jacobian.
We show in proposition \ref{AJfunct} that  
\begin{itemize}
\item $AJ_U$ for $U\in\SmVar(\mathbb C)$ is independent of the choise of a compactification $(X,D)$, $U=X\backslash D$,
$X\in\PSmVar(\mathbb C)$, $D\subset X$ n.c.d ; 
\item $AJ_U$ is covariantly functorial in $U\in\SmVar(\mathbb C)$ for proper morphisms
\item $AJ_U$ is contravariantly functorial in $U\in\SmVar(\mathbb C)$ for all morphisms.
\end{itemize}

In the second section (section 3), we study the relative case. 
Let $f_U:U\to S$ an open morphism which is the restriction to the complementary of a divisor $D\subset X$ 
of a smooth projective morphism $f:X\to S$, $X,S\in\SmVar(\mathbb C)$, such that
$D$ restrict on each fiber $X_s$ of $f$ to a normal crossing divisor $D_s\subset X_s$. We then introduce
the (holomorphic) Leray filtration on the complexes of $\log D$ currents and $\nul D$ differential forms giving rise to the commutative
diagramm of inclusion of bifiltered complexes of sheaves on $X^{an}$ (cf proposition \ref{FLqsi1} and proposition \ref{FLqsi}) : 
\begin{equation}
\xymatrix{
(\Omega^{\bullet}_X(\nul D),F_b,L)\ar@{^{(}->}[r]\ar@{^{(}->}[d] & (\mathcal A^{\bullet}_X(\nul D),F,L)\ar@{^{(}->}[d] \\
(\Omega^{\bullet}_X,F_b,L)\ar@{^{(}->}[r]\ar@{^{(}->}[d] & (\mathcal A^{\bullet}_X,F,L)\ar@{^{(}->}[d]  \\
(\Omega^{\bullet}_X(\log D),F_b,L)\ar@{^{(}->}[r] & (\mathcal A^{\bullet}_X(\log D),F,L)\ar@{^{(}->}[r] & (\mathcal D_X(\log D),F,L)}
\end{equation}
whose rows are bifiltered quasi-isomorphisms of sheaves.
As in the first section, we note the $E_1$ degenerescence of the filtered complex $(f_*\mathcal A_X^{\bullet}(\nul D),F)$
where $F$ is the Fr\"olicher filtation and we reinterpret the Poincare duality paring 
\begin{equation}
<\cdot,\cdot>_{ev_f}:(\mathcal H^k_S(f_U),F)\otimes (\mathcal H_S^{2d_X-k}(f_{X,D}),F)\to O_S \; \; 
\lambda\otimes\mu\mapsto (\lambda.\mu)([X])
\end{equation}
which is a morphism of variation of mixed Hodge structure, as the one induced in cohomology by the pairing
\begin{equation}
<\cdot,\cdot>_{ev_f}:(f_*\mathcal D_{X/S}^k(\log D)),F)\otimes (\Gamma(X,f_*\mathcal A_{X/S}^{2d_X-k}(\nul D)),F)\to\mathbb C \; \; 
T\otimes\eta\mapsto f_*(T\wedge\eta)
\end{equation}
Here, 
\begin{itemize}
\item $\mathcal H^k_S(f_U)=\mathcal H^kf_*\mathcal D_X^{\bullet}(\log D)\simeq Rf_{U*}\mathbb C\otimes_{\mathbb C} O_S$ and
\item $\mathcal H^k_S(f_{X,D})=\mathcal H^kf_*\mathcal A_X^{\bullet}(\nul D)\simeq Rf_{X,D*}\mathbb C\otimes_{\mathbb C} O_S$ 
\end{itemize}
are sheaves of $O_S$ modules on $S^{an}$ whose evaluations on $s\in S$ are $H^k(U_s,\mathbb C)$
and $H^k(X_s,D_s,\mathbb C)$  respectively,
and the filtration $F$ is the one induced by the Fr\"olicher filtration (see definition \ref{Hodgesub}). 
For $s\in S$, since the fiber $U_s\subset U$ is closed in $U^{an}$ and $U^{an}$ is paracompact, 
we have $(R^kf_{U*}\mathbb C)_s\xrightarrow{\sim} H^k(U_s,\mathbb C)$.  
We have the canonical quasi isomorphism
$Rf_{X,D*}\mathbb C=Rf_{U!}\mathbb C\to\Cone(Rf_*\mathbb C\to Rf_{D*}\mathbb C)[-1]$. 
On the other hand, $(R^kf_{X*}\mathbb C)_s\xrightarrow{\sim} H^k(X_s,\mathbb C)$ and $(R^kf_{D*}\mathbb C)_s\xrightarrow{\sim} H^k(D_s,\mathbb C)$
since the fibers $X_s\subset X$ and $D_s\subset D$ are closed in $X^{an}$ and $D^{an}$ respectively and $X^{an}$ and $D^{an}$ are compact (hence paracompact).  
Hence, for $s\in S$, $(R^kf_{X,D*}\mathbb C)_s\xrightarrow{\sim} H^k(X_s,D_s,\mathbb C)$

To see the $E_1$ degenerescence of $(f_*\mathcal A_{X/S}^{\bullet}(\nul D)),F)$, we 
prove (c.f. proposition \ref{IDAXDR} and corollary \ref{relvtR}) that map of filtered complexes of sheaves on $S^{an}$
\begin{equation}
<\tau>:(f_*\mathcal A_{X/S}^{\bullet}(\nul D)),F)\to(f_*\mathcal A^{\bullet}_{(X,D)/S},F),
 \;  \; <\tau>(<\omega>)=(<\omega>,0,\cdots, 0)
\end{equation}
is a filtered quasi-isomorphism and use the $E_1$ degenerescence of $(f_*\mathcal A^{\bullet}_{(X,D)/S},F)$.
The commutative diagramm of bicomplexes of sheaves on $X^{an}$ (c.f.proposition \ref{IdInProp}, see also remark \ref{IdInRem})
\begin{equation}
\xymatrix{
\phi^{r,\bullet,\bullet}:\Gr_L^r\mathcal A^{\bullet,\bullet}_X(\nul D)\ar^{\sim}[r]\ar[d] 
& \mathcal A^{\bullet-r,\bullet}_{X/S}(\nul D)\otimes_{O_X}f^*\Omega^r_S\ar[d] \\
\phi^{r,p,q}:\Gr_L^r\mathcal A^{\bullet,\bullet}_X(\log D)\ar^{\sim}[r] &
 \mathcal A^{\bullet-r,\bullet}_{X/S}(\log D)\otimes_{O_X}f^*\Omega^r_S}
\end{equation}
given by taking the inner product with a relevement of a vector field on $S^{an}$, allows us to define (c.f. subsection 3.3)
the Gauss Manin connexions relative to the local systems $H^k_{\mathbb Z}(f_U)$ and $H^k_{\mathbb Z}(f_{X,D})$ 
satisfying by definition the transversality property and featuring in the commutative diagramm :
\begin{equation}
\xymatrix{
 \, & \nabla : F^p\mathcal H_S^k(f_{X,D})\ar[ld]\ar[dd]\ar[r] & F^{p-1}\mathcal H^k_S(f_{X,D})\otimes_{O_S}\Omega_S\ar[ld]\ar[dd] \\
\nabla : F^p\mathcal H_S^k(f)\ar[r]\ar[rd] & F^{p-1}\mathcal H^k_S(f)\otimes_{O_S}\Omega_S\ar[rd] & \, \\
\, & \nabla : F^p\mathcal H_S^k(f_{U})\ar[r] & F^{p-1}\mathcal H^k_S(f_{U})\otimes_{O_S}\Omega_S}
\end{equation}
and we denote $\bar\nabla$ the induced connexion on graded pieces.
Let $Z\in\mathcal Z^p(U,n)^{pr/X,\hom S}_{\partial=0}$ a closed Bloch cycle on $U$, such that its closure 
$\bar Z\in\mathcal Z^{p}(X\times\mathbb\square^n)$ in $X\times\mathbb\square^n$
is a Bloch cycle which intersect all the fibers of $f$ properly and assume $[\Omega_{Z|U_s}]=[\Omega_Z]_{|U_s}]=0\in H^{2p-n}(U_s,\mathbb C)$.
for all $s\in S$. Then the current $R_Z$ induces by restriction on each fiber a function
\begin{equation}
\nu_Z:s\in S\mapsto [R'_{Z_s}]=ev_{X_s}[R_{Z_s}]\in J^{p,2p-n-1}(U_s)
\end{equation}
In theorem \ref{normalfunction}, we prove using 
the duality and the $E_1$ degenerescence  of $(f_*\mathcal A^{\bullet}_X(\nul D),F)$,
the following generalization a classical result:
$\nu_Z\in NS(f_U)\subset\Gamma(S,J^{p,2p-n-1}(f_U))$ is a normal function, that is is holomorphic and horizontal.
Here 
\begin{equation*}
J^{p,2p-n-1}(f_U)=\mathcal H_S^{2p-n-1}(f_U)/(F^p\mathcal H_S^{2p-n-1}(f_U)\oplus H^{2p-n-1}_{\mathbb Z}(f_U))
\end{equation*}
is the relative intermediate jacobian.
As a normal function $\nu_Z$ has an infinitesimal invariant
$\delta\nu_Z\in\Gamma(S,\mathcal H_S^{p-1,p-n}(f_U)/\Im(\bar\nabla))$ (c.f. the end of the subsection 3.3).
On the other hand the class $[\Omega_Z]_{|U_s}=0\in H^{2p-n}(U_s,\mathbb C)$ of the current $\Omega_Z$ restrict
to zero on the fiber by hypothesis leading to a class $[\Omega_Z]\in\Gamma(S,L^1R^{p-n}f_*\Omega^p_X(\log D))$,
which has an infinitesimal invariant :
\begin{equation*}
\delta[\Omega_Z]=\bar{r^{\vee}}^{-1}(\psi^2_L([\Omega_Z]/L^2))=\bar{r^{\vee}}^{-1}([\Omega_Z]/L^2)\in\Gamma(S,\mathcal H^{p-1,p-n}(f_U)\otimes_{O_S}\Omega_S/Im(\bar\nabla))
\end{equation*}
where, c.f. subsection 3.3, 3.4 and 3.5
\begin{itemize}
\item $\psi^2_L:\Gr_L^1R^{p-n}f_*\Omega^p_X(\log D)=E_{\infty}^{1,p-n}\hookrightarrow R^{p-n}f_*(\Gr_L^1\Omega^p_X(\log D))=E_1^{1,p-n}$
is the inclusion of sheaves on $S^{an}$ induced by the spectral sequence associated to the complex $(\Omega^p_X(\log D),L)$ :
for degree  reason no arrow $d_r$, $r\geq 2$ can lead to $E_r^{1,p-n}$. 
\item $\bar{r^{\vee}}:\Omega_S\otimes\mathcal H_S^{p-1,p-n}(f_U)/\Im(\bar\nabla)\to R^{p-n}f_*(\Gr_L^1\Omega^p_X(\log D))$ 
is the isomorphism induced by
\begin{itemize}
\item the morphism of sheaves on $S^{an}$
$r^{\vee}:R^{p-n}f_*\Gr_L^1\Omega^p_X(\log D)\to R^{p-n}f_*(\Omega^p_X(\log D)/L^2)$  
(induced in relative cohomoloy by the morphism of sheaves on $X^{an}$ $r^{\vee}:\Gr_L^1\Omega^p(\log D)\to\Omega^p_X(\log D)/L^2$),
\item  the isomorphism of sheaves on $S^{an}$ 
$\phi^{1,p}:R^{p-n}f_*\Gr_L^1\Omega^p_X(\log D)\xrightarrow{\sim}\Omega_S\otimes\mathcal H_S^{p-1,p-n}(f_U)$
(induced in the $f$ direct image cohomology by the isomorphism of complexes of sheaves on $X^{an}$
$\phi^{1,p,\bullet}:\Gr^1_L\mathcal A^{p,\bullet}(\log D)_X\xrightarrow{\sim}\mathcal A^{p-1,\bullet}_{X/S}(\log D)\otimes\Omega_S$)
\end{itemize}
\end{itemize}
In theorem \ref{main} (c.f. subection 3.5), we prove using 
the duality and the $E_1$ degenerescence  of $(f_*\mathcal A^{\bullet}_{X/S}(\nul D),F)$,
the following generalization a result of Voisin (\cite[theorem 19.14]{Voisin}), which is one of the main result of this paper:
\begin{thm}
Let $Z=\sum_i n_iZ_i\in\mathcal Z^p(U,n)_{\partial=0}^{pr/X\hom/S}$ such that $\pi_X(\bar Z_i)\subset X$ 
is a local complete intersection for all $i$.
Then $\delta\nu_Z=\delta[\Omega_Z]\in\Gamma(S,\Omega_S\otimes\mathcal H_S^{p-1,p-n}(f_U)/\Im(\bar\nabla))$.
\end{thm}
   
In the last section (section 4), we give two results on the relative higher Abel Jacobi map for
families of ample open hypersurface section of high degree of smooth projective variety $Y\in\PSmVar(\mathbb C)$. 
In subsection 4.1, we give (c.f.theorem \ref{cormain}) the following application of theorem \ref{main}.
Let $Y\in\PSmVar(\mathbb C)$ together with an embedding $Y\subset(\mathbb P^1)^n$. 
Consider the commutative diagram \ref{HF} of families of hypersurface sections of degre $d$ and $e$, whose squares are cartesians : 
\begin{equation}
\xymatrix{
f_D:\mathcal D=\mathcal X\cap\mathcal Z\ar@{^{(}->}[r]^{k_{\mathcal D}}\ar@{^{(}->}[d]^{i_{\mathcal D}} & \mathcal Z\ar@{^{(}->}[d]\ar[rd] \\
f:\mathcal X\ar@{^{(}->}[r]^{i_{\mathcal X}} &  Y\times S_d\times S_e\ar[r]^{p_{d,e}} & S_d\times S_e \\
f_U:\mathcal U=\mathcal X\backslash\mathcal D\ar@{^{(}->}[r]^{i_{\mathcal U}}\ar@{^{(}->}[u]^{j_{\mathcal U}} &  
(Y\times S_d\times S_e)\backslash\mathcal Z\ar@{^{(}->}[u]\ar[ru]}
\end{equation}
and denote by $p_Y:Y\times S_d\times S_e\to Y$ the other projection.
Note that $\mathcal X,\mathcal Z,\mathcal D\in\PSmVar(\mathbb C)$, since 
$p_{Y|\mathcal X}:\mathcal X\to Y$, $p_{Y|\mathcal Z}:\mathcal Z\to Y$, $p_{Y|\mathcal D}:\mathcal D\to Y$
are projective bundles and $Y$ is smooth.
For $0\in S_e$, consider the pullback of this diagram  :
\begin{equation}
\xymatrix{
f^o_D:D=X\cap (Z_0\times S_d)\ar@{^{(}->}[r]^{k_D}\ar@{^{(}->}[d]^{i_D} & Z_0\times S_d\ar@{^{(}->}[d]\ar[rd] \\
f^o:X=\mathcal X_{S_d\times 0}\ar@{^{(}->}[r]^{i_X} &  Y\times S_d\times 0\ar[r]^{p^0_d} & S_d \\
f^o_U:U=X\backslash D\ar@{^{(}->}[r]^{i_U}\ar@{^{(}->}[u]^{j} &  
(Y\backslash Z_0\times S_d)\ar@{^{(}->}[u]\ar[ru]}
\end{equation}
where $Z_0=p^0_Y(\mathcal Z_{S_d\times 0})\subset Y$, $p^0_Y=p_{Y|Y\times S_d\times 0}:Y\times S_d\times 0\to Y$ being the projection,
so that we have $\mathcal Z_{S_d\times 0}=Z_0\times S_d$.
Then $Y\backslash Z_0$ is an affine variety. 
We have $H^{d_Y}(Y\backslash Z_0,\mathbb C)^0=H^{d_Y}(Y\backslash Z_0,\mathbb C)$ 
(see subsection 4.1 for the definition of the primitive cohomology of a smooth quasi-projective variety $V$ as the kernel
of the action of $\Delta(V_H)\subset V\times V$ where $V_H\subset V$ is an ample hypersurface section). 
For a morphism $T\to S_d$, we consider the pullback of the diagram (\ref{HF0}) :
\begin{equation}
\xymatrix{
f^{T}_D:D_T=X_T\cap (Z_0\times T)\ar@{^{(}->}[r]^{k_{D_T}}\ar@{^{(}->}[d]^{i_{D_T}} & Z_0\times T\ar@{^{(}->}[d]\ar[rd] \\
f^{T}:X_T\ar@{^{(}->}[r]^{i_{X_T}} &  Y\times T\times 0\ar[r]^{p^T} & T \\
f^{T}_U:U_T=X_T\backslash D_T\ar@{^{(}->}[r]^{i_{U_T}}\ar@{^{(}->}[u]^{j_{U_T}} &  
(Y\backslash Z_0)\times T\ar@{^{(}->}[u]\ar[ru],}
\end{equation}
where $X_T=X\times_{S_d} T$, $U_T=U\times_{S_d} T$, $D_T=D\times_{S_d} T$. 
We then have a version of Nori connectness theorem for families of ample open hypersurfaces of $Y\in\PSmVar(\mathbb C)$
(c.f. theorem \ref{nori}).
\begin{thm}
Assume $d_Y\geq 4$
Let $0\in S_e$ sufficiently general and $S\subset S_d$ the open subset over which
such that the morphisms $f^0:X\to S_d$ and $f^0_D:D\to S_d$ are smooth projective.
Then, if $d,e>>0$, for all smooth morphism $T\to S_d$, 
\begin{itemize}
\item[(i)] $i_{X_T}^*:H^{d_Y-p}(Y\times T,\Omega^p_{Y\times T}(\log(Z_0\times T)))\xrightarrow{\sim} H^{d_Y-p}(X_T,\Omega_{X_T}^p(\log D_T))$
is an isomorphism,
\item[(ii)] $i_{U_T}^*:H^{d_Y}((Y\backslash Z_0)\times T,\mathbb C)\xrightarrow{\sim} H^{d_Y}(U_T,\mathbb C)$
is an isomorphism of mixed hodge structure.
\end{itemize}
\end{thm}
Then using theorem \ref{main} and this version of Nori connectness theorem we the prove the following : 
\begin{thm}
Assume $d_Y\geq 4$.
Let $0\in S_e$ sufficiently general and $S\subset S_d$ the open subset over which
such that the morphisms $f^0:X\to S_d$ and $f^0_D:D\to S_d$ are smooth projective.
Let $Z\in\mathcal Z^{p}(Y\backslash Z_0,2p-d_Y)^{pr/Y}_{\partial=0}$ such that 
$[\Omega_Z]\neq 0\in H^{d_Y}(Y\backslash Z_0,\mathbb C)$. Then for $s\in S$ general,
$AJ_{U_s}(Z_s):=[R'_{Z_s}]\neq 0\in J^{p,d_Y-1}(U_s)$.
\end{thm}
Finally, we note that this version of Nori connectness theorem implies the following (c.f. theorem \ref{ImAJUs})
which is a version of a result of Green and M\"uller-Stach \cite{GrMS} for open ample hypersurface of a smooth projective variety  :  
\begin{thm}
Assume $d_Y\geq 4$.
Let $0\in S_e$ sufficiently general and $S\subset S_d$ the open subset over which
such that the morphisms $f^0:X\to S_d$ and $f^0_D:D\to S_d$ are smooth projective.
Consider the commutative diagram
\begin{equation*}
\xymatrix{
\CH^p(Y\backslash Z_0,2p-d_Y,\mathbb Q)\ar^{i_{U_s}^*}[r]\ar^{\mathcal R^p(Y,Z_0)}[d] &
\CH^p(U_s,2p-d_Y,\mathbb Q)\ar^{\overline{\mathcal R^p(X_s,D_s)}}[d] \\
H^{d_Y}_{\mathcal D}(Y,Z_0,\mathbb Q)\ar[r] & H^{d_Y}_{\mathcal D}(X_s,D_s,\mathbb Q)/J^{p,d_Y}(Y\backslash Z_0)_{\mathbb Q}}
\end{equation*}
Then for a general point $s\in S$, $\Im(\overline{\mathcal R(X_s,D_s)})=\Im(\overline{\mathcal R(X_s,D_s)}\circ i_{U_s}^*)$. 
\end{thm}
That is if $d_Y\geq 4$, $0\in S_e$ and $s\in S\subset S_e$ are general,
the image of the primitive part of the Abel Jacobi map :
\begin{equation*}
AJ^0_{U_s}:\mathcal Z^p(U_s,2p-d_Y,\mathbb Q)^{\hom}_{\partial=0}\to J^{p,d_Y-1}(U_s)_{\mathbb Q}/i_{U_s}^*J^{p,d_Y-1}(Y\backslash Z_0)_{\mathbb Q}
\end{equation*}
is modulo torsion generated by the $AJ^0_{U_s}(Z_{|U_s})$ for $Z\in\mathcal Z^{p}(Y\backslash Z_0,2p-d_Y,\mathbb Q)_{\partial=0}$.

%%%%%%%%%%%%%%%%%%%%%%%%%%%%%%%%%%%%%%%%%%%%%%%%%%%%%%%%%%%%%%%%%%%%%%%%%%%%%%%%%%%%%%%%%
\section{Higher Abel Jacobi map for open varieties}
%%%%%%%%%%%%%%%%%%%%%%%%%%%%%%%%%%%%%%%%%%%%%%%%%%%%%%%%%%%%%%%%%%%%%%%%%%%%%%%%%%%%%%%%%

%Recall that for any $V\in\SmVar(\mathbb C)$ quasi-projective there exist $Y\in\PSmVar(\mathbb C)$ such that $Y\backslash V$
%is a normal crossing divisor with smooth components.

Let $X\in\SmVar(\mathbb C)$ and $D=\cup_{j=1}^{s}D_j\subset X$ a normal crossing divisor with smooth components $D_j$.
Let $U=X\backslash D$. 
Denote by 
\begin{itemize}
\item $j:U\hookrightarrow X$ the open inclusion,
\item $i_D:D\hookrightarrow X$ and $i_{D_j}:D_j\hookrightarrow X$, for $j\in\left\{1,\ldots,s\right\}$, the closed inclusions.
\item $\pi_X:X\times(\mathbb P^1)^n\to X$ and $\pi_{(\mathbb P^1)^n}:X\times(\mathbb P^1)^n\to(\mathbb P^1)^n$ the projections. 
%$\pi_U:U\times(\mathbb P^1)^n\to U$, $\pi_{(\mathbb P^1)^n}:U\times(\mathbb P^1)^n\to(\mathbb P^1)^n$ the projections. 
\end{itemize}

Denote by $D_{\bullet}$ the simplicial algebraic variety associated to $D$ : $D_{J}=\cap_{j\in J}D_j$ for $J\subset \left\{1,\ldots,s\right\}$
with morphisms the alternate sum of the inclusion maps $i_{D_{J'},D_J}:D_{J'}\hookrightarrow D_{J}$ for $J\subset J'$.
Let
\begin{equation*}
\xymatrix{
i_{D_{\bullet}}:D_{\bullet}\ar@{^{(}->}[r]^{a_D} & D\ar@{^{(}->}[r]^{i_D} & X},
\xymatrix{
 \; \; \; \: i_{D_J}:D_{J}\ar@{^{(}->}[r]^{a_{D_J}} & D\ar@{^{(}->}[r]^{i_D} & X}, 
\end{equation*}
be the morphism of simplicial algebraic varieties given by the inclusions $i_{D_J}:D_J\hookrightarrow X$ of the smooth varieties $D_J$ in $X$.

The adjunction morphism of complexes of sheaves on $D^{an}$ $\ad(a_D):\mathbb C_D\to Ra_{D*}a_{D}^*\mathbb C_D$ is a quasi-isomorphism
and $Ra_{D*}a_{D}^*\mathbb C_D=a_{D*}\mathcal{A}^{\bullet}_{D_{\bullet}}$.
By definition, $a_{D*}\mathcal{A}_{D_{\bullet}}^{\bullet}$ is the total complex of sheaves on $D^{an}$ associated to 
the double complex $(a_{D_J*}\mathcal{A}_{D_J}^k,d,D_r)$, 
where $D_r=\sum_{J,\card J=r}\sum_{J'\subset J,\card J=r-1}(-1)^{l}i^*_{D_J',D_J}$.  
Denote by $\omega_{|D_J}:=i_{D_J}^*\omega$.
We have the adjunction morphism of complexes of sheaves on $X^{an}$ : 
\begin{eqnarray*}
i_{D_{\bullet}}^*:\mathcal{A}_X^{\bullet}\to i_{D_{\bullet}*}\mathcal{A}_{D_{\bullet}}^{\bullet}, \; \; \;   
\omega\in\Gamma(V,\mathcal{A}_X^k)\mapsto \\
i_{D_{\bullet}}^*(\omega)=(\omega_{|D_1},\cdots,\omega_{|D_s},0,\cdots 0)
\in\Gamma(D\cap V,(\mathcal{A}_{D_{\bullet}}^{\bullet})^k):=\oplus_J \,\Gamma(D_J\cap V,\mathcal{A}_{D_J}^{k-\card J+1}).
\end{eqnarray*}

%==================================================================================
\subsection{The relative complex of differential forms for the pair $(X,D)$}
%==================================================================================

\begin{defi}
The relative complex of sheaves of holomorphic forms for the pair $(X,D)$ is
$\Omega^{\bullet}_{X,D}:=\Cone(i_{D_{\bullet}}^*:\Omega_X\to i_{D_{\bullet}*}\Omega_{D_{\bullet}})$, that is,
for $V\subset X$ an open subset,
\begin{equation*}
\Gamma(V,\Omega_{X,D}^{\bullet})=
\Cone(i_{D_{\bullet}}^*:\Gamma(V,\Omega_X^{\bullet})\to\Gamma((D\cap V)_{\bullet},\Omega_{D_{\bullet}}^{\bullet}))[-1] 
\end{equation*}
\begin{itemize}
\item $\Gamma(V,\Omega^p_{X,D})=\Gamma(V,\Omega_X^p)\oplus(\oplus_J\,\Gamma(D_J\cap V,\Omega_{D_J}^{p-\card J}))$ 
\item $\partial(\omega,\eta_J)=(\partial\omega \, , \, \omega_{|D_1\cap V}-\partial\eta_1 \, ,\, \cdots \, , \, \omega_{|D_s\cap V}-\partial\eta_s \, , \, \cdots, \, 
\eta_{2,\cdots,s|D_{1,\cdots s}\cap V}+\cdots+(-1)\eta_{1,\ldots,s-1|D_{1,\cdots s}\cap V})-\partial\eta_{1,\cdots s}$.
\end{itemize}
There is the filtration induced by the filtration b\^ete $F_b$ : 
\begin{equation*}
F_b^p\Omega^{\bullet}_{X,D}=\Omega_X^{\bullet\geq p}\oplus i_{D_{\bullet}*}\Omega_{D_{\bullet}}^{\bullet\geq p}[-1] \; \; ; \; \;
\Gr_{F_b}^p\Omega^{\bullet}_{X,D}=\Omega_X^p[-p]\oplus(\oplus_Ji_{D_J*}\Omega_{D_J}^p[-p-\card J]) \; \; ; \; \;
\end{equation*}
\end{defi}

\begin{defi}\label{defAXD}
The relative complex of sheaves of differential forms for the pair $(X,D)$ is
$\mathcal A^{\bullet}_{X,D}=\Cone(i_{D_{\bullet}}^*:\mathcal{A}_X^{\bullet}\to i_{D_{\bullet}*}\mathcal{A}_{D_{\bullet}}^{\bullet})[-1]$, that is,  
for $V\subset X$ an open subset,
\begin{equation*}
\Gamma(V,\mathcal A_{X,D}^{\bullet})=
\Cone(i_{D_{\bullet}}^*:\Gamma(V,\mathcal A_X^{\bullet})\to\Gamma((D\cap V)_{\bullet},\mathcal A_{D_{\bullet}}^{\bullet}))[-1] 
\end{equation*}
\begin{itemize}
\item $\Gamma(V,\mathcal A^k_{X,D})=\Gamma(V,\mathcal A_X^k)\oplus(\oplus_J\,\Gamma(D_J\cap V,\mathcal A_{D_J}^{k-\card J}))$ 
\item $d(\omega,\eta_J)=(d\omega \, , \, \omega_{|D_1\cap V}-d\eta_1 \, ,\, \cdots \, , \, \omega_{|D_s\cap V}-d\eta_s \, , \, \cdots, \, 
\eta_{2,\cdots,s|D_{1,\cdots s}\cap V}+\cdots+(-1)\eta_{1,\ldots,s-1|D_{1,\cdots s}\cap V})-d\eta_{1,\cdots s}$
\end{itemize}
It is a filtered complex of sheaves on $X^{an}$ by the Fr\"olicher filtration $F$ ; there is also the weight filtration $W$ with respect to the sequence 
$D_{1\ldots s},\cdots,\sqcup_{j=1}^sD_j, X$ : for $V\subset X$ an open subset,
\begin{equation*}
\Gamma(V,F^p\mathcal A^k_{X,D})=\Gamma(V,F^p\mathcal A_X^k)\oplus(\oplus_J\,\Gamma(D_J\cap V,F^p\mathcal A_{D_J}^{k-\card J})) \; \; ; \; \;
\Gamma(V,W_l\mathcal A^k_{X,D})=\oplus_{\card J\leq l}\,\Gamma(D_J\cap V,\mathcal A_{D_J}^{k-\card J}).
\end{equation*}
If $X\in\PSmVar(\mathbb C)$ is smooth projective, it is clear that 
$(\Gamma(X,\mathcal{A}^{\bullet}_{X,D}),F,W)$ is a mixed hodge complex \cite{PS} so that the spectral sequence
given by the Fr\"olicher filtration $F$ is $E^1$ degenerate. 

\end{defi}

\begin{prop}\label{AXD}
\begin{itemize}
\item [(i)] The wedge product induces an isomorphism of complexes of sheaves on $X^{an}$
\begin{equation*}
w_X:\Gr^p_{F_b}\Omega^{\bullet}_{X,D}\otimes_{O_X}(\mathcal A_X^{0,\bullet},\bar\partial)
\xrightarrow{\sim}(\mathcal A^{p,\bullet}_{X,D},\bar\partial)
\end{equation*}
\item [(ii)] The inclusion of filtered complexes of sheaves on $X^{an}$
\begin{equation*}
(\Omega^{\bullet}_{X,D},F_b)\hookrightarrow(\mathcal A_{X,D}^{\bullet},F),
\end{equation*} 
is a filtered quasi-isomorphism.
\end{itemize}
\end{prop}

\begin{proof}

(i): We check that it define a morphism of complex. The fact that it is an isomorphism is clear.
Assume for simplicity that $D_1=D$.
Let $V\subset X$ an open subset, $\omega\in(V,\Omega^p_X)$ and $\gamma\in\Gamma(V,\mathcal A^{0,q}_X)$.
Then 
\begin{eqnarray*}
d(\omega\wedge\gamma,0)&=&(d(\omega\wedge\gamma) \, , \, (\omega\wedge\gamma )_{|D}) 
=(\partial\omega\wedge\gamma+(-1)^{p}\omega\wedge d\gamma \, , \,
\omega_{|D}\wedge\gamma_{|D}) \\
&=&(\partial\omega\wedge\gamma+(-1)^{p}\omega\wedge d\gamma \, , \, \omega_{|D}\wedge\gamma_{|D})
\in\Gamma(V,F^p\mathcal A^{p+q+1}_{X,D})
\end{eqnarray*}
Thus taking the quotient by $F^{p+1}$, we obtain
\begin{equation*}
\bar\partial(\omega\wedge\gamma,0)=(\omega\wedge\bar\partial\gamma \, , \,\omega_{|D}\wedge\gamma_{|D})
\in\Gamma(V,\mathcal A^{p,q+1}_{X,D})
\end{equation*}

(ii): Ihis comes from (i). We can also see (ii) directly : we have the commutative diagram
\begin{equation*}
\xymatrix{
(\Omega^{\bullet}_X,F_b)\ar[r]^{i_{D_{\bullet}}^*}\ar@{^{(}->}[d] & (a_{D_{\bullet}}\Omega^{\bullet}_{D_{\bullet}},F_b)\ar@{^{(}->}[d] \\
(\mathcal A^{\bullet}_X,F)\ar[r]^{i_{D_{\bullet}}^*} & (a_{D_{\bullet}}\mathcal A^{\bullet}_{D_{\bullet}},F)}
\end{equation*}
whose column are filtered quasi-isomorphism, thus the morphism 
$(\Omega^{\bullet}_{X,D},F_b)\hookrightarrow(\mathcal A_{X,D}^{\bullet},F)$ is a filtered quasi-isomorphisms.

\end{proof}

%=================================================================================================================
\subsection{Complex of differential forms whose restriction on $D$ vanishes and log currents for the pair $(X,D)$}
%=================================================================================================================

\begin{defi}\cite{King}\cite{Jansen}

The bicomplex $(\mathcal A_X^{\bullet,\bullet}(\log D),\partial,\bar{\partial})$ of sheaf on $X^{an}$ for the pair $(X,D)$ is :
\begin{equation}
\mathcal A^{p,q}_X(\log D):=\Omega_X^p(\log D)\otimes_{O_X}\mathcal{A}_X^{0,q}
\xrightarrow{\sim}\Omega_X^p(\log D)\wedge\mathcal A_X^{0,q},
\end{equation}
together with the holomorphic and anti-holomorphic differential $\partial$ and $\bar{\partial}$ respectively.
The induced filtration on the total complex 
$(\mathcal A_X^{\bullet}(\log D),d)=\Tot(\mathcal A_X^{\bullet,\bullet}(\log D),\partial,\bar{\partial})$, 
with differential $d=\partial+\bar{\partial}$, is the Fr\"olicher Filtration.

\end{defi}

\begin{defi}\cite{King}

\begin{itemize}
\item Denote by 
\begin{equation*}
\Omega_X^p(\nul D):=\cap_{j=1}^s\ker(i_{D_j}^*\Omega^p_X\to i_{D_j*}:\Omega^p_{D_j})\subset\Omega^p_X,
\end{equation*}
the locally free sheaf of $O_X$ module on $X^{an}$ consisting of holomorphic $p$ forms whose restriction to $D$ vanishes.
\item The bicomplex $(\mathcal A_X^{\bullet,\bullet}(\nul D),\partial,\bar{\partial})$ of sheaf on $X^{an}$ for the pair $(X,D)$ is :
\begin{equation*}
\mathcal A^{p,q}_X(\nul D):=\Omega_X^p(\nul D)\otimes_{O_X}\mathcal{A}_X^{0,q}
\xrightarrow{\sim}\Omega_X^p(\nul D)\wedge\mathcal A_X^{0,q}\subset\mathcal A_X^{p,q},
\end{equation*}
together with the holomorphic and anti-holomorphic differential $\partial$ and $\bar{\partial}$ respectively.
The induced filtration on the total complex 
$(\mathcal A_X^{\bullet}(\nul D),d)=\Tot(\mathcal A_X^{\bullet,\bullet}(\nul D),\partial,\bar{\partial})$, 
with differential $d=\partial+\bar{\partial}$, is the Fr\"olicher Filtration.
\item The bicomplex $(\mathcal A_X^{\bullet,\bullet}(\nul D_{\infty}),\partial,\bar{\partial})$ of sheaf on $X^{an}$ for the pair $(X,D)$ is :
\begin{equation*}
\mathcal A^{p,q}_X(\nul D_{\infty}):=\cap_{j=1}^s\ker(i_{D_j}^*:\mathcal A_X^{p,q}\to i_{D_j*}\mathcal A_{D_j}^{p,q})\subset\mathcal A^{p,q}_X.
\end{equation*}
together with the holomorphic and anti-holomorphic differential $\partial$ and $\bar{\partial}$ respectively.
The induced filtration on the total complex 
$(\mathcal A_X^{\bullet}(\nul D_{\infty}),d)=\Tot(\mathcal A_X^{\bullet,\bullet}(\nul D_{\infty}),\partial,\bar{\partial})$, 
with differential $d=\partial+\bar{\partial}$, is the Fr\"olicher Filtration.
\end{itemize}
By definition we have inclusion of bicomplexes
$\mathcal A^{\bullet,\bullet}_X(\nul D)\subset\mathcal A^{\bullet,\bullet}_X(\nul D_{\infty})\subset\mathcal A^{\bullet,\bullet}_X$.
Denote by 
$t_{X,D}:\mathcal{A}^{\bullet,\bullet}_X(\nul D)
\hookrightarrow\mathcal A^{\bullet,\bullet}_X(\nul D_{\infty})\hookrightarrow\mathcal{A}^{\bullet,\bullet}_X$
the inclusion of bicomplexes of sheaves on $X^{an}$. 
\end{defi}

\begin{prop}\label{ideal}
\begin{itemize}
\item[(i)]The subcomplex of sheaves on $X^{an}$
$\Omega_X^{\bullet}(\nul D)=\mathcal I_D\Omega_X^{\bullet,\bullet}(\log D)\subset\Omega_X^{\bullet,\bullet}(\log D)$ 
and the subbicomplex of sheaves on $X^{an}$ 
$\mathcal A_X^{\bullet,\bullet}(\nul D)\subset\mathcal A_X^{\bullet,\bullet}(\log D)$ 
are a graded, respectively bigraded, ideal for the wedge product.
\item[(ii)] The sheaves of $O_X$ modules $\Omega_X^{p}(\log D)$ and $\Omega_X^{p}(\nul D)$ are locally free of rank $C_{d_X}^p$.
Moreover, the wedge product $w_X$ induces an isomorphism of sheaves of $O_X$ modules
$\Omega_X^{d_X-p}(\nul D)\xrightarrow{\sim}D^{\vee}_{O_X}(\Omega_X^p(\log D))\otimes_{O_X} K_X$.
\end{itemize}
\end{prop}

\begin{proof}
(i): This is proved in \cite{King}.

(ii): The fact that these sheaves are locally free is proved in \cite{King}.
The the wedge product induces an isomorphism of sheaves of $O_X$ modules on $X^{an}$ :
\begin{equation*}
w_X:\Omega_X^p(\log D)\otimes_{O_X}\Omega_X^{d_X-p}(\nul D)\xrightarrow{\sim} K_X.
\end{equation*}
Indeed, for $V\subset X$ an open subset such that $V\subset\mathbb C^{d_X}$ and $D\cap V=V(z_1\cdots z_r)$,
$w_X$ put together terms of the form 
\begin{itemize}
\item $(\bigwedge_{i\in I\subset\left\{1,\cdots r\right\}}\frac{dz_i}{z_i})\wedge(\bigwedge_{j\in J\subset\left\{r+1,\cdots d_X\right\}}dz_j)$
and
\item $\Pi_{i\in I}z_i(\bigwedge_{k\in \left\{1,\cdots r\right\}\backslash I} dz_k)\wedge(\bigwedge_{l\in\left\{r+1,\cdots d_X\right\}\backslash J}dz_l)$,
\end{itemize}
with $\card I+\card J=p$.
\end{proof}

\begin{prop}\label{XD}
\begin{itemize}
\item [(i)] The wedge product induces an isomorphism of  complexes of sheaves on $X^{an}$
\begin{equation*}
w_X:\Omega^p_X(\nul D)\otimes_{O_X}(\mathcal A_X^{0,\bullet},\bar\partial)
\xrightarrow{\sim}(\mathcal A_X^{p,\bullet}(\nul D),\bar\partial).
\end{equation*}
\item [(ii)] The inclusion of filtered complexes of sheaves on $X^{an}$
\begin{equation*}
(\Omega^{\bullet}_X(\nul D),F_b)\hookrightarrow(\mathcal A_X^{\bullet}(\nul D),F),
\end{equation*} 
is a filtered quasi-isomorphism.
\end{itemize}
\end{prop}

\begin{proof}
(i): Tt is clear that it is a morphism of complex since for $V\subset X$ an open subset,
$\omega\in\Gamma(V,\Omega^p(\nul D))$ and $\gamma\in\Gamma(V,\mathcal A^{0,q}_X)$ we have
$\bar\partial(\omega\wedge\gamma)=\omega\wedge\bar\partial\gamma$. It is an isomorphism by definition.

(ii): This comes from (i) : we have the Dolbeau resolutions
\begin{equation}
0\to\Omega^{p}_X(\nul D)\to\Omega^p_X(\nul D)\otimes_{O_X}(\mathcal A_X^{0,\bullet},\bar\partial)
\xrightarrow{\sim}(\mathcal A_X^{p,\bullet}(\nul D),\bar\partial).
\end{equation}

\end{proof}

We now give the definition of the complex of sheaves of currents :

\begin{defi}\cite{King}
The logaritmic complex 
$(\mathcal D_X^{\bullet}(\log D),d):=D^{\vee}(\mathcal A_X^{\bullet}(\nul D),d)$ 
of sheaf on $X^{an}$ of currents for the pair $(X,D)$ is
the Verdier dual of $\mathcal A^{\bullet}_X(\nul D)$ :
\begin{equation*}
V\subset X \; \mbox{an open subset} \; 
\mapsto\Gamma(V,\mathcal D_X^{k}(\log(D))=\Gamma_c(V,\mathcal A_X^{2d_X-k}(\nul D))^\vee
\end{equation*}
It is a filtered complex by the Fr\"olicher filtration $F$. Indeed we get a bifiltered complex of
sheaves on $X^{an}$: for $V\subset X$ an open subset
\begin{eqnarray*} 
\Gamma(V,\mathcal D_X^{p,q}(\log(D))=\left\{
T\in\Gamma(V,\mathcal D_X^{p+q}(\log(D)), \; \mbox{s.t.} \; T_{|\Gamma_c(V,\mathcal A_X^{r,s}(\nul D))^\vee}=0, \;
\mbox{for} \; (r,s)\neq(d_X-p,d_X-q)\right\} \\
\xrightarrow{\sim}\Gamma_c(V,\mathcal A_X^{d_X-p,d_X-q}(\nul D))^\vee
\end{eqnarray*}
together with the holomorphic and anti-holomorphic differential $\partial$ and $\bar{\partial}$ respectively.
The induced filtration on the total complex $\mathcal D_X^{\bullet}(\log D)=\Tot(\mathcal D_X^{\bullet,\bullet}(\log D))$ with differential
$d=\partial+\bar{\partial}$ is the Fr\"olicher Filtration.

\end{defi}

\begin{itemize}
\item We have the restriction map of filtered bicomplexes of sheaves on $X^{an}$ 
$r_{X,D}=t_{X,D}^{\vee}:\mathcal{D}_X^{\bullet,\bullet}\to\mathcal{D}_X^{\bullet,\bullet}(\log D)$ 
which is the (Verdier) dual to the inclusion $t_{X,D}$ :
for $V\subset X$ an open subset, 
\begin{equation*}
T\in\Gamma(V,\mathcal D_X^{p,q})
\mapsto r_{X,D}(T):(\eta\in\Gamma_c(V,\mathcal A_X^{d_X-p,d_X-q}(\nul D)\mapsto T(\eta)) 
\end{equation*}
\item The morphism of complexes of abelian groups 
\begin{equation*}
Int:C^{\diff}_{2d_X-\bullet}(X,\mathbb Z)\to C^{\diff,BM}_{2d_X-\bullet}(X,\mathbb Z)\hookrightarrow\Gamma(X,\mathcal{D}_X^{\bullet})
\xrightarrow{r_{X,D}}\Gamma(X,\mathcal D^{\bullet}_X(\log D)) 
\end{equation*}
given by integration factors through the quotient map 
$r_{X,D}:C^{\diff}_{2d_X-\bullet}(X,\mathbb Z)\to C^{\diff}_{2d_X-\bullet}(X,D,\mathbb Z)$ 
to the embedding of complexes of abelian groups :
\begin{equation*}
C^{\diff}_{2d_X-\bullet}(X,D,\mathbb Z)\to C^{\diff,BM}_{2d_X-\bullet}(X,D,\mathbb Z)\hookrightarrow\Gamma(X,\mathcal D^{\bullet}_X(\log D)), \; \; 
\gamma\mapsto (\eta\in\Gamma(X,\mathcal A_X^{2d_X-\bullet}(\nul D))\mapsto Int(\gamma)(\eta)=\int_{\gamma}\eta)
\end{equation*}
\item We have the wedge product which is the morphism of bicomplexes of presheaves on $X^{an}$  
\begin{eqnarray}\label{wedge}
w_X:\mathcal D_X^{\bullet,\bullet}\otimes_{O_X}\mathcal A_X^{\bullet,\bullet}\to\mathcal{D}_X^{\bullet,\bullet} \; \;
\mbox{for} \; V\subset X \; \mbox{an open subset} \\
T\otimes\omega\in\Gamma(V,\mathcal D_X^{p,q})\otimes\Gamma(V,\mathcal A_X^{r,s})
\mapsto T\wedge\omega:(\eta\in\Gamma_c(V,\mathcal A_X^{d_X-r+p,d_X-s+q})\mapsto T(\omega\wedge\eta)). 
\end{eqnarray}
It restricts to the morphism of bicomplexes of presheaves on $X^{an}$
\begin{eqnarray}\label{wedgeR}
w_X:\mathcal D_X^{\bullet,\bullet}(\log D)\otimes_{O_X}\mathcal A_X^{\bullet,\bullet}(\nul D)\to\mathcal{D}_X^{\bullet,\bullet} \; \; 
\mbox{for} \; V\subset X \; \mbox{an open subset} \\
T\otimes\omega\in\Gamma(V,\mathcal D_X^{p,q}(\log D))\otimes\Gamma(V,\mathcal A_X^{r,s}(\nul D))
\mapsto T\wedge\omega:(\eta\in\Gamma_c(V,\mathcal A_X^{d_X-r+p,d_X-s+q})\mapsto T(\omega\wedge\eta)). 
\end{eqnarray}
and also to the morphism of bicomplexes of presheaves on $X^{an}$
\begin{eqnarray*}
w_X:\mathcal D_X^{\bullet,\bullet}(\log D)\otimes_{O_X}\mathcal A_X^{\bullet,\bullet}\to\mathcal{D}_X^{\bullet,\bullet}(\log D) \; \; 
\mbox{for} \; V\subset X \; \mbox{an open subset} \\
T\otimes\omega\in\Gamma(V,\mathcal D_X^{p,q}(\log D))\otimes\Gamma(V,\mathcal A_X^{r,s}(\nul D))
\mapsto T\wedge\omega:(\eta\in\Gamma_c(V,\mathcal A_X^{d_X-r+p,d_X-s+q}(\nul D))\mapsto T(\omega\wedge\eta)). 
\end{eqnarray*}
\item We have embeddings of sheaves on $X^{an}$ 
$int:\mathcal A_X^{p,q}(\log D)\hookrightarrow\mathcal D_X^{p,q}$
given by integration :
for $V\subset X$ an open subset, 
\begin{equation*}
\omega\in\Gamma(V,\mathcal A_X^{p,q}(\log D))
\mapsto (\eta\in\Gamma_c(V,\mathcal A_X^{d_X-p,d_X-q})\mapsto int(\omega)(\eta)=\int_{V}\omega\wedge\eta) 
\end{equation*}
These integrals are convergent because $D$ is a normal crossing divisor.
Note that they do not define an embedding of bicomplexes (they do not commute with the differentials).
\end{itemize}

Denote by
$\iota^{p,q}:\mathcal A_X^{p,q}(\log D)\xrightarrow{int}\mathcal D_X^{p,q}\xrightarrow{r_{X,D}}\mathcal D_X^{p,q}(\log D)$
the composition : for $V\subset X$ an open subset
\begin{equation*}
\omega\in\Gamma(V,\mathcal A_X^{p,q}(\log D))
\mapsto (\eta\in\Gamma_c(V,\mathcal A_X^{d_X-p,d_X-q}(\nul D))\mapsto \iota(\omega)(\eta)=\int_{V}\omega\wedge\eta) 
\end{equation*}
We have then the following

\begin{thm}\cite[Theorem 1.3.11]{King}\label{module}
\item The compositions $\iota^{p,q}=r_{X,D}\circ int:\mathcal A_X^p(\log D)^{p,q}\to\mathcal D_X^{p,q}(\log D)$
define an embedding of bicomplexes of sheaves on $X^{an}$
\begin{equation*}
\iota:\mathcal A_X^{\bullet,\bullet}(\log D)\hookrightarrow\mathcal D_X^{\bullet,\bullet}(\log D). 
\end{equation*}
\item The bicomplex of sheaves on $X^{an}$ $\mathcal D_X^{\bullet,\bullet}(\log D)$ 
is a bigraded $\mathcal A_X^{\bullet,\bullet}(\log D)$ module by the map of sheaves on $X^{an}$ :
\begin{equation*}
\alpha^{p,q}_{r,s}:\mathcal A_X(\log D)^{r,s}\otimes_{O_X}\mathcal D_X^{p,q}\to\mathcal D_X^{p+r,q+r}(\log D)
\end{equation*}
If $(r,s)=(1,0)$, this map is given by, for $V\subset X$ open subset such that $V\subset\mathbb C^{d_X}$ as an open subset and $D\cap V=V(z)$, 
$T\in\Gamma(V,\mathcal D_X^{p,q})$, $\alpha(\frac{dz}{z}\otimes T)=r_{X,D}(\frac{T}{z}\wedge dz)$, 
where $\frac{T}{z}\in\Gamma(V,\mathcal D^{p,q}_X)$ is a current such that $z\frac{T}{z}=T$.
\item This bigraded module structure induces, an isomorphism of sheaves on $X^{an}$ 
\begin{equation}\label{alphaiso}
\alpha^{p,q}:=\alpha^{p,0}_{0,q}:\Omega_X^p(\log D)\otimes_{O_X}\mathcal D_X^{0,q}\xrightarrow{\sim}\mathcal D_X^{p,q}(\log D).
\end{equation}
\end{thm}

\begin{prop}\label{GrReso}
We have the folowing exact sequences of sheaves on $X^{an}$, they are the Dolbeau resolution of locally free sheaves of $O_X$ modules
$\Omega_X^p(\log D)$ and $\Omega^{d_X-p}(\nul D)$ respectively : 
\begin{eqnarray*}
&0&\to\Omega_X^p(\log D)\to\Gr_F^p\mathcal D_X^{p+\bullet}(\log D)=(\mathcal D_X^{p,\bullet}(\log D),\bar{\partial}) \\
%\xrightarrow{\bar{\partial}}\cdots\xrightarrow{\bar{\partial}}\mathcal D_X^{p,d_X}(\log D)\to 0 \\
&0&\to\Omega^{d_X-p}(\nul D)\to\Gr_F^{d_X-p}\mathcal A_X^{d_X-p+\bullet}(\nul D)=(\mathcal A_X^{d_X-p,\bullet}(\nul D),\bar{\partial}).
%\xrightarrow{\bar{\partial}}\cdots\xrightarrow{\bar{\partial}}\mathcal A_X^{d_X-p,d_X}(\nul D)\to 0
\end{eqnarray*}
\end{prop}

\begin{proof}
The second resolution is given by proposition \ref{XD} (ii).
The first one follows from the isomorphisms (\ref{alphaiso}) of theorem \ref{module} : 
$\Omega_X^p(\log D)\otimes_{O_X}\mathcal D_X^{0,q}\xrightarrow{\sim}\mathcal D_X^{p,q}(\log D)$.
\end{proof}

\begin{prop}\cite{King}\label{Jansen}
The embeddings of filtered complexes of sheaves on $X^{an}$, where $F$ 
is the Fr\"olicher filtration and $F_b$ the filtration b\^ete : 
\begin{equation*}
\xymatrix{(\Omega_X^{\bullet}(\log D),F_b)\ar@{^{(}->}[r] & (\mathcal A_X^{\bullet}(\log D), F)\ar@{^{(}->}[r]^{\iota} &
(\mathcal D_X^{\bullet}(\log D), F)}
\end{equation*}
are filtered quasi-isomorphism.
\end{prop}

\begin{proof}
It comes from the Dolbeau resolution
of the sheaf $\Omega^p(\log D)$ (proposition \ref{GrReso}). 
\end{proof}

%========================================================================================================
\subsection{Degenerescence in $E_1$ of the Fr\"olicher filtration for complex of differential forms whose restriction to $D$ vanish and duality}
%========================================================================================================

Consider the following inclusion of filtered complexes of sheaves on $X^{an}$, where $F$ is the Fr\"olicher filtration,  
\begin{eqnarray}\label{tau}
\tau:(\mathcal A ^{\bullet}_X(\nul D),F)\hookrightarrow(\mathcal A^{\bullet}_{X,D},F), \\
\omega\in\Gamma(V,\mathcal A^k_X(\nul D))\mapsto\tau(\omega)=(\omega,0,\cdots,0)\in\Gamma(V,\mathcal A_{X,D}^k), \;  
\mbox{for} \; V\subset X \; \mbox{an open subset}.
\end{eqnarray}
Then, have the following :

\begin{prop}\label{IDAXD}

(i) The restriction 
$\tau:(\Omega_X^{\bullet}(\nul D),F_b)\hookrightarrow(\Omega^{\bullet}_{X,D},F_b)$
of $\tau$ is a filtered quasi-isomorphism of complexes of sheaves.
 
(ii): Consider the embeddings of filtered complex of sheaves on $X^{an}$ :
\begin{equation}\label{IDAXDiii}
\xymatrix{(j_{!}\mathcal A^{\bullet}_U,F)\ar@{^{(}->}[r]^{t^c_U} & (\mathcal A^{\bullet}_X(\nul D),F)\ar@{^{(}->}[r]^{\tau}
 & (\mathcal A^{\bullet}_{X,D},F)}
\end{equation}
Then $\tau$ is a filtered quasi-isomorphism. The inclusion $t^c_U$ is quasi-isomorphism but NOT a filtered quasi-isomorphism.

(iii) The inclusion map $\tau:(\mathcal{A}^{\bullet}_X(\nul D),F)\hookrightarrow(\mathcal{A}^{\bullet}_{X,D},F)$, 
 is a filtered quasi-isomorphism of complexes of presheaves, that is for all open subset $V\subset X$, and all
integer $p$ the restriction
\begin{equation*}
\tau:\Gamma(V,F^p\mathcal A_X^{\bullet}(\nul D))\hookrightarrow\Gamma(V,F^p\mathcal A_{X,D}^{\bullet}), 
\end{equation*}
of $\tau$ are quasi-isomorphisms. 

\end{prop}

\begin{proof}

(i):The sequence of complexes of sheaves on $X^{an}$
\begin{equation*}
0\to \Omega_X^p(\nul D)\xrightarrow{t_{X,D}} \Omega_X^p\xrightarrow{D_1}\bigoplus_{j=1}^s i_{D_j*}\Omega_{D_j}^p
\xrightarrow{D_2}\cdots\xrightarrow{D_s}i_{D_{1,\ldots s}*}\Omega_{D_{1\ldots s}}^p\to 0,
\end{equation*}
is exact. This prove (i). 

(ii): By (i),
\begin{equation*}
\tau:(\Omega_X^{\bullet}(\nul D),F_b)\hookrightarrow(\Omega^{\bullet}_{X,D},F_b)
\end{equation*}
of is a filtered quasi-isomorphism of complexes of sheaves. 
On the other side,
\begin{itemize}
\item the inclusion $(\Omega_X^{\bullet}(\nul D),F_b)\hookrightarrow(\mathcal A_X^{\bullet}(\nul D),F)$
is a filtrered quasi-isomorphism of complexes of sheaves by proposition \ref{XD} (ii)
\item the inclusion $(\Omega_{X,D}^{\bullet},F_b)\hookrightarrow(\mathcal A_{X,D}^{\bullet},F)$
is a filtrered quasi-isomorphism of complexes of sheaves by proposition \ref{AXD} (ii).
\end{itemize}
Hence,
\begin{equation*}
\tau:(\mathcal A_X^{\bullet}(\nul D),F)\hookrightarrow(\mathcal A_{X,D}^{\bullet},F)
\end{equation*}
is a filtered quasi-isomorphism of complexes of sheaves.
The fact that two complexes of sheaves are quasi-isomorphic to $j_!\mathbb C_U$ by \cite{Jansen}.
This prove (ii).

(ii): By (ii), the inclusion maps of complexes of sheaves on $X^{an}$ 
\begin{equation*}
\tau:F^p\mathcal A_X^{\bullet}(\nul D)\hookrightarrow F^p\mathcal A_{X,D}^{\bullet} 
\end{equation*}
are quasi-isomorphism of complexes of sheaves. Thus, for all every open subset $j_V:V\hookrightarrow X$,
$j_V^*\tau:j_V^*F^p\mathcal A_X^{\bullet}(\nul D)\to j_V^*F^p\mathcal A_{X,D}^{\bullet}$ are quasi-isomorphism of complexes of sheaves.
Hence, for every open subset $V\subset X$, the maps 
\begin{equation*}
\tau:\mathbb H^{\bullet}(V,F^p\mathcal A_X^{\bullet}(\nul D))\hookrightarrow\mathbb H^{\bullet}(V,F^p\mathcal A_{X,D}^{\bullet})
\end{equation*}
are quasi-isomorphism of complexes of $\mathbb C$-vector spaces.
The sheaves $F^p\mathcal A^k_X(\nul D)$, $F^p\mathcal A^k_X$ and $i_{D_J*}F^p\mathcal A^k_{D_J}$ are sheaves of $O^{\infty}_X$ modules on $X^{an}$, 
so are c-soft (because the existence of partition of unity) 
and thus acyclic for the global section functor on each open subset $V\subset X$ ($X^{an}$ is a denombrable union of compact subsets).
Hence, for every open subset $V\subset X$, 
\begin{equation*}
H^k\Gamma(V,F^p\mathcal A_X^{\bullet}(\nul D))=\mathbb H^k(V,F^p\mathcal A_X^{\bullet}(\nul D)) \; \; \mbox{and} \; \; 
H^k\Gamma(V,F^p\mathcal A^{\bullet}_{X,D})=\mathbb H^k(V,F^p\mathcal A^{\bullet}_{X,D}).
\end{equation*} 
This proves (iii).
\end{proof}

\begin{cor}\label{ltcutau}
The following embeddings complexes of sheaves on $X^{an}$  : 
\begin{itemize}
\item $\xymatrix{j_*\mathbb C_U\ar@{^{(}->}[r] & \Omega_X^{\bullet}(\log D)\ar@{^{(}->}[r]^l & j_*\mathcal A^{\bullet}_{U}}$, and
\item $\xymatrix{j_!\mathbb C_U\ar@{^{(}->}[r] & j_!\mathcal A^{\bullet}_{U}\ar@{^{(}->}[r]^{t^c_{U}} & \mathcal A^{\bullet}_{X}(\nul D)
\ar@{^{(}->}[r]^{\tau} & \mathcal A^{\bullet}_{X,D}}$
\end{itemize}
are quasi-isomorphisms.
\end{cor}

\begin{proof}
The fact that the first sequence of inclusion are quasi-isomorphism comes from the resolution
$0\to \mathbb C_U\to\mathcal{A}^{\bullet}_U$ and the proposition \ref{Jansen}.
The fact that the second sequence of inclusion are quasi-isomorphism is given by proposition \ref{IDAXD}(ii).
\end{proof}

\begin{rem}
Note that the embedding of filtered complexes of sheaves on $X^{an}$
$l:(\Omega_X(\log D),F_b)\hookrightarrow(j_*\mathcal A^{\bullet}_{U},F)$ is NOT a filtered quasi-isomorphism.
\end{rem}

\begin{cor}\label{relvt}

Suppose $X\in\PSmVar(\mathbb C)$ is smooth projective, then
\begin{itemize}
\item[(i)] the spectral sequence associated to the filtered complex 
$(\Gamma(X,\mathcal A_X^{\bullet}(\nul D)),F)$ by Fr\"olicher filtration $F$ is $E^1$ degenerate.

\item[(ii)] for all integer $k,p$, the map induced on hypercohomology of the quotient map
\begin{eqnarray*}
H^k\Gamma(X,F^p\mathcal A_X^{\bullet}(\nul D))\to 
H^{k}\Gamma(X,\Gr_F^p\mathcal A_X^{\bullet}(\nul D))=H^{k-p}\Gamma(X,\mathcal A_X^{p,\bullet}(\nul D))=H^{k-p}(X,\Omega_X^p(\nul D)), \\
\, [\omega] \, \mapsto \, [\omega^{p,k-p}] \; \; \mbox{for} \; \; \omega\in\Gamma(X,F^p\mathcal{A}_X^k(\nul D))
\end{eqnarray*}
is surjective.
\end{itemize}

\end{cor}

\begin{proof}

(i) By proposition \ref{IDAXD} (iii), the inclusion map of complexes of $\mathbb C$ vector spaces
$\tau:(\Gamma(X,\mathcal{A}_X^{\bullet}(\nul D),F)\to(\Gamma(X,\mathcal{A}^{\bullet}_{X,D}),F)$ 
is a filtered quasi-isomorphism. 
On the other hand the spectral sequence associated to $(\Gamma(X,\mathcal{A}^{\bullet}_{X,D}),F)$
is $E_1$ degenerate (see definition \ref{defAXD}). 
Thus the spectral sequence associated to $(\Gamma(X,\mathcal{A}_X^{\bullet}(\nul D),F)$ is $E_1$ degenerate.

(ii) This is a classical fact on spectral sequence that (ii) is equivalent to (i) see for example \cite{PS}.

\end{proof}

%\begin{rem}
%We can make a relevement of the map of corollary \ref{relvt}(ii) explicit. Indeed assume for simplicity that $D=D_1\subset X$ consist of one component.
%Let $\omega\in\Gamma(X,\mathcal{A}^{p,k-p}(\nul D))^{\bar{\partial}=0}$ that is 
%$\omega\in\Gamma(X,\mathcal{A}^{p,k-p})$ such that $\omega_{|D}=0$ and $\bar{\partial}\omega=0$.
%Then because of the degenerescence of the Fr\"olicher filtration for $X\in\PSmVar(\mathbb C)$, there exist
%$\alpha\in\Gamma(X,F^p\mathcal A_X^k)^{d=0}$ such that $\alpha^{p,k-p}=\omega$. 
%Then $\alpha^{p,k-p}_{|D}=\omega_{|D}=0$ so that $\alpha_{|D}\in\Gamma(D,F^{p+1}\mathcal A_D)^{d=0}$.
%By strictness of $i^*_D:H^k(X,\mathbb C)\to H^k(D,\mathbb C)$ for the Hodge filtration,
%there exist $\beta\in\Gamma(X,F^{p+1}\mathcal A_X)^{d=0}$ such that $[\beta_{|D}]=[\alpha_{|D}]\in H^k(D,\mathbb C)$
%(take the Hodge decompositions of $H^k(X,\mathbb C)$ and $H^k(D,\mathbb C)$). 
%We have $\beta_{|D}-\alpha_{D}\in\Gamma(D,F^{p+1}\mathcal A_D)$ and $[\beta_{|D}-\alpha_{D}]=0\in H^k(D,\mathbb C)$.
%Because of the degenerescence of the Fr\"olicher filtration for $D\in\PSmVar(\mathbb C)$, and the surjectivity
%of the map $i_D^*:\Gamma(X,F^{p+1}\mathcal A^{k-1}_X)\to\Gamma(D,F^{p+1}\mathcal A^{k-1}_D)$ there exist 
%$\beta'\in\Gamma(X,F^{p+1}\mathcal A^{k-1}_X)$ such that $\alpha_{|D}=\beta_{|D}+d\beta'_{|D}$.
%Then $\gamma=\alpha-\beta-d\beta'\in\Gamma(X,F^p\mathcal A_X^k(\nul D^{\infty}))^{d=0}$ satisfy $\gamma^{p,k-p}=\alpha^{p,k-p}=\omega$. 
%\end{rem}

\begin{defi}\label{fhodge}
If $X\in\PSmVar(\mathbb C)$ is smooth projective, the hodge filtration on the $\mathbb C$ vector spaces
$H^k(U,\mathbb C)=H^k\Gamma(X,\mathcal D_X^{\bullet}(\log D))$ and 
$H^k(X,D,\mathbb C)=H^k\Gamma(X,\mathcal A_X^{\bullet}(\nul D))$ 
are given by the Fr\"olicher fitration $F$ of the
filtred complexes of sheaves on $X^{an}$ $(\mathcal D_X^{\bullet}(\log D),F)$ and $(\mathcal A_X^{\bullet}(\nul D),F)$ respectively.
The $E_1$ degenerescence of the Fr\"olicher filtration 
(corollary \ref{relvt}(i) for the complex $\Gamma(X,\mathcal A_X^{\bullet}(\nul D),F)$), say that the
following canonical surjective maps are isomorphisms :
\begin{itemize}
\item $H^k\Gamma(X,F^p\mathcal D_X^{\bullet}(\log D))\xrightarrow{\sim}F^pH^k(U,\mathbb C)$ 
\item $H^k\Gamma(X,F^p\mathcal A_X^{\bullet}(\nul D))\xrightarrow{\sim}F^pH^k(X,D,\mathbb C)$.
\end{itemize}
and their  $F$ graded pieces are
\begin{itemize}
\item $H^{p,k-p}(U,\mathbb C):=F^pH^k(U,\mathbb C)/F^{p+1}H^k(U,\mathbb C)\xrightarrow{\sim}
H^{k-p}\Gamma(X,\mathcal D_X^{p,\bullet}(\log D))=H^{k-p}(X,\Omega^p_X(\log D))$ 
\item $H^{p,k-p}(X,D,\mathbb C):=F^pH^k(X,D,\mathbb C)/F^{p+1}H^k(X,D,\mathbb C)\xrightarrow{\sim}
H^{k-p}\Gamma(X,\mathcal A_X^{p,\bullet}(\nul D))=H^{k-p}(X,\Omega_X^p(\nul D))$
(see also corollary \ref{relvt} (ii)).
\end{itemize}
\end{defi}

The wedge product $w_X$ (\ref{wedgeR}) of bicomplexes of presheaves on $X^{an}$ gives the morphism of filtered
complex of presheaves on $X^{an}$   
\begin{equation}
w_X:(\mathcal D_X^{\bullet}(\log D),F)\otimes_{O_X}(\mathcal A_X^{2d_X-\bullet}(\nul D),F)\to\mathcal D_X^{2d_X}.
\end{equation}

We have then the following :

\begin{prop}\label{dual}
If $X\in\PSmVar(\mathbb C)$, the pairing of filtred complexes of $\mathbb C$ vector spaces  :
\begin{eqnarray*}
ev_X=a_{X*}w_X=<\cdot,\cdot>_{ev_X}:
(\Gamma(X,\mathcal D_X^{\bullet}(\log D)),F)\otimes_{\mathbb C}(\Gamma(X,\mathcal A_X^{2d_X-\bullet}(\nul D)),F)
\to(\Gamma(X,\mathcal D^{2d_X}_X),F), \\
T\otimes\omega\mapsto T(\omega)=a_{X*}(T\wedge\omega)
\end{eqnarray*}
induces on cohomology isomorphisms
\begin{itemize}
\item $ev_X: H^k(U,\mathbb C)/F^pH^k(U,\mathbb C)\xrightarrow{\sim}(F^{d_X-p+1}H^{2d_X-k}(X,D,\mathbb C)))^{\vee}$ and
\item $ev_X: H^k(X,\Omega_X^p(\log D))\xrightarrow{\sim}H^{d_X-k}(X,\Omega_X^{d_X-p}(\nul D))^{\vee}$.
\end{itemize}
Note that for $\omega\in\Gamma(X,\mathcal A^k_X(\log D))^{d=0}$ a closed log form and 
$\eta\in\Gamma(X,\mathcal A_X^{2d_X-k}(\nul D))^{d=0}$, we have $<[\omega],[\eta]>_{ev_X}=\int_{X}\omega\wedge\eta$. 
\end{prop}

\begin{proof}
The fact that the pairing induced in cohomology is non degenerated is Poincare duality for the pair $(X,D)$ which is
a morphism of mixed hodge structures since the class of the wedge product of a closed log current by a closed nul form is
the cup product of the two classes (c.f.\cite{PS} for example).
\end{proof}

\begin{rem}
If $X\in\PSmVar(\mathbb C)$, the Fr\"olicher filtration of $(\Gamma(X,\mathcal A_X^{\bullet}(\log D)),F^{\bullet})$
is $E_1$ degenerate because it is a mixed hodge complex. On the other hand $\iota$ is a filtered quasi-isomorphism (proposition\ref{Jansen}).
Thus the Fr\"olicher filtration of 
$(\Gamma(X,\mathcal D_X^{\bullet}(\log D)),F^{\bullet}):=(\Gamma(X,\mathcal A_X^{2d_X-\bullet}(\nul D)),F^{\bullet})^{\vee}$
is $E_1$ also degenerate.
But the Folinger fitration on $j_*\mathcal A^{\bullet}_U$ is not $E_1$ degenerate and the hypercohomogogy of his graded piece
$\mathbb H^k(X,j_*\Gr_F^p\mathcal A^{\bullet}_X)=H^k(U,\Omega^p_U)$ vanishes for $k>0$ if $X$ is affine.
\end{rem}

%=======================================================================
\subsection{The higher Abel Jacobi map for $U$}
%=======================================================================

Recall that for any $V\in\SmVar(\mathbb C)$ quasi-projective there exist $Y\in\PSmVar(\mathbb C)$ such that $Y\backslash V$
is a normal crossing divisor with smooth components.
In this subsection, we assume that $X\in\PSmVar(\mathbb C)$ is smooth projective.

Denote by $\mathcal{Z}^p(U,\bullet)^{pr/X}\subset\mathcal{Z}^p(U,\bullet)$ the subcomplex consisting
of closed cycles on $U\times\square^{\bullet}$ such that their closure on $X\times\square^{\bullet}$ intersect all face properly.
By Bloch, the latter is quasi-isomorphic to the former.
By definition, there is an exact sequence of complexes of abelian groups
\begin{equation}\label{locexseq}
0\to\mathcal{Z}^p(D,\bullet)\xrightarrow{i_{D*}}\mathcal{Z}^p(X,\bullet)\xrightarrow{j^*}\mathcal{Z}^p(U,\bullet)^{pr/X}\to 0.
\end{equation}
For $Z\in\mathcal{Z}^p(U,\bullet)^{pr/X}$ denote by 
$\bar{Z}=\sum_i n_i\bar{Z}_i\in\mathcal{Z}^p(X\times(\mathbb P^1)^n)$ the closure of $Z=\sum_i n_iZ_i\in\mathcal{Z}^p(U\times\square^n)$. 
For $Z\in\mathcal{Z}^p(U,\bullet)_{\partial=0}^{pr/X}$, we have, by \ref{locexseq},
$\partial \bar{Z}\in i_{D*}\mathcal{Z}^{p-1}(D,\bullet)\subset\mathcal{Z}^p(X,\bullet)$. 

Let 
\begin{equation*}
C^{\mathcal D}_{\bullet}(X,\mathbb Z)=\Cone(C_{2d_X-2p+\bullet}^{\diff}(X,D,\mathbb Z)\oplus\Gamma(X,F^p\mathcal{D}^{2p+\bullet}_X) 
\to\Gamma(X,\mathcal{D}^{2p+\bullet-1}_X))
\end{equation*}
be the Deligne homology complex of $X$, 
\begin{equation*}
C^{\mathcal D}_{\bullet}(X,D,\mathbb Z)=\Cone(C_{2d_X-2p+\bullet}^{\diff}(X,D,\mathbb Z)\oplus\Gamma(X,F^p\mathcal{D}^{2p+\bullet}_X(\log D)) 
\to\Gamma(X,\mathcal{D}^{2p+\bullet-1}_X(\log D)))
\end{equation*}
be relative homology complex of $(X,D)$, and
\begin{equation*}
r^{\mathcal D}_{X,D}:C^{\mathcal D}_{\bullet}(X,\mathbb Z)\to C^{\mathcal D}_{\bullet}(X,D,\mathbb Z), \; \;
(T,\Omega, R)\mapsto (r_{X,D}(T),r_{X,D}(\Omega), r_{X,D}(R))
\end{equation*}
be the quotient map.

There is the classical realization maps
\begin{eqnarray*}
\mathcal{R}^p(X,D):\mathcal{Z}^p(U,\bullet)^{pr/X}\to C^{\mathcal D}_{\bullet}(X,D,\mathbb Z), \; \; \;
Z\mapsto (T_Z,\Omega_Z, R_Z):=r^{\mathcal D}_{X,D}(T_{\bar Z},\Omega_{\bar Z},R_{\bar Z})  
\end{eqnarray*}
where, c.f.\cite{MKerr},
\begin{itemize}
\item $T_Z=r_{X,D}(T_{\bar Z})=\sum_i n_i\pi_X((X\times T_{\square^n})\cap \bar{Z}_i)\in C^{\diff}_{2d_X-2p+n}(X,D,\mathbb Z)$,
 we have $dT_Z=T_{\partial Z}$ 
\item $\Omega_Z=r_{X,D}(\Omega_{\bar Z}):\omega\in\Gamma(X,\mathcal{A}_X^{2d_X-2p+n}(\nul D))\mapsto \\
\Omega_{\bar Z}(\omega)=\sum_i n_i\int_{\bar{Z}_i}\pi_X^*\omega\wedge\pi_{(\mathbb P^1)^n}^*\Omega_{\square^n}:=
\lim_{\epsilon\to 0}\sum_i\int_{\bar{Z}_{i\epsilon}}\pi_X^*\omega\wedge\pi_{(\mathbb P^1)^n}^*\Omega_{\square^n}$, 
it is a current of type $(p,p-n)$, i.e. $\Omega_Z\in \Gamma(X,\mathcal D_X^{p,p-n}(\log D))$. 
\item $R_Z=r_{X,D}(R_{\bar Z}):\omega\in\Gamma(X,A_X^{2d_X-2p+n+1}(\nul D))\mapsto \\
R_{\bar Z}(\omega)=\sum_i n_i\int_{\bar{Z}_i}\pi_X^*\omega\wedge\pi_{(\mathbb P^1)^n}^*R_{\square^n}:=
\lim_{\epsilon\to 0}\sum_i\int_{\bar{Z}_{i\epsilon}}\pi_X^*\omega\wedge\pi_{(\mathbb P^1)^n}^*R_{\square^n}$,
\item we have $d\Omega=\Omega_{\partial Z}$ since we have 
$d\Omega_{\square^n}=2i\pi\sum_{l=0}^n(-1)^l\Omega_{\square^n}(z_0,\cdots,\hat{z_l},\cdots,z_n)\delta(z_l)$,
and we have $dR_Z=\Omega_Z-(2i\pi)R_{\partial Z}-(2i\pi)^nT_Z$ since we have
$dR_{\square^n}=\Omega_{\square^n}-2i\pi(-1)^l\sum_{l=0}^nR_{\square^n}(z_0,\cdots,\hat{z_l},\cdots,z_n)\delta(z_l)-(2i\pi)^nT_{\square^n}$.
\end{itemize}

The currents $T_Z$ and $\Omega_Z$ are closed if $\partial Z=0$ that is if $\partial \bar{Z}\in i_{D*}\mathcal{Z}^{p-1}(D,n)$. 
For $Z\in\mathcal{Z}^p(U,n)_{\partial=0}^{pr/X}$,
the equality $dR_Z=\Omega_Z-(2i\pi)^nT_Z$ shows that $[\Omega_Z]=(2i\pi)[T_Z]\in H^{2n-p}(U,\mathbb C)$.

Denote by $\mathcal{Z}^p(U,n)_{\partial=0}^{pr/X,\hom}\subset\mathcal{Z}^p(U,n)^{pr/X}$ the subspace consisting
of $Z\in\mathcal{Z}^p(U,n)^{pr/X}$ such that $\partial Z=0$ and $[\Omega_Z]=0\in H^{2p-n}(U,\mathbb C)$, 
that is $\Omega_Z\in\Gamma(X,\mathcal D_X(\log D))$ is exact.  
Let $Z\in\mathcal{Z}^p(U,n)_{\partial=0}^{pr/X,\hom}$.
Then, for a choice of $d^{-1}\Omega_Z\in\Gamma(X,\mathcal D_X(\log D))$ and of $\partial^{-1}T_Z\in C^{\diff}_{2d_X-2p}(X,D)$,
the current
\begin{equation*}
R'_Z=R_Z-d^{-1}\Omega_Z-(2i\pi)^nd^{-1}T_Z\in\Gamma(X,\mathcal{A}_X^{2d_X-2p+n+1}(\nul D))^{\vee}
\end{equation*}
is closed, that is $R'_Z\in\Gamma(X,\mathcal{D}_X^{2d_X-2p+n+1}(\log D))^{d=0}$.

\begin{defi}
The complex analytic variety
\begin{equation*}
J^{p,k}(U)=H^{k}(U,\mathbb C)/(F^pH^{k}(U,\mathbb C)\oplus H^{k}(U,\mathbb Z))
\end{equation*}
is the intermediate jacobian.
By proposition \ref{dual}, $ev_X$ induces an isomorphism of complex varieties
$ev_X:J^{p,k}(U)\xrightarrow{\sim}(F^{d_X-p+1}H^{2d_X-k}(X,D,\mathbb C))^{\vee}/H_{2d_X-k}(X,D,\mathbb Z)$.
The map
\begin{equation*}
AJ_U:\mathcal{Z}^p(U,n)_{\partial=0}^{pr/X,\hom}\to\CH^p(U,n)^{\hom}\to J^{p,2p-n-1}(U), \; \; Z\mapsto AJ(Z)=[R'_Z]
\end{equation*}
is the higher Abel Jacobi map
\end{defi}

\begin{prop}\label{R'}
For $Z\in\mathcal{Z}^p(U,n)_{\partial=0}^{pr/X,\hom}$, there exist a topological cycle $\Gamma_{\bar Z}\in C^{\diff}_{2d_X-2p+1}(X,D,\mathbb Z)$
such that $\partial\Gamma_{\bar{Z}}^\epsilon=\bar{Z}_{\epsilon}$ for $0<\epsilon<<1$.
This gives, for $\omega\in\Gamma(X,\mathcal A^{2d_X-2p+n+1}(\nul D))^{d_X=0}$, 
\begin{eqnarray*}
R_Z(\omega):&=&\lim_{\epsilon\to 0}\sum_i n_i\int_{\bar{Z}_{i\epsilon}}\pi_X^*\omega\wedge\pi_{(\mathbb P^1)^n}^*R_{\square^n} \\
&=&\lim_{\epsilon\to 0}\int_{\Gamma_{\bar{Z}}^\epsilon}\pi_X^*\omega\wedge\pi_{(\mathbb P^1)^n}^*\Omega_{\square^n}:=
\int_{\Gamma_{\bar Z}}\pi_X^*\omega\wedge\pi_{(\mathbb P^1)^n}^*\Omega_{\square^n} 
\end{eqnarray*}
In particular, $R_Z$ restrict to a closed current on the subspace 
$\Gamma(X,F^{d_X-p+1}\mathcal A_X^{2d_X-2p+n+1}(\nul D))\subset\Gamma(X,\mathcal A_X^{2d_X-2p+n+1}(\nul D))$,
that is $R_Z\in\Gamma(X,F^{d_X-p+1}\mathcal A_X^{2d_X-2p+n+1}(\nul D))^{\vee,d=0}$ and we have
\begin{equation}
AJ_U(Z)=[R'_Z]=ev_X([R_Z]).
\end{equation}
\end{prop}

\begin{proof}

It a straightforward generalization of \cite{MKerrLMS} proposition 5.1 : for $\omega\in\Gamma(X,\mathcal A^{2d_X-2p+n+1}(\nul D))^{d=0}$,
we have 
\begin{eqnarray*}
\sum_i n_i\int_{\bar{Z}_{i\epsilon}}\pi_X^*\omega\wedge\pi_{(\mathbb P^1)^n}^*R_{\square^n} 
&=&\int_{\Gamma_{\bar{Z}}^\epsilon}d(\pi_X^*\omega\wedge\pi_{(\mathbb P^1)^n}^*R_{\square^n}) 
\; \; \mbox{by Stokes formula} \\ 
&=&\sum_i n_i\int_{\bar{Z}_{i\epsilon}}\pi_X^*\omega\wedge\pi_{(\mathbb P^1)^n}^*\Omega_{\square^n} 
\end{eqnarray*}
since  $d\omega=0$, $\omega_{|D}=0$, 
$dR_{\square^n}=\Omega_{\square^n}-(2i\pi)\sum_{l=0}^n(-1)^lR(z_0,\cdots,\hat{z_l},\cdots,z_n)\delta(z_l)-(2i\pi)^nT_{\square^n}$ and
$\partial\bar Z\in i_{D*}\mathcal Z^p(D,n)$.

\end{proof}

\begin{prop}\label{AJfunct}
\begin{itemize}
\item[(i)] The higher Abel Jacobi map of a smooth quasi-projective variety $V\in\SmVar(\mathbb C)$
is independent of a the choice of a compactification $(Y,Y\backslash V)$, $Y\in\PSmVar(\mathbb C)$, 
with $E=Y\backslash V$ a normal crossing divisor.
\item[(ii)] The higher Abel Jacobi map is functorial in $V\in\SmVar(\mathbb C)$ covariantly for proper morphisms.
\item[(iii)] The higher Abel Jacobi map is functorial in $V\in\SmVar(\mathbb C)$ contravariantly for all morphism.
\end{itemize}
\end{prop}

\begin{proof}

(i):Let $(Y,E)$ and $(Y',E')$ be two such compactification of $V$. Then exist another compactification $(Y'',E'')$ together
with two morphism of pairs $g:(Y'',E'')\to(Y,E)$, $g':(Y'',E'')\to(Y',E')$ such that $g\circ j''=j\circ I_V$ and $g'\circ j''=j'\circ I_V$. 
One can take $Y''\to\bar\Delta_V\subset Y\times Y'$ a desingularisation of the closure of the diagonal of $V$ inside $Y\times Y'$.

(ii): Let $f:U\to V$ be proper morphism. 
Then there exists a compactification $\bar f:X\to Y$ of $f$ such that $\bar f(D)\subset E$.
That is $\bar f$ induces a morphism of pair $\bar f:(X,D)\to(Y,E)$ and $\bar f\circ j=j\circ f$.  
Then, for $Z\in\mathcal Z^p(U,n)^{pr/X}$, we have 
\begin{equation*}
f_*(T_Z,\Omega_Z,R_Z)=r_{Y,E}(T_{\bar f_*\bar Z},\Omega_{\bar f_*\bar Z},R_{\bar f_*\bar Z}).
\end{equation*}

(iii): Let $h:U\to V$ be any morphism and $\bar h:X\to Y$ be any compactification of $h$. That is $\bar h\circ j=j\circ h$. Let 
\begin{equation*}
\mathcal Z^p(V,\bullet)^{pr/Y,pr/h}\subset\mathcal Z^p(V,\bullet)^{pr/Y}
\end{equation*}
be the subcomplex of abelian group consiting of cycles $Z=\sum_i n_i Z_i$
such that $\codim(\bar h^{-1}(\bar{Z}_i),X)=p$ for all $i$ and such that $h^{-1}(Z):=\sum_i n_ih^{-1}(Z_i)\in\mathcal Z^p(U,n)^{pr/X}$ 
(that is whose closure in $X$ intersect all faces of $X\times\square^n$ properly). By Bloch this inclusion is a quasi-isomorphism. 
Then, for $Z\in\mathcal Z^p(V,n)^{pr/Y,pr/h}$, 
considering $\bar h^{-1}(\bar Z):=\sum_i n_i\bar h^{-1}(\bar Z_i)\in\mathcal Z^p(X,n)$, we have
\begin{itemize}
\item $\supp(\overline{h^{-1}(Z)})\subset\supp(\bar h^{-1}(\bar Z))$ and 
$(T_{h^{-1}(Z)},\Omega_{h^{-1}(Z)},R_{h^{-1}(Z)})
=r_{X,D}(T_{\bar h^{-1}(\bar Z)},\Omega_{\bar h^{-1}(\bar Z)},R_{\bar h^{-1}(\bar Z)})$,
\item $\bar h^*(T_{\bar Z},\Omega_{\bar Z},R_{\bar Z})=(T_{\bar h^{-1}(\bar Z)},\Omega_{\bar h^{-1}(\bar Z)},R_{\bar h^{-1}(\bar Z)})$,
see \cite{King} for the definition of the pullback or Gynsin map for current, and
$\bar h^*[(T_{\bar Z},\Omega_{\bar Z},R_{\bar Z})]
=[\bar h^*(T_{\bar Z},\Omega_{\bar Z},R_{\bar Z})]=[(T_{\bar h^{-1}(\bar Z)},\Omega_{\bar h^{-1}(\bar Z)},R_{\bar h^{-1}(\bar Z)})]$.
\end{itemize}
Hence, for $Z\in\mathcal Z^p(V,n)_{\partial=0}^{pr/Y,pr/h}$, 
\begin{eqnarray*}
h^*[(T_Z,\Omega_Z,R_Z)]=r_{X,D}\bar h^*[(T_{\bar Z},\Omega_{\bar Z},R_{\bar Z})]&=& 
[r_{X,D}(T_{\bar h^{-1}(\bar Z)},\Omega_{\bar h^{-1}(\bar Z)},R_{\bar h^{-1}(\bar Z)})] \\
&=&[(T_{h^{-1}(Z)},\Omega_{h^{-1}(Z)},R_{h^{-1}(Z)})]\in H^{\mathcal D}_{2d_X-2p+n}(X,D).
\end{eqnarray*}

\end{proof}

%%%%%%%%%%%%%%%%%%%%%%%%%%%%%%%%%%%%%%%%%%%%%%%%%%%%%%%%%%%%%%%%%%%%%%%%%%%%%%%%%%%%%%%%%%%%%%%
\section{Relative Higher Abel Jacobi map for open morphism and infinitesimal invariants}
%%%%%%%%%%%%%%%%%%%%%%%%%%%%%%%%%%%%%%%%%%%%%%%%%%%%%%%%%%%%%%%%%%%%%%%%%%%%%%%%%%%%%%%%%%%%%%%

Let $X,S\in\SmVar(\mathbb C)$ and $f:X\to S$ be a smooth projective morphism.
Consider $U\subset X$ an open subset such that $D=X\backslash U$ has the property that 
$D_s\subset X_s$ is a normal crossing divisor (with smooth components) for all $s\in S$. 
Denote by $j:U\hookrightarrow X$ the inclusion and $f_U=f\circ j:U\to S$. Let $d=d_X-d_S$.

%============================================================================================
\subsection{The Leray fitration on the complexes of sheaves $\mathcal A_X(\log D)$, $\mathcal A_X(\nul D)$ and $\mathcal{D}_X(\log D)$ on $X^{an}$}
%=============================================================================================

The exact sequence of sheaves on $X^{an}$:
$0\to f^*\Omega^1_S\to\Omega^1_X\to\Omega^1_{X/S}\to 0$ 
gives the following exact sequences of sheaves on $X^{an}$ :
\begin{equation}
0\to \mathcal I_D\otimes_{O_X}f^*\Omega^1_S\to\Omega^1_X(\nul D)\to\Omega^1_{X/S}(\nul D)\to 0 
\end{equation}
%By definition, we have the inclusion of complexes of sheaves on $X^{an}$ :
%$\Omega_X^{\bullet}(\nul D)\subset\Omega_X^{\bullet}\subset\Omega_X^{\bullet}(\log D)$.

\begin{defi}

The complex of sheaves on $X^{an}$ $\Omega_X^{\bullet}(\nul D)$ is clearly a graded ideal of $\Omega_X^{\bullet}$ 
(see also \ref{ideal} for a stronger result).
The Leray filtration on the complexes of sheaves on $X^{an}$ :
$\Omega_X^{\bullet}(\nul D)\subset\Omega_X^{\bullet}\subset\Omega_X^{\bullet}(\log D)$.
is then defined by  
\begin{itemize}
\item $L^r\Omega_X^p(\nul D):=f^*\Omega^r_S\wedge\Omega_X^{p-r}(\nul D), \; \; \Omega^p_{X/S}(\nul D):=\Gr^0_L\Omega_X^p(\nul D)$ 
\item $L^r\Omega_X^p(\log D):=f^*\Omega^r_S\wedge\Omega_X^{p-r}(\log D), \; \; \Omega^p_{X/S}(\log D):=\Gr^0_L\Omega_X^p(\log D)$
\end{itemize}

The (holomorphic) Leray filtrations on the bicomplexes 
$(\mathcal A_X^{\bullet,\bullet}(\nul D),\partial,\bar{\partial})
\subset(\mathcal A_X^{\bullet,\bullet},\partial,\bar{\partial})
\subset(\mathcal A_X^{\bullet,\bullet}(\log D),\partial,\bar{\partial})$
of sheaves on $X^{an}$ are then defined by :
\begin{itemize}
\item $L^r\mathcal A^{p,q}_X(\nul D):=L^r\Omega_X^p(\nul D)\wedge\mathcal A^{0,q}_X\subset\mathcal A_X^{p,q}(\nul D)$, \; $L^r\mathcal A^{p,q}_X:=L^r\Omega_X^p\wedge\mathcal A^{0,q}_X\subset\mathcal A_X^{p,q}$,     
\item $L^r\mathcal A^{p,q}_X(\log D):=L^r\Omega_X^p(\log D)\wedge\mathcal A^{0,q}_X\subset\mathcal A_X^{p,q}(\log D)$. 
\end{itemize}
We denote
$\mathcal A^{p,q}_{X/S}(\nul D):=\Gr_L^0\mathcal A^{p,q}_X(\nul D)$, 
$\mathcal A^{p,q}_{X/S}:=\Gr_L^0\mathcal A^{p,q}_X$ and 
$\mathcal A^{p,q}_{X/S}(\log D):=\Gr_L^0\mathcal A^{p,q}_X(\log D)$. 
their first graded pieces.

This gives the (holomorphic) Leray fitration on its total complex $(\mathcal A_X^{\bullet}(\log D),d)$.

\end{defi}

\begin{rem}
Note that the holomorphic Leray fitrations 
$L^r\mathcal A^k_X(\nul D)=f^*\Omega_S^r\wedge\mathcal A_X^{k-r}(\nul D)\subset f^*\mathcal A_S^r\wedge\mathcal A_X^{k-r}(\nul D)$ 
$L^r\mathcal A^k_X=f^*\Omega_S^r\wedge\mathcal A_X^{k-r}\subset f^*\mathcal A_S^r\wedge\mathcal A_X^{k-r}$ 
$L^r\mathcal A^k_X(\log D)=f^*\Omega_S^r\wedge\mathcal A_X^{k-r}(\log D)\subset f^*\mathcal A_S^r\wedge\mathcal A_X^{k-r}(\log D)$ 
are include in the differential Leray filtration but not equal since we only pullback from $S$ forms with zero anti-holomorphic part.
\end{rem}

\begin{prop}\label{FLqsi1}
We get the following inclusions of bifiltered complexes of sheaves on $X^{an}$ :
\begin{equation*}
\xymatrix{
(\Omega_X^{\bullet}(\nul D),L)\ar@{^{(}->}[r]\ar@{^{(}->}[d] & (\Omega_X^{\bullet},L)\ar@{^{(}->}[r]\ar@{^{(}->}[d] 
& (\Omega_X^{\bullet}(\log D),L)\ar@{^{(}->}[d] \\ 
(\mathcal A_X^{\bullet}(\nul D),F,L)\ar@{^{(}->}[r] & (\mathcal A_X^{\bullet},F,L)\ar@{^{(}->}[r] & 
(\mathcal A_X^{\bullet}(\log D),F,L).}
\end{equation*}
\end{prop}

\begin{proof}
By definition, the inclusion of complexes and bicomplexes of sheaves on $X^{an}$ :
\begin{equation*}
\Omega_X^{\bullet}(\nul D)\subset\Omega_X^{\bullet}\subset\Omega_X^{\bullet}(\log D) \; \; \mbox{and} \; \;
\mathcal A_X^{\bullet,\bullet}(\nul D)\subset\mathcal A_X^{\bullet,\bullet}\subset\mathcal A_X^{\bullet,\bullet}(\log D).
\end{equation*}
are by definition compatible with the Leray filtration (even strictly compatible).
\end{proof}

\begin{prop}\label{IdInProp}
Taking interior product gives the following identifications of sheaves on $X^{an}$ : 
for $0\leq r\leq d_S$ and $0\leq r\leq 2d_S$ respectively
\begin{equation}\label{IdIn}
\xymatrix{
\phi^{r,p}:\Gr_L^r\Omega^p_X(\nul D)\ar^{\sim}[r]\ar[d] & \Omega^{p-r}_{X/S}(\nul D)\otimes_{O_X}f^*\Omega^r_S\ar[d] \\
\phi^{r,p}:\Gr_L^r\Omega^p_X(\log D)\ar^{\sim}[r] & \Omega^{p-r}_{X/S}(\log D)\otimes_{O_X}f^*\Omega^r_S
}, 
\xymatrix{
\phi^{r,p,q}:\Gr_L^r\mathcal A^{p,q}_X(\nul D)\ar^{\sim}[r]\ar[d] & \mathcal A^{p-r,q}_{X/S}(\nul D)\otimes_{O_X}f^*\Omega^r_S\ar[d] \\
\phi^{r,p,q}:\Gr_L^r\mathcal A^{p,q}_X(\log D)\ar^{\sim}[r] & \mathcal A^{p-r,q}_{X/S}(\log D)\otimes_{O_X}f^*\Omega^r_S
}
\end{equation}  
which are induced by, for $V\subset X$ an open subset, 
\begin{eqnarray*}
\omega\in\Gamma(V,L^r\Omega^p_X(\log D))
\mapsto(u\in\Gamma(V,f^*(\wedge^rT_S))\mapsto <\iota(\tilde u)\omega>\in\Gamma(V,\Omega^{p-r}_{X/S}(\log D)), \\
\omega\in\Gamma(V,L^r\mathcal A^{p,q}_X(\log D))
\mapsto(u\in\Gamma(V,f^*(\wedge^rT_S))\mapsto <\iota(\tilde u)\omega>\in\Gamma(V,\mathcal A^{p-r,q}_{X/S}(\log D)),
\end{eqnarray*}
where $\tilde u\in\Gamma(V,\wedge^r T_X)$ is a relevement of $u$, that is satisfy $df(\tilde u)=u$
and $<\cdot>:L^r\mathcal A_X(\log D)\to\Gr_L^r\mathcal A_X(\log D)$ denote the quotient class map for the Leray filtration. These maps
are independent of the choice of a relevement $\tilde u$ since $\omega$ is in $L^r$ 
(thus the interior product by a wedge product of vector fields tangent to the fibers of $f$ vanishes).
\end{prop}

\begin{proof}
The only thing that is perhaps non trivial is that 
for $\omega=\omega^p\wedge\omega^{0,q}\in\Gamma(V,L^1\mathcal A_X^{p,q}(\nul D))$ and $u\in\Gamma(V,f^*T_S)$, 
$<\iota(\tilde u)\omega>\in\Gamma(V,\mathcal A_{X/S}^{p-1,q}(\nul D))$. 
We have, since $\iota(\tilde u)\omega^{0,q}=0$ for type reason,
\begin{equation*}
\iota(\tilde u)\omega=\iota(\tilde u)\omega^p\wedge\omega^{0,q}\in\Gamma(V,\mathcal A^{p-1,q}_{X}). 
\end{equation*}
We have to prove that
$\iota(\tilde u)\omega^p\in\Gamma(V,\Omega_X^{p-1})$ vanishes on the fibers $D_s=X_s\cap D\subset X$ of $f_D:D\to S$. 
This comes from the fact that $D$ is transversal to the fibers of $f$.
Indeed, let $x\in D$ and $u^{p-1}_{D,f}\in \wedge^{p-1}T_{D_s,x}$ with $s=f(x)$.
Since $D$ is transversal to the fibers of $f$, $T_xX=Vect(T_xD,T_xX_{f(x)})$. 
Hence, there exist $\lambda_D,\lambda_f\in\mathbb C$ such that 
$\tilde u(x)=\lambda_D\tilde u_D+\lambda_f\tilde u_f$ with
$\tilde u_D\in T_xD$ and $\tilde u_f\in T_xX_{f(x)}$. This gives
\begin{equation*}
\tilde u(x)\wedge u^{p-1}_{D,f}=\lambda_D\tilde u_D\wedge u^{p-1}_{D,f}+\lambda_f\tilde{u}_f\wedge u^{p-1}_{D,f}
\end{equation*}
Now,
\begin{itemize}
\item since  $\omega^p\in\Gamma(V,L^1\Omega_X^p)$, $\omega^p_{|X_{f(x)}}=0$, hence $\omega^p(x)(\lambda_f\tilde{u}_f\wedge u^{p-1}_{D,f})=0$ 
(this says that $\iota(\tilde{u})\omega$ does not depends of the choice of the relevement $\tilde u$ of $u$). 
\item since $\omega^p_{|D}=0$, $\omega^p(x)(\lambda_D\tilde{u}_D\wedge u^{p-1}_{D,f})=0$. 
\end{itemize}
Thus, $\iota(\tilde u)\omega^p(x)(u^{p-1}_{D,f})=\omega^p(x)(\tilde u(x)\wedge u^{p-1}_{D,f})=0$.  
This shows that $(\iota(\tilde u)\omega^p)_{|D_s}=0$. 
Hence, $(\iota(\tilde u)\omega^p)_{|D}=f^*\gamma\wedge\eta_D^{p-1}\in\Gamma(V\cap D,L^1\Omega_D^p)$, 
where $\gamma\in\Gamma(f(V),\Omega_S)$ and $\eta_D\in\Gamma(V\cap D,\Omega_D)$, and thus
\begin{equation*}
(\iota(\tilde u)\omega)_{|D}=f_D^*\gamma\wedge\eta_D^{p-1}\wedge\gamma^{0,q}_{|D}\in\Gamma(V\cap D,L^1\mathcal A_D^{p,q}).
\end{equation*}
Now, shrinking $V\subset X$ if necessary, there exist $\eta\in\Gamma(V,\Omega^{p-1}_X)$ such that $\eta_{|D}=\eta_D$.
Take $\omega'=f^*\gamma\wedge\eta^{p-1}\in\Gamma(V,L^1\mathcal A^{p,q}_X)$.
Then $\iota(\tilde u)\omega-\omega'\in\Gamma(V,\mathcal A^{p,q}_X(\nul D))$ and
\begin{equation*}
<\iota(\tilde u)\omega>=<\iota(\tilde u)\omega-\omega'>\in\Gamma(V,\mathcal A_{X/S}^{p-1,q}(\nul D)).
\end{equation*}
\end{proof}

\begin{rem}\label{IdInRem}

The maps $\phi^{r,p}$ and $\phi^{r,p,q}$ define morphism of complexes $\phi^{r,\bullet}$ and $\phi^{r,\bullet,\bullet}$. 
Indeed, recall that for $\eta\in\Gamma(V,\mathcal A_X^{p,q}(\log D))$ and $v\in\Gamma(V,\wedge^rT_X)$, we have 
$d\iota(v)\eta=\iota(v)d\eta+L_v\eta$, where $L_v$ is the Lie derivative. 
Now if $\omega\in\Gamma(V,L^r\mathcal A^{p,q}_X(\log D))$, we have
\begin{equation*}
\phi(d_{X/S}<\omega>)(u)=\phi(<d\omega>)(u)=<\iota(\tilde u)d\omega>=<d\iota(\tilde u)\omega>=d_{X/S}<\iota(\tilde u)\omega>, 
\end{equation*}
since $L_{\tilde u}\omega\in\Gamma(V,L^r\mathcal A^{p,q}_X(\log D))$.

\end{rem}

The Leray filtration is compatible with proposition \ref{XD} :

\begin{prop}\label{XDL}
\begin{itemize}
\item [(i)] The wedge product induces an isomorphism of filtered complexes of sheaves on $X^{an}$
\begin{equation*}
w_X:(\Omega^p_X(\nul D),L)\otimes_{O_X}\mathcal A_X^{0,\bullet}
\xrightarrow{\sim}(\mathcal A^{p,\bullet}(\nul D),L)
\end{equation*}
\item [(ii)] The inclusion of bifiltered complexes of sheaves on $X^{an}$
\begin{equation*}
(\Omega^{\bullet}_X(\nul D),F_b,L)\hookrightarrow(\mathcal A_X^{\bullet}(\nul D),F,L),
\end{equation*} 
is a bifiltered quasi-isomorphism.
\end{itemize}
\end{prop}

\begin{proof}
(i): By proposition \ref{XD}(i), it is a morphism of complex. It is an isomorphism by definition.

(ii): This comes from (i) : we have the Dolbeau resolutions
\begin{equation*}
0\to L^r\Omega^{p}_X(\nul D)\to L^r\Omega^p_X(\nul D)\otimes_{O_X}(\mathcal A_X^{0,\bullet},\bar\partial)
\xrightarrow{\sim}(L^r\mathcal A^{p,\bullet}(\nul D),\bar\partial).
\end{equation*}

\end{proof}

We now give the definition of the Leray filtration on complexes of currents :

\begin{defi}
The Leray filtration the logaritmic complex of sheaves of currents  
$(\mathcal D_X^{\bullet,\bullet}(\log D),\partial,\bar{\partial})$ on $X^{an}$ is :
\begin{equation*}
L^r\mathcal D^{p,q}_X(\log D):=\alpha^{p,q}(L^r\Omega_X^p(\log D)\otimes_{O_X}\mathcal D^{0,q}_X)\subset\mathcal D_X^{p,q}(\log D). 
\end{equation*}
\end{defi}

By definition, the wegde product $w_X$ (\ref{wedgeR}) is compatible with the Leray filtration on gives
the morphism of filtered complexes of presheaves on $X^{an}$ :
\begin{equation*}
w_X:(\mathcal D_X^{\bullet,\bullet}(\log D),L^r)\otimes_{O_X}(\mathcal A_X^{\bullet,\bullet}(\nul D),L^s)\to 
(\mathcal D_X^{\bullet,\bullet},L^{r+s})
\end{equation*}.
In particular it induces on the first graded piece the morphism of presheaves on $X^{an}$
\begin{equation}\label{wedgeRL}
<w_X>:\mathcal D_{X/S}^{\bullet,\bullet}(\log D))\otimes_{O_X}\mathcal A_{X/S}^{\bullet,\bullet}(\nul D)\to
\mathcal D_{X/S}^{\bullet,\bullet}.
\end{equation}

\begin{rem}
For $j_V:V\hookrightarrow X$ an open subset, the pairing 
\begin{equation*}
f_{V!}j_V^*w_{X}:(\Gamma(V,\mathcal D_X^{p,q}(\log D))/L^1)\otimes_{\mathbb C}\Gamma_c(V,\mathcal A_X^{d-p,d-q}(\nul D)/L_1)\to
\mathbb C, \; \; T\otimes\omega\mapsto f_{V!}(T\wedge\omega)
\end{equation*} 
shows that 
$\Gamma(V,\mathcal D_{X/S}^{p,q}(\log(D))=\Gamma_c(V,\mathcal A^{d-p,d-q}_{X/S}(\nul D))^\vee$.
That is,
$\mathcal D^{p,q}_{X/S}(\log D):=\mathcal D^{p,q}_X(\log D)/L_1$ is
the verdier dual $\mathcal A^{d-p,d-q}_{X/S}(\nul D):\mathcal A^{p,q}_X(\nul D)/L_1$. 
\end{rem}

\begin{prop}\label{LGrReso}
For all integer $0\leq r\leq d_X$, the Dolbeau resolutions of proposition \ref{GrReso} induces resolutions
\begin{itemize}
\item (i) $0\to L^r\Omega^p_X(\log D)\to L^r\mathcal D_X^{p,\bullet}(\log D)$
\item (ii) $0\to L^r\Omega^p_X(\nul D)\to L^r\mathcal A_X^{p,\bullet}(\nul D)$
\end{itemize}
\end{prop}

\begin{proof}
The second resolution is given by proposition \ref{XDL}(ii).
The first one follows from the isomorphisms
$\alpha^{p,q}:L^r\Omega_X^p(\log D)\otimes_{O_X}\mathcal D_X^{0,\bullet}\xrightarrow{\sim} L^r\mathcal D_X^{p,\bullet}(\log D)$.
\end{proof}

\begin{prop}\label{FLqsi}
The following embeddings of bifiltered complexes of sheaves on $X^{an}$  : 
\begin{equation*}
\xymatrix{
\mathcal (\Omega_X^{\bullet}(\log D), F^{\bullet},L^{\bullet})\ar@{^{(}->}[r] &
(\mathcal A_X^{\bullet}(\log D), F^{\bullet},L^{\bullet})\ar@{^{(}->}[r]^{\iota} &
(\mathcal D_X^{\bullet}(\log D), F^{\bullet},L^{\bullet})},
\end{equation*}
is a bifiltred quasi isomorphism of complexes of sheaves.
In particular,
\begin{equation*}
\mathcal (\Omega_{X/S}^{\bullet}(\log D), F_b^{\bullet})\to
(\mathcal A_{X/S}^{\bullet}(\log D), F^{\bullet})\xrightarrow{<\iota>}
(\mathcal D_{X/S}^{\bullet}(\log D), F^{\bullet}),
\end{equation*}
where $<\iota>$ is the morphism induced by $\iota$ on $\Gr_L^0$, are filtered quasi-isomorphism.
\end{prop}

\begin{proof}
This comes from proposition \ref{LGrReso}(i).
\end{proof}

%In particular, the wegde product $w_X$ (\ref{wedgeR}) is compatible with the Leray filtration on gives
%the morphism of filtred complexes of presheaves on $X^{an}$ :
%\begin{equation}
%w_X:(\mathcal D_X^{\bullet,\bullet}(\log D),L^r)\otimes_{O_X}(\mathcal A_X^{\bullet,\bullet}(\nul D),L^s)\to 
%(\mathcal D_X^{\bullet,\bullet},L^{r+s})
%\end{equation}
%In particular it induces on the first graded piece the morphism of presheaves on $X^{an}$
%\begin{equation}\label{wedgeRL}
%[w_X]:\mathcal D_{X/S}^{\bullet,\bullet}(\log D))\otimes_{O_X}\mathcal A_{X/S}^{\bullet,\bullet}(\nul D)\to
%\mathcal D_{X/S}^{\bullet,\bullet}
%\end{equation}

\begin{defi}\label{defAXDR}
Leray filtration on the complex of sheaves on $X^{an}$ $\mathcal A^{\bullet}_{X,D}$ is given by
$L^r\mathcal A^{\bullet}_{X,D}:=\Cone(i_D^*:L^r\mathcal A^{\bullet}_X\to L^r\mathcal A^{\bullet}_{D_{\bullet}})[-1]\subset\mathcal A^{\bullet}_{X,D}$, 
that is for $V\subset X$ an open subset
\begin{equation*}
\Gamma(V,L^r\mathcal A^{\bullet}_{X,D})=\Cone(i_{D_{\bullet}}^*:\Gamma(V,L^r\mathcal A^{\bullet}_X))
\to\Gamma((V\cap D)_{\bullet},L^r\mathcal A^{\bullet}_{D_{\bullet}})\subset\Gamma(V,\mathcal A^{\bullet}_{X,D})
\end{equation*}
is the subcomplex whose terms are
$\Gamma(V,L^r\mathcal A^k_{X,D})=\Gamma(V,L^r\mathcal A^k_X)\oplus(\oplus_J\Gamma(V\cap D_J,L^r\mathcal A_{D_J}^{k-\card J}))
\subset\Gamma(V,\mathcal A^k_{X,D})$.
We will consider the complex of sheaf on $X^{an}$
$\mathcal A^{\bullet}_{X,D/S}:=\Gr^0_L\mathcal A^{\bullet}_{X,D} $
Since the morphisms $f:X\to S$ and $f_D:D\to S$ are smooth projective, the spectral sequence associated to the
Fr\"olicher filtration $(f_*\mathcal A^{\bullet}_{(X,D)/S},F)$ on this complex is $E_1$ degenerate.
of sheaves on $S$ is $E_1$ degenerate.
\end{defi}

\begin{prop}\label{AXDL}
\begin{itemize}
\item [(i)] The wedge product induces an isomorphism of filtered complexes of sheaves on $X^{an}$
\begin{equation*}
w_X:(\Omega^p_{X,D},L)\otimes_{O_X}\mathcal A_X^{0,\bullet}
\xrightarrow{\sim}(\mathcal A^{p,\bullet}_{X,D},L)
\end{equation*}
\item [(ii)] The inclusion of bifiltered complexes of sheaves on $X^{an}$
\begin{equation*}
(\Omega^{\bullet}_{X,D},F_b,L)\hookrightarrow(\mathcal A_{X,D}^{\bullet},F,L),
\end{equation*} 
is a bifiltered quasi-isomorphism.
\end{itemize}
\end{prop}

\begin{proof}

(i): This is a morphism of complexes by proposition \ref{AXD}(i).

(ii): It follow from (i).We can also see (ii) directly : we have the commutative diagram
\begin{equation}
\xymatrix{
(\Omega^{\bullet}_X,F_b,L)\ar[r]^{i_{D_{\bullet}}^*}\ar@{^{(}->}[d] & (a_{D_{\bullet}}\Omega^{\bullet}_{D_{\bullet}},F_b,L)\ar@{^{(}->}[d] \\
(\mathcal A^{\bullet}_X,F,L)\ar[r]^{i_{D_{\bullet}}^*} & (a_{D_{\bullet}}\mathcal A^{\bullet}_{D_{\bullet}},F,L)
}
\end{equation}
where the columns are bifiltered quasi-isomorphisms.
\end{proof}

%===========================================================================
\subsection{$E_1$ degeneresence and duality in the relative case}
%============================================================================

The inclusion (\ref{tau}) of filtered complexes of sheaves on $X^{an}$ 
$\tau:(\mathcal{A}^{\bullet}_X(\nul D),F)\hookrightarrow(\mathcal{A}^{\bullet}_{X,D},F)$
is by definition compatible with the Leray filtration. Hence $\tau$ is an inclusion of bi filtered complexes of sheaves on $X^{an}$ 
\begin{equation*}
\tau:(\mathcal A^{\bullet}_X(\nul D),F,L)\hookrightarrow(\mathcal{A}^{\bullet}_{X,D},F,L).
\end{equation*}
Denote by $<\tau>:(\mathcal A^{\bullet}_{X/S}(\nul D),F)\hookrightarrow(\mathcal{A}^{\bullet}_{(X,D)/S},F)$
the map induced on $\Gr_L^0$.
Similary 
$t_U^c:(j_!\mathcal A^{\bullet}_U,L)\hookrightarrow(\mathcal A^{\bullet}_X(\nul D),L)$ (c.f proposition \ref{IDAXD}) and  
$l:(\Omega^{\bullet}_X(\log D),L)\hookrightarrow(j_*\mathcal A^{\bullet}_U,L)$ (c.f corollary \ref{ltcutau})  
are inclusions of filtred complexes of sheaves on $X^{an}$.
Then,

\begin{prop}\label{IDAXDR}

(i) The restriction $\tau:(\Omega^{\bullet}_X(\nul D),L)\hookrightarrow(\Omega_{X,D},L)$ of $\tau$
is a filtered quasi-isomorphism of sheaves.

(ii): Consider embeddings of bifiltered complex of sheaves on $X^{an}$ given by \ref{IDAXDiii} :
\begin{equation*}
\xymatrix{
(j_{!}\mathcal A^{\bullet}_U,F,L)\ar@{^{(}->}[r]^{t^c_U} & (\mathcal A^{\bullet}_X(\nul D),F)\ar@{^{(}->}[r]^{\tau} & (\mathcal A^{\bullet}_{X,D},F,L)}
\end{equation*}
Then $\tau$ is a bifiltered quasi-isomorphism of sheaves. 
It induces the maps of filtered complex of sheaves on $X^{an}$ :
\begin{equation*}
(j_{!}\mathcal A^{\bullet}_{U/S},F)\xrightarrow{<t^c_U>}(\mathcal A^{\bullet}_{X/S}(\nul D),F)\xrightarrow{<\tau>}(\mathcal A^{\bullet}_{(X,D)/S},F)
\end{equation*}
where $<t^c_U>$ are the morphism induced by $t^c_U$ on $\Gr_L^0$ 
and $\mathcal A^{\bullet}_{U/S}=j^*\mathcal A^{\bullet}_{X/S}=\Gr^0_L\mathcal A^{\bullet}_U$.
In particular, $<\tau>$ is a filtered quasi-isomorphism. 
The inclusion $<t^c_U>$ is quasi-isomorphism but NOT a filtered quasi-isomorphism.

(iii) The inclusion map $\tau:(\mathcal{A}^{\bullet}_X(\nul D),F,L)\hookrightarrow(\mathcal{A}^{\bullet}_{X,D},F,L)$, 
 is a bi-filtered quasi-isomorphism of complexes of presheaves, that is for all open subset $V\subset X$,
and for all integers $p,r$ the restriction
\begin{equation*}
\tau:\Gamma(V,L^rF^p\mathcal{A}_X^{\bullet}(\nul D))\hookrightarrow\Gamma(V,L^rF^p\mathcal{A}_{X,D}^{\bullet}) 
\end{equation*}
of $\tau$ are quasi-isomorphisms. 

\end{prop}

\begin{proof}

(i): The sequence of complexes of sheaves on $X^{an}$
\begin{equation}\label{ARL}
0\to L^r\Omega_X^p(\nul D)\xrightarrow{t_{X,D}} L^r\Omega_X^p\xrightarrow{D_1}\bigoplus_{j=1}^s i_{D_j*}L^r\Omega_{D_j}^p
\xrightarrow{D_2}\cdots\xrightarrow{D_s}i_{D_{1,\ldots s}*}L^r\Omega_{D_{1\ldots s}}^p\to 0,
\end{equation}
is exact. This prove (i). 

(ii): By (i),
\begin{equation*}
\tau:(\Omega_X^{\bullet}(\nul D),F_b,L)\hookrightarrow(\Omega^{\bullet}_{X,D},F_b,L)
\end{equation*}
of is a bifiltered quasi-isomorphism of complexes of sheaves. 
On the other side,
\begin{itemize}
\item the inclusion $(\Omega_X^{\bullet}(\nul D),F_b,L)\hookrightarrow(\mathcal A_X^{\bullet}(\nul D),F,L)$
is a bifiltrered quasi-isomorphism of complexes of sheaves by proposition \ref{XDL} (ii)
\item the inclusion $(\Omega_{X,D}^{\bullet},F_b,L)\hookrightarrow(\mathcal A_{X,D}^{\bullet},F,L)$
is a bifiltrered quasi-isomorphism of complexes of sheaves by proposition \ref{AXDL} (ii).
\end{itemize}
Hence,
\begin{equation*}
\tau:(\mathcal A_X^{\bullet}(\nul D),F,L)\hookrightarrow(\mathcal A_{X,D}^{\bullet},F,L)
\end{equation*}
is a bifiltered quasi-isomorphism of complexes of sheaves.
This prove (ii).

(ii): By (ii), the inclusion maps of complexes of sheaves on $X^{an}$ 
\begin{equation*}
\tau:L^rF^p\mathcal A_X^{\bullet}(\nul D)\hookrightarrow L^rF^p\mathcal A_{X,D}^{\bullet} 
\end{equation*}
are quasi-isomorphism of complexes of sheaves. Thus, for all every open subset $j_V:V\hookrightarrow X$,
$j_V^*\tau:j_V^*L^rF^p\mathcal A_X^{\bullet}(\nul D)\to j_V^*L^rF^p\mathcal A_{X,D}^{\bullet}$ are quasi-isomorphism of complexes of sheaves.
Hence, for every open subset $V\subset X$, the maps 
\begin{equation*}
\tau:\mathbb H^{\bullet}(V,L^rF^p\mathcal A_X^{\bullet}(\nul D))\hookrightarrow\mathbb H^{\bullet}(V,L^rF^p\mathcal A_{X,D}^{\bullet})
\end{equation*}
are quasi-isomorphism of complexes of $\mathbb C$-vector spaces.
The sheaves $L^rF^p\mathcal A^k_X(\nul D)$, $L^rF^p\mathcal A^k_X$ and $i_{D_J*}L^rF^p\mathcal A^k_{D_J}$ are sheaves of $O^{\infty}_X$ modules on $X^{an}$, 
so are c-soft (because the existence of partition of unity) 
and thus acyclic for the global section functor on each open subset $V\subset X$ ($X^{an}$ is a denombrable union of compact subsets).
Hence, for every open subset $V\subset X$, 
\begin{equation*}
H^k\Gamma(V,L^rF^p\mathcal A_X^{\bullet}(\nul D))=\mathbb H^k(V,L^rF^p\mathcal A_X^{\bullet}(\nul D)) \; \; \mbox{and} \: \; 
H^k\Gamma(V,L^rF^p\mathcal A^{\bullet}_{X,D})=\mathbb H^k(V,L^rF^p\mathcal A^{\bullet}_{X,D}).
\end{equation*} 
This proves (iii).

\end{proof}

\begin{cor}\label{Rcohm}
The following maps of complexes of sheaves on $X^{an}$  : 
\begin{itemize}
\item $j_*f_U^*O_S\to\Omega_{X/S}^{\bullet}(\log D)\xrightarrow{<l>}j_*\mathcal A^{\bullet}_{U/S}$, and
\item $j_!f_U^*O_S\to j_!\mathcal A^{\bullet}_{U/S}\xrightarrow{<t^c_{U}>}\mathcal A^{\bullet}_{X/S}(\nul D)
\xrightarrow{<\tau>}\mathcal A^{\bullet}_{(X,D)/S}$
\end{itemize}
are quasi-isomorphisms.
\end{cor}

\begin{proof}
The fact that the maps of the first sequence are quasi-isomorphism comes from the resolution
$0\to f_U^*O_S\to\mathcal{A}^{\bullet}_{U/S}$.
The fact that the maps of the second sequence are quasi-isomorphism is given by proposition \ref{IDAXDR}(ii)
\end{proof}

\begin{cor}\label{relvtR}

(i) The spectral sequence associated to the filtred complex of sheaves on $S^{an}$
$(f_*\mathcal A_{X/S}^{\bullet}(\nul D)),F)$ by Fr\"olicher filtration $F$ is $E^1$ degenerate.

(ii)For all integer $k,p$, the map induced on relative hypercohomology of the quotient map 
$F^p\mathcal A _{X/S}^{\bullet}(\nul D)\to\Gr_F^p\mathcal A _{X/S}^{\bullet}(\nul D)$
\begin{equation*}
\mathcal H^kf_*F^p\mathcal A _{X/S}^{\bullet}(\nul D)\to 
\mathcal H^kf_*\Gr_F^p\mathcal A^{\bullet}_{X/S}(\nul D))=\mathcal H^{k-p}f_*\mathcal A_{X/S}^{p,\bullet}(\nul D))
= R^{k-p}f_*\Omega_{X/S}^p(\nul D)),
\end{equation*} 
given by for $W\subset S$ an open subset and $\omega\in\Gamma(X_W,F^p\mathcal{A}_{X/S}^k(\nul D))^{d_{X/S}=0}$,
\begin{equation*}
[\omega]\in\Gamma(W,\mathcal H^kF^p\mathcal A _{X/S}^{\bullet}(\nul D)) \, \mapsto \, 
[\omega^{p,k-p}]\in\Gamma(W,\mathcal H^{k-p}\mathcal A_{X/S}^{p,\bullet}(\nul D)) 
\end{equation*}
is surjective.

\end{cor}

\begin{proof}

(i) By proposition \ref{IDAXDR} (iii), the map of complexes of sheaves on $X^{an}$
$<\tau>:(\mathcal{A}_{X/S}^{\bullet}(\nul D),F)\to(\mathcal{A}_{(X,D)/S}^{\bullet},F)$ 
is a filtered quasi-isomorphism of complexes of presheaves. 
Hence, the map of complexes of sheaves on $S^{an}$
$f_*<\tau>:(f_*\mathcal{A}_{X/S}^{\bullet}(\nul D),F)\to(f_*\mathcal{A}_{(X,D)/S}^{\bullet},F)$ 
is a filtered quasi-isomorphism of complexes presheaves, hence a filtered quasi-isomorphism of complexes of sheaves.  
On the other hand the spectral sequence associated to the complex of sheaves $(f_*\mathcal{A}^{\bullet}_{(X,D)/S},F)$
is $E_1$ degenerate (see definition \ref{defAXDR}).
 Thus the spectral sequence associated to $(f_*\mathcal{A}_X^{\bullet}(\nul D),F)$ is $E_1$ degenerate.

(ii) This is a classical fact on spectral sequence that (ii) is equivalent to (i) see for example \cite{PS}.

\end{proof}

Denote by 
$H^k_{\mathbb Z}(f_U):=R^kf_{U*}\mathbb Z_U$, $H^k_{\mathbb C}(f_U):=R^kf_{U*}\mathbb C_U$, and by 
$H^k_{\mathbb Z}(f_{X,D}):=R^kf_{U!}\mathbb Z_U$, $H^k_{\mathbb C}(f_{X,D}):=R^kf_{U!}\mathbb C_U$. 
For $s\in S$, since the fiber $U_s\subset U$ is closed in $U^{an}$ and $U^{an}$ is paracompact, 
we have $(R^kf_{U*}\mathbb C)_s\xrightarrow{\sim} H^k(U_s,\mathbb C)$.  
We have the canonical quasi isomorphism
$Rf_{X,D*}\mathbb C=Rf_{U!}\mathbb C\to\Cone(Rf_*\mathbb C\to Rf_{D*}\mathbb C)[-1]$. 
On the other hand, $(R^kf_{X*}\mathbb C)_s\xrightarrow{\sim} H^k(X_s,\mathbb C)$ and $(R^kf_{D*}\mathbb C)_s\xrightarrow{\sim} H^k(D_s,\mathbb C)$
since the fibers $X_s\subset X$ and $D_s\subset D$ are closed in $X^{an}$ and $D^{an}$ respectively and $X^{an}$ and $D^{an}$ are compact (hence paracompact).  
Hence, for $s\in S$, $(R^kf_{X,D*}\mathbb C)_s\xrightarrow{\sim} H^k(X_s,D_s,\mathbb C)$.

In our situation, the $H^k_{\mathbb Z}(f_U)$ and the $H^k_{\mathbb Z}(f_{X,D})$ are local systems on $S^{an}$ 
because the maps $f:X\to S$ and $f_{D_J}:D_J\to S$ are smooth projective.
For $0\leq k\leq 2d$ (otherwise the sheaves are zero), the sheaves of $O_S$ modules 
$\mathcal H^k_S(f_U):=H^k_{\mathbb C}(f_U)\otimes_{\mathbb C_S}O_S=R^kf_{U*}\mathbb C_U\otimes_{\mathbb C_S}O_S$
are locally free and we will denote again $\mathcal H^k_S(f_U)$ the corresponding holomorphic vector bundles on $S$
For $0\leq k\leq 2d$ (otherwise the sheaves are zero), the sheaves of $O_S$ modules 
$\mathcal H^k_S(f_{X,D}):=H^k_{\mathbb C}(f_{X,D})\otimes_{\mathbb C_S}O_S=R^kf_{U!}\mathbb C_U\otimes_{\mathbb C_S}O_S$
are locally free and we will denote again $\mathcal H^k_{S}(f_{X,D})$ the corresponding holomorphic vector bundles on $S$.

\begin{prop}
We have the following isomorphisms of sheaves on $S^{an}$ :
\begin{itemize}
\item $\mathcal H^k_S(f_U)\xrightarrow{\sim}R^kf_{U*}(f^*O_S)=\mathcal H^kf_{U*}\mathcal A^{\bullet}_{U/S}
=\mathcal H^kf_*\mathcal A^{\bullet}_{X/S}(\log D)=\mathcal H^kf_*\mathcal D^{\bullet}_{X/S}(\log D)$
\item $R^kf_{U!}(f^*O_S)=\mathcal H^kf_{U!}\mathcal A^{\bullet}_{U/S}=\mathcal H^kf_*\mathcal A^{\bullet}_{X/S}(\nul D)
\xrightarrow{\sim}\mathcal H^k_{S}(f_{X,D})$.
\end{itemize}
\end{prop}

\begin{proof}
These two isomorphism are given by the two projection formula. The equalities comes from corollary \ref{Rcohm}.
\end{proof}

\begin{rem}
In our situation, since $H^k(f_U)$ and $H^k(f_{X,D})$ are local systems, these isomorphisms can be explicited in common local trivialisations
of the differentially locally trivial maps $f:X^{an}\to S^{an}$, $f_{D_J}:D^{an}_J\to S^{an}$.
\end{rem}

\begin{defi}\label{Hodgesub}
The Hodge filtrations on the vector bundles $\mathcal H^k_S(f_U)$ and $\mathcal H^k_S(f_{X,D})$
is the one given by the Fr\"olicher filtration $F$ on the complexes of sheaves on $S^{an}$ 
$f_*\mathcal A^{\bullet}_{X/S}(\log D)$ and $f_*\mathcal A^{\bullet}_{X/S}(\nul D)$ respectively. 
By the $E_1$ degenerescence of the spectral sequences associated to $(f_*\mathcal A^{\bullet}_{X/S}(\log D),F)$ and
$(f_*\mathcal A^{\bullet}_{X/S}(\nul D),F)$(corollary \ref{relvtR}(i)),
the following canonical surjective maps of sheaves on $S^{an}$ are isomorphisms
\begin{itemize}
\item $\mathcal H^kf_*F^p\mathcal A^{\bullet}_{X/S}(\log D)=\mathcal H^kf_*F^p\mathcal D^{\bullet}_{X/S}(\log D)
\xrightarrow{\sim}F^p\mathcal H^k_S(f_U)$
\item $\mathcal H^kf_*F^p\mathcal A^{\bullet}_{X/S}(\nul D)\xrightarrow{\sim} F^p\mathcal H^k_{S}(f_{X,D})$
\end{itemize}
and their graded pieces are
\begin{itemize}
\item $\mathcal H^{p,k-p}_S(f_U):=F^p\mathcal H^k_S(f_U)/F^{p+1}\mathcal H^k_S(f_U)
%\xrightarrow{\sim}\mathcal H^kf_*F^p\mathcal A^{\bullet}_{X/S}(\log D)/\mathcal H^kf_*F^{p+1}\mathcal A^{\bullet}_{X/S}(\log D)
\xrightarrow{\sim}\mathcal H^{p-k}f_*\mathcal A^{p,\bullet}_{X/S}(\log D)
=\mathcal H^{p-k}f_*\mathcal D^{p,\bullet}_{X/S}(\log D)=R^{k-p}f_*\Omega^p_{X/S}(\log D)$
\item $\mathcal H^{p,k-p}_S(f_{X,D}):=F^p\mathcal H^k_S(f_{X,D})/F^{p+1}\mathcal H^k_S(f_{X,D})
%\xrightarrow{\sim}\mathcal H^kf_*F^p\mathcal A^{\bullet}_{X/S}(\nul D)/\mathcal H^kf_*F^{p+1}\mathcal A^{\bullet}_{X/S}(\nul D)
\xrightarrow{\sim}\mathcal H^{p-k}f_*\mathcal A^{p,\bullet}_{X/S}(\nul D)=R^{k-p}f_*\Omega^p_{X/S}(\nul D)$
(see also corollary \ref{relvtR}(ii)).

\end{itemize}

\end{defi}

The wedge product (\ref{wedge}) is a bifiltered morphism of complexes of presheaves on $X^{an}$ :
\begin{equation*}
w_X:(\mathcal D_X^{\bullet}(\log D),F,L)\otimes_{O_X}(\mathcal A_X^{\bullet}(\nul D),F,L)\to\mathcal D_X^{\bullet}
\end{equation*}
and induces the  pairings of filtered complexes of presheaves on $S^{an}$:
\begin{itemize}
\item 
$ev_f=f_*<w_X>=<\cdot,\cdot>_{ev_f}:(f_*\mathcal D_{X/S}^{\bullet}(\log D),F)\otimes_{O_S}(f_*\mathcal A_{X/S}^{2d-\bullet}(\nul D),F)
\to (f_*\mathcal D_{X/S}^{2d},F)$,
given by, for $W\subset S$, 
$T\otimes\omega\in\Gamma(X_W,\mathcal D_{X/S}^{\bullet}(\log D))\otimes_{\mathbb C}\Gamma(X_W,\mathcal A_{X/S}^{2d-\bullet}(\nul D))
\mapsto <T,\omega>_{ev_f}=f_{X_W*}(T\wedge\omega)$
\item 
$f_*ev_X=f_*w_X=<\cdot,\cdot>_{f_*w_X}:(f_*\mathcal D_{X}^{\bullet}(\log D)/L^2,F)\otimes_{O_S}(f_*L^{d_S-1}\mathcal A_{X}^{2d_X-\bullet}(\nul D),F)
\to (f_*\mathcal D_X^{2d_X},F)$,  
given by, for $W\subset S$, 
$T\otimes\omega\in\Gamma(X_W,\mathcal D_{X}^{\bullet}(\log D))\otimes_{\mathbb C}\Gamma(X_W,\mathcal A_{X}^{2d_X-\bullet}(\nul D))
\mapsto <T,\omega>_{f_*ev_X}=f_{X_W*}(T\wedge\omega)$.
\end{itemize}

\begin{prop}\label{dualR}
\begin{itemize}
\item [(i)] The pairing of filtered complexes of presheaves on $S^{an}$:

$ev_f=<\cdot,\cdot>_{ev_f}:(f_*\mathcal D_{X/S}^{\bullet}(\log D),F)\otimes_{O_S}(f_*\mathcal A_{X/S}^{2d-\bullet}(\nul D),F)
\to (f_*\mathcal D^{2d}_{X/S},F)$,

induces on cohomology isomorphisms of sheaves on $S^{an}$ (see definition \ref{Hodgesub}):
\begin{itemize}
\item $ev_f:\mathcal H^{k}_S(f_U)/F^p\mathcal H^k_S(f_U)\xrightarrow{\sim}D^{\vee}_{O_S}(\mathcal H^{k}_S(f_{X,D}))$ and  
\item $ev_f:\mathcal H^{p,k-p}_S(f_U)\xrightarrow{\sim}D^{\vee}_{O_S}(\mathcal H^{d-p,d-k}_S(f_{X,D}))$. 
\end{itemize}
\item [(ii)] The pairing of filtered complexes of presheaves on $S^{an}$:

$f_*ev_X=<\cdot,\cdot>_{f_*ev_X}:(f_*(\mathcal D_X^{\bullet}(\log D)/L^2),F)\otimes_{O_S}(f_*L^{d_S-1}\mathcal A_X^{2d_X-\bullet}(\nul D),F)
\to (f_*\mathcal D_X^{2d_X},F)$

induces on cohomology isomorphisms of sheaves on $S^{an}$:
\begin{equation*}
f_*ev_X:R^qf_*(\Omega_X^p(\log D)/L^2)\xrightarrow{\sim}R^{d-q}f_*(L^{d_S-1}\Omega_X^{d_X-p}(\nul D)).
\end{equation*}
\end{itemize}
\end{prop}

\begin{proof}
(i):As these sheaves on $S^{an}$ are locally free sheaves of $O_S$ modules, it suffices to show that the evaluation of the induced maps
at every point $s\in S$ are isomorphisms.
But this is Poincare duality for the pair $(X_s,D_s)$ (c.f proposition \ref{dual}). 

(ii): As in (i), since these sheaves on $S^{an}$ are locally free sheaves of $O_S$ modules, 
it suffices to show that the evaluation of the induced maps at every point $s\in S$ are isomorphisms.
But this is Serre duality for $X_s$ since 
\begin{equation}
\Omega^p_X(\log D)_{|X_s}=D_{O_{X_s}}^{\vee}(\Omega_X^{d_X-p}(\nul D)_{|X_s})\otimes K_{X|X_s}
=D_{O_{X_s}}^{\vee}(\Omega_X^{d_X-p}(\nul D)_{|X_s})\otimes K_{X_s}
\end{equation}
by proposition \ref{ideal}(ii) for $X$, and the fact that $K_{X|X_s}\simeq K_{X_s}$. 

\end{proof}

\begin{prop}\label{RestCur} 
For $s\in S$ and
\begin{itemize}
\item 
$T\in\Gamma(W(s),f_*\mathcal D^{\bullet}_{X/S}(\log D))^{d_{X/S}=0}$ and 
$T'\in\Gamma(W(s),f_*\mathcal D^{\bullet,\bullet}_{X/S}(\log D))^{\bar\partial_{X/S}=0}$, 
whose restriction to the fibers of $f$ is proper 
(c.f.\cite{King} for the definition of the pullback or Gynsin map for currents)
\item $\eta\in\Gamma(W(s),f_*\mathcal A^{2d-\bullet}_{X/S}(\nul D))^{d_{X/S}=0}$ and
$\eta'\in\Gamma(W(s),f_*\mathcal A^{d-\bullet,d-\bullet}_{X/S}(\nul D))^{\bar\partial_{X/S}=0}$,
\end{itemize}
where $W(s)\subset S$ is a neighborhood of $s$ in $S$, we have
\begin{equation*}
<T,\eta>_{ev_f}(s)=X_s\cdot T(\eta)=<T_{|X_s},\eta_{|X_s}>_{ev_{X_s}} \; \; \mbox{and} \; \; 
<T',\eta'>(s)=X_s\cdot T'(\eta')=<T'_{|X_s},\eta'_{|X_s}>_{ev_{X_s}}. 
\end{equation*}
This gives on cohomology
\begin{equation*}
<[T],[\eta]>_{ev_f}(s)=<[T](s),[\eta](s)>_{ev_{X_s}} \; \; \mbox{and} \; \; 
<[T'],[\eta']>(s)=<[T'](s),[\eta'](s)>_{ev_{X_s}}. 
\end{equation*}
In particular, if $\omega\in\Gamma(W(s),f_*\mathcal A^{\bullet}_{X/S}(\log D))^{d_{X/S}=0}$
and $\omega'\in\Gamma(W(s),f_*\mathcal A^{\bullet,\bullet}_{X/S}(\log D))^{\bar\partial_{X/S}=0}$  
are log forms, then $<[\omega],[\eta]>(s)=\int_{X_s}\omega\wedge\eta$ 
and $<[\omega'],[\eta']>(s)=\int_{X_s}\omega'\wedge\eta')$. 
\end{prop}

\begin{proof}
See \cite[proposition 3.2.2]{King}.
\end{proof}

%============================================================
\subsection{The Gauss-Manin connexion}
%=============================================================

We have the the commutative diagram of filtered complexes of sheaves on $X^{an}$, $F$ being the Fr\"olicher filtration on $F_b$ the filtration bête,
\begin{equation}\label{GM}
\xymatrix{
0\ar[r] & (\Gr_L^1\Omega^{\bullet}_X(\log D),F_b)\ar[r]^{r^{\vee}}\ar[d] & (\Omega^{\bullet}_X(\log D)/L^2,F_b)\ar[r]^{q}\ar[d] 
& (\Omega^{\bullet}_{X/S}(\log D),F_b)\ar[r]\ar[d] & 0 \\
0\ar[r] & (\Gr_ L^1\mathcal{A}^{\bullet}_X(\log D),F)\ar[r]^{r^{\vee}}\ar[d] & (\mathcal{A}^{\bullet}_{X}(\log D)/L^2,F)\ar[r]^{q}\ar[d] 
& (\mathcal{A}^{\bullet}_{X/S}(\log D),F)\ar[r]\ar[d] & 0 \\
0\ar[r] & (\Gr_ L^1\mathcal{D}^{\bullet}_X(\log D),F)\ar[r]^{r^{\vee}} & (\mathcal{D}^{\bullet}_{X}(\log D)/L^2,F)\ar[r]^{q} 
& (\mathcal{D}^{\bullet}_{X/S}(\log D),F)\ar[r] & 0}
\end{equation}
where the row are by definition exact sequences of filtered complexes 
(the embedding $r^{\vee}=\Gr_L^1\hookrightarrow L^0/L^2$ is the quotient of the inclusion $L^1\subset L^0$ by $L^2$ 
and $q:L^0/L^2\to\Gr_L^0$ is the projection )
and the column are filtered quasi-isomorphisms by proposition \ref{FLqsi}.

Consider also the commutative diagram of filtered complexes of sheaves on $X^{an}$ whose rows are exact :
\begin{equation}\label{GMc}
\xymatrix{
0\ar[r] & (\Gr_L^1\mathcal A^{\bullet}_X(\nul D),F)\ar[r]^{r^{\vee}}\ar[d] & (\mathcal A^{\bullet}_X(\nul D)/L^2,F)\ar[r]^{q}\ar[d] 
& (\mathcal A^{\bullet}_{X/S}(\nul D),F)\ar[r]\ar[d] & 0 \\
0\ar[r] & (\Gr_ L^1\mathcal{A}^{\bullet}_X,F)\ar[r]^{r^{\vee}}\ar[d] & (\mathcal{A}^{\bullet}_{X}/L^2,F)\ar[r]^{q}\ar[d] 
& (\mathcal{A}^{\bullet}_{X/S},F)\ar[r]\ar[d] & 0 \\
0\ar[r] & (\Gr_ L^1\mathcal{A}^{\bullet}_X(\log D),F)\ar[r]^{r^{\vee}} & (\mathcal{A}^{\bullet}_{X}(\log D)/L^2,F)\ar[r]^{q} 
& (\mathcal{A}^{\bullet}_{X/S}(\log D),F)\ar[r] & 0} 
\end{equation}

\begin{defi}\label{GMdefi}
The Gauss Manin connexions of the bundles $\mathcal H^k_S(f_U)$, $\mathcal H^k_S(f_{X,D})$ respectively, 
are induced by the connecting morphism associated 
to the long cohomological exact sequence of last, respectively first, row of the diagram (\ref{GMc})
\begin{itemize}
\item $\nabla : F^p\mathcal H_S^k(f_U)\to\mathcal H^{k+1}f_*\Gr_L^1F^p\mathcal A^{\bullet}_X(\log D) 
=F^{p-1}\mathcal H^{k}_S(f_U)\otimes_{O_S}\Omega_S$,
\item $\nabla : F^p\mathcal H_S^k(f_{X,D})\to\mathcal H^{k+1}f_*\Gr_L^1F^p\mathcal A^{\bullet}_X(\nul D) 
=F^{p-1}\mathcal H^{k}_S(f_{X,D})\otimes_{O_S}\Omega_S$,
\end{itemize}
where the above equalities are given by the identifications (\ref{IdIn}) of propsition \ref{IdInProp} (see also remark \ref{IdInRem})
and by the projection formula ($X^{an}$ being paracompact the canonical map of sheaves on $S^{an}$
$f_*F\otimes_{C_X}G\to f_*(F\otimes f^*G)$ is an isomorphism).

Hence, for $W\subset S$ an open subset,
$\omega\in\Gamma(W,f_*F^p\mathcal A^{\bullet}_{X/S}(\log D))^{d_{X/S}=0}
=\Gamma(X_W,F^p\mathcal A^{\bullet}_{X/S}(\log D))^{d_{X/S}=0}$, 
$\eta\in\Gamma(W,f_*F^p\mathcal A^{\bullet}_{X/S}(\nul D))^{d_{X/S}=0}
=\Gamma(X_W,F^p\mathcal A^{\bullet}_{X/S}(\nul D))^{d_{X/S}=0}$ 
and $u\in\Gamma(W,T_S)$,
\begin{equation*}
\nabla_u([\omega])=\phi^{1,\bullet,\bullet}([d\omega])=[<\iota(\tilde{u})d\omega>] \; \; \mbox{and} \; \; 
\nabla_u([\eta])=\phi^{1,\bullet,\bullet}([d\eta])=[<\iota(\tilde{u})d\eta>],
\end{equation*}
where $\tilde u\in\Gamma(X_W,T_X)$ is a relevement of $u$ (i.e. $df(\tilde u)=u$). 
\end{defi}

\begin{rem}
The diagram (\ref{GMc}) and the identifications (\ref{IdIn}) of proposition \ref{IdInProp} induces the commutative diagrams
\begin{equation}
\xymatrix{
 \, & \nabla : F^p\mathcal H_S^k(f_{X,D})\ar[ld]\ar[dd]\ar[r] & F^{p-1}\mathcal H^k_S(f_{X,D})\otimes_{O_S}\Omega_S\ar[ld]\ar[dd] \\
\nabla : F^p\mathcal H_S^k(f)\ar[r]\ar[rd] & F^{p-1}\mathcal H^k_S(f)\otimes_{O_S}\Omega_S\ar[rd] & \, \\
\, & \nabla : F^p\mathcal H_S^k(f_{U})\ar[r] & F^{p-1}\mathcal H^k_S(f_{U})\otimes_{O_S}\Omega_S}
\end{equation}
\end{rem}

\begin{defiprop}
Let $\bar\nabla$  the morphism induced by $\nabla$ on graded pieces.
\begin{itemize}
\item $\bar\nabla :\mathcal H_S^{p,k-p}(f_U)\to\mathcal H_S^{p-1,k-p+1}(f_U)\otimes_{O_S}\Omega_S$ 
\item $\bar\nabla :\mathcal H_S^{p,k-p}(f_{X,D})\to\mathcal H_S^{p-1,k-p+1}(f_{X,D})\otimes_{O_S}\Omega_S$,
\end{itemize}
Then, for $W\subset S$ an open subset,
$\omega'\in\Gamma(W,f_*\mathcal A^{p,k-p}_{X/S}(\log D))^{\bar\partial_{X/S}=0}=
\Gamma(X_W,\mathcal A^{p,\bullet}_{X/S}(\log D))^{\bar\partial_{X/S}=0}$,
$\eta'\in\Gamma(W,f_*\mathcal A^{p,k-p}_{X/S}(\nul D))^{\bar\partial_{X/S}=0}=
\Gamma(X_W,\mathcal A^{p,\bullet}_{X/S}(\nul D))^{\bar\partial_{X/S}=0}$ and
$u\in\Gamma(W,T_S)$,
\begin{equation*}
\bar\nabla([\omega'])=\phi^{1,\bullet,\bullet}(\bar\partial\omega')=[<\iota(\tilde{u})\bar\partial\omega'>] \; \; \mbox{and} \; \;
\bar\nabla([\eta'])=\phi^{1,\bullet,\bullet}(\bar\partial\eta')=[<\iota(\tilde{u})\bar\partial\eta'>]
\end{equation*}
\end{defiprop}

\begin{proof}
This follows from corollary \ref{relvtR}(ii) and the description of the morphism $\nabla$.
\end{proof}

\begin{prop}\label{manin}
For simplicity of notation denote by $<\cdot,\cdot>=<\cdot,\cdot>_{ev_f}$.
\begin{itemize}
\item[(i)] We have, for $s\in S$, $u\in T_{S,s}$,
$\lambda\in\Gamma(W(s),\mathcal H^k_S(f_U))$ and $\mu\in\Gamma(W(s),\mathcal H^{2d-k}_S(f_{X,D}))$,
where $W(s)\subset S$ is an open neighborhood of $s$ in $S$ :
\begin{equation*}
<\nabla_u\lambda,\mu>(s)=d_u<\lambda,\mu>(s)-<\lambda,\nabla_u\mu>(s)
\end{equation*}
\item[(ii)] The pairing $ev_f=<\cdot,\cdot>$ induces isomorphisms
\begin{equation*}
ev_f:\mathcal H_S^{p,q}(f_U)\otimes\Omega_S/\Im(\bar\nabla)
\xrightarrow{\sim}(\mathcal H^{d-p,d-q}_S(f_{X,D})\otimes T_S)^{\,^t\bar\nabla=0}
\end{equation*}
where $^t\bar\nabla(\mu\otimes u)=\bar\nabla_u\mu$.
\end{itemize}
\end{prop}

\begin{proof}
(i): Shrinking $W(s)$ if necessary, there exist closed forms
 $\omega\in\Gamma(W(s),f_*\mathcal A^k_{X/S}(\log D))^{d_{X/S}=0}$
and $\eta\in\Gamma(W(s),f_*\mathcal A^{2d-k}_{X/S}(\nul D))^{d_{X/S}=0}$
such that $[\omega]=\lambda$ and $[\eta]=\mu$.
Then, 
\begin{eqnarray*}
d_u<\lambda,\mu>(s)&=&d_u<[\omega],[\eta]>(s) \\
&=&d_u(s'\mapsto\int_{X_{s'}}\omega\wedge\eta)(s)  
=\int_{X_s}\iota(\tilde{u})d(\omega\wedge\eta) \\ 
&=&\int_{X_s}(\iota(\tilde{u})d\omega)\wedge\eta+\int_{X_s}\omega\wedge(\iota(\tilde{u})d\eta)  
=<\nabla_u[\omega],[\eta]>(s)+<[\omega],\nabla_u[\eta]>(s)
\end{eqnarray*}

(ii): If $\lambda'\in\Gamma(W(s),F^p\mathcal H_S^k(f_U))$ and $\mu'\in\Gamma(W(s),F^{k-p+1}\mathcal H_S^k(f_U))$
we have $<\lambda',\mu'>=0$ as Poincare duality for the pair $(X_s,D_s)$ is a morphism of mixed hodge structures.
Hence by (i), $<\nabla_u\lambda',\mu'>=<\lambda',\nabla_u\mu'>$. 
Thus, $<\bar\nabla_u\bar\lambda',\bar\mu'>=<\bar\lambda',\bar\nabla_u\bar\mu'>$.
Point (ii) follows from this equality.

\end{proof}

The F graded piece of first and last rows of the diagram (\ref{GM}) is the commutative diagram 
\begin{equation}\label{GMgr}
\xymatrix{
0\ar[r] & \Omega^{p-1}_{X/S}(\log D)\otimes f^*\Omega_S\ar[r]^{r^{\vee}}\ar[d] & \Omega^p_X(\log D)/L^2\ar[r]^{q}\ar[d] 
& \Omega^p_{X/S}(\log D)\ar[r]\ar[d] & 0 \\
0\ar[r] & \mathcal{D}^{p-1,\bullet}_{X/S}(\log D)\otimes f^*\Omega_S\ar[r]^{r^{\vee}} & \mathcal{D}^{p,\bullet}_{X}/L^2\ar[r]^{q} 
& \mathcal{D}^{p,\bullet}_{X/S}(\log D)\ar[r] & 0,}
\end{equation}
Dually to this diagram we have the following commutative diagram :
\begin{equation}\label{DGMgr}
\xymatrix{
0\ar[r] & \Omega^{d-p}_{X/S}(\nul D)\ar[r]^{q^{\vee}}\ar[d] & L^{d_S-1}\Omega^{d_X-p}_X(\nul D)\ar[r]^{r}\ar[d]
 & \Omega^{d-p+1}_{X/S}(\nul D)\otimes_{O_X} f^*T_S\ar[r]\ar[d] & 0 \\
0\ar[r] & \mathcal{A}^{d-p,\bullet}_{X/S}(\nul D)\ar[r]^{q^{\vee}} & L^{d_S-1}\mathcal A^{d_X-p,\bullet}_X(\nul D)\ar[r]^{r} 
& \mathcal{A}^{d-p+1,\bullet}_{X/S}(\nul D)\otimes_{O_X} f^*T_S\ar[r] & 0,}
\end{equation}
whose rows are by definition exact sequence of complexes of sheaves, 
where 
\begin{itemize}
\item $q^{\vee}:\omega\mapsto\omega\wedge f^*\kappa$ is the dual of $q$,
\item  $r:L^{d_S-1}\Omega^{d_X-p}_X(\nul D)\to\Gr_L^{d_S-1}\Omega^{d_X-p}_X(\nul D)
\xrightarrow{\sim}\Omega^{d-p+1}_{X/S}(\nul D)\otimes_{O_X} f^*\Omega^{d_S-1}_S$ 
is the dual of $r^{\vee}$, 
\end{itemize}
and whose columns are quasi-isomorphism by proposition \ref{LGrReso}(the Dolbeau resolutions).

The maps $ev_f=f_*<w_X>$ and $f_*w_X$ induces a pairing between the images by $f_*$ 
of the second rows of the diagramms (\ref{GMgr}) and (\ref{DGMgr}) of sheaves on $X^{an}$. 
Thus, by proposition \ref{dualR},
it induces an isomorphism between the two long cohomological exact sequences of these two exact sequences :
\begin{equation}\label{GMgrDGMgr}
\xymatrix{
\ar[r]^{\bar\nabla}
&\mathcal H^{p-1,q}_S(f_U)\otimes f^*\Omega_S\ar[r]^{r^\vee}\ar[d]^{ev_f}_{\sim} 
& R^qf_*(\Omega_X^p(\log D)/L^2)\ar[r]^{q^{\vee}}\ar[d]^{f_*w_X}_{\sim} & \ar[d]^{ev_f}_{\sim}\mathcal H^{p-1,q}(f_U)  \\
\ar[r]^{(\,^t\bar\nabla)^{\vee}}
&D_{O_S}^{\vee}(\mathcal H^{d-p,d-q}_S(f_{X,D})\otimes T_S)\ar[r]^{r^{\vee}} 
& D_{O_S}^{\vee}(R^{d-q}f_*L^{d_S-1}\Omega_X^{d_X-p}(\nul D))\ar[r]^{q^{\vee}} & D^{\vee}_{O_S}(\mathcal H^{d-p+1,d-q}_S(f_{X,D}))}
\end{equation}

\subsection{Normal functions and infinitesimal invariants}

\begin{defi}
The relative intermediate jacobian of $f$ is
the of the fibration of complex analytic varieties
\begin{equation*}
J^{p,k}(f_U)=\mathcal H^{k}_S(f_U)/(F^p\mathcal H^{k}_S(f_U)\oplus H^{k}_{\mathbb Z}(f_U))\to S.
\end{equation*}
By proposition \ref{dualR}, the map $ev_f$ induces an isomorphism over $S$
\begin{equation*}
ev_f:J^{p,k}(f_U)\xrightarrow{\sim} D^{\vee}_{O_S}(F^{d-p+1}\mathcal H^{2d-k}_S(f_{X,D}))/H_{2d-k,\mathbb Z}(f_{X,D})
\end{equation*}

A normal function is a holomorphic section
$\nu\in\Gamma(S,J^{p,k}(f_{U}))$
of the fibration $\mathcal J^{p,k}(f_U)\to S$, such that every local relevement 
$\tilde\nu_W\in\Gamma(W,\mathcal H^{k}_S(f_U))$ of $\nu$ over an open subset $W\subset S$
is horizontal, i.e. $\tilde\nu_W$ is holomorphic and satisfy
$\nabla\tilde\nu_W\in\Gamma(W,F^{p-1}\mathcal H^{k}_S(f_U)\otimes_{O_S}\Omega_S)$.
Denote by $NF(f_U)(S)\subset\Gamma(S,J^{p,k}(f_{U}))$ the subspace of normal functions.
\end{defi}

\begin{defiprop}
\begin{itemize}
\item Let $\nu\in NF(f_U)(S)$. Then for $W\subset S$ and $\tilde\nu_W\in\Gamma(W,\mathcal H^{k}_S(f_U))$,
the class 
$$[\overline{\nabla\tilde\nu_W}]\in\Gamma(W,\mathcal H_S^{p-1,k-p+1}(f_U)\otimes_{O_S}\Omega_S/\Im\bar\nabla)$$
of the projection $\overline{\nabla\tilde\nu_W}\in\Gamma(W,\mathcal H_S^{p-1,k-p+1}(f_U)\otimes_{O_S}\Omega_S)$
of $\nabla\tilde\nu_W$ modulo the image of $\bar\nabla$ does not depends on the choice of a relevement.
Thus the local sections $[\overline{\nabla\tilde\nu_W}]$ patches together to get the infinitesimal invariant of $\mu$ 
\begin{equation*}
\delta\nu\in\Gamma(S,(\mathcal H^{p-1,k-p+1}_S(f_U)\otimes_{O_S}\Omega_S)/\Im\bar\nabla), \;
\delta\nu_{|W}=[\overline{\nabla\tilde\nu_W}] \; \mbox{for all local relevements}.
\end{equation*}  
\item Let $\nu\in\Gamma(S,J^{p,k}(f_{U}))$, then using the exact sequence of sheaves on $S^{an}$
$0\to H_{\mathbb Z}^{k}(f_U)\to\mathcal H^k_S(f_U)/F^p\mathcal H^k_S(f_U)\to J^{p,k}(f_{U})\to 0$ by definition
of $J^{p,k}(f_{U})$, $\nu$ has a cohomology class $[\nu]\in H^1(S, H_{\mathbb Z}^k(f_U))$.
\end{itemize}
\end{defiprop}

\begin{proof}
Standard.
\end{proof}

%================================================================================
\subsection{Relative Abel jacobi map and infinitesimal invariants}
%=================================================================================

Denote by
\begin{itemize}
\item $\mathcal{Z}^p(U,\bullet)^{pr/X/S}\subset\mathcal{Z}^p(U,\bullet)^{pr/X}$ be the subcomplex consisting of $Z=\sum_in_iZ_i\in\mathcal Z^p(U,n)$
such that their closures $\bar Z=\sum_in_i\bar Z_i\in\mathcal Z^p(X,n)$ intersect all the fibers $X_s\hookrightarrow X$ of $f:X\to S$ properly.
By Bloch this inclusion of complexes of abelian group is a quasi-isomorphism : 
consider the generic fiber of $f$ and go on by a decreasing induction on the dimension of subvarieties of $S$.
\item $\mathcal{Z}^p(U,n)_{\partial=0}^{pr/X\hom/S}\subset\mathcal{Z}^p(U,n)^{pr/X/S}$ the subspace such that $\partial Z=0$ 
and $[\Omega_{Z_s}]=0\in H^{2p-n}(U_s,\mathbb C)\xrightarrow{\sim}H^{2d-2p+n}(X_s,D_s,\mathbb C)^{\vee}$ for all $s\in S$. 
\end{itemize}

Let $Z\in\mathcal{Z}^p(U,n)_{\partial=0}^{pr/X\hom/S}$. 
Recall that its closure $\bar{Z}\in\mathcal Z^p(X,n)$ satisfy $\partial\bar{Z}\in i_{D*}\mathcal Z^{p-1}(D,n)$.  
Let 
\begin{equation}
R_{Z/S}=q(R_Z)\in\Gamma(X,\mathcal D_{X/S}^{2p-n-1}))=\Gamma(S,f_*\mathcal D^{2p-n-1}_{X/S}) 
\end{equation}
where $q:\mathcal D_X(\log D)\to\mathcal D_{X/S}(\log D)$ is the quotient map of sheaves on $X^{an}$.
By hypothesis, for all $s\in S$, 
\begin{equation}
[\Omega_Z]_{|X_s}=[\Omega_{Z|X_s}]=[\Omega_{Z_s}]=0\in H^{2p-n}(U_s,\mathbb C). 
\end{equation}
Hence, by proposition \ref{R'}, for all $s\in S$, $R_{Z_s}=R_{Z|X_s}$ 
restrict to a closed current on $F^{d-p+1}\mathcal A_{X_s}^{2d-2p+n+1}(\nul D)$.
That is, $R_{Z/S}$ restrict to a closed current on $F^{d-p+1}\mathcal A_{X/S}^{2d-2p+n+1}(\nul D)$,
and choice of $\Gamma_{\bar Z_s}\in C^{\diff}_{2d-2p+n+1}(X_s,D_s,\mathbb Z)$ 
such that 
\begin{equation}
\partial\Gamma^{\epsilon}_{\bar Z_s}=\bar Z_{s\epsilon}\in C^{\diff}_{2d-2p+n}(X_s,D_s,\mathbb Z), 
\end{equation}
for each $s\in S$ gives the following section 
$\tilde\nu_Z\in\Gamma(S,D_{O_S}^{\vee}(F^{d-p+1}\mathcal H^{2d-2p+n+1}_S(f_{X,D})))$ 
of the dual vector bundle of the mixed hodge subbundle :
\begin{eqnarray*}
\tilde\nu_Z(s):=ev_f(R_{Z/S})(s):=(\eta\in\Gamma(X_{W(s)},F^{d-p+1}\mathcal A_{X/S}^{2d-2p+n+1}(\nul D))^{d_{X/S}=0} \\
\mapsto <[R_{Z/S}],[\eta]>_{ev_f}(s)=R_{Z_s}(\eta_{|X_s})=\int_{\Gamma_{\bar{Z}_s}}\pi_X^*\eta\wedge\pi_{(\mathbb P^1)^n}\Omega_{\square^n})
\end{eqnarray*}
where $W(s)\subset S$ is an open neighborhood of $s$ in $S$, the first equality follows from proposition \ref{RestCur},
and the last equality from proposition \ref{R'}.

\begin{thm}\label{normalfunction}
Let $Z=\sum_i n_i Z_i\in\mathcal Z^p(U,n)_{\partial=0}^{pr/X\hom/S}$. Then,
\begin{equation*}
\nu_Z\in\Gamma(S,D_{O_S}^{\vee}(F^{d-p+1}\mathcal H_S^{2d-2p+n+1}(f_{X,D}))/H_{\mathbb Z,2d-2p+n+1}(f_{X,D})), \; \;
\nu_Z(s)=[\tilde\nu_Z(s)] 
\end{equation*}
is a normal function (i.e. holomorphic and horizontal), the higher normal function associated to $Z$. 
\end{thm}

\begin{proof}
For simplicity of the notation denote $<\cdot,\cdot>=<\cdot,\cdot>_{ev_f}$. 

Let $s\in S$. There exists a diffeomorphism 
\begin{equation*}
T:(X_{W(s)},D_{W(s)})\xrightarrow{\sim}(X_s,D_s)\times W(s),\; \;  T(x)=(T_s,f), 
\end{equation*}
over a sufficialy small open neigborhood $W(s)\subset S$ of $s$ in $S$.
As $[\Omega_{Z|W(s)}]=0\in H^{BM}_{2d-2p+n}(X_{W(s)},D_{W(s)},\mathbb C)$, 
there exist $\Gamma_{\bar Z_{W(s)}}\in C^{\diff,BM}(X_{W(s)},D_{W(s)},\mathbb Z)$,
intersecting properly the fibers of $f:X\to S$, such that $\partial\Gamma^{\epsilon}_{\bar Z_{W(s)}}=Z_{\epsilon|W(s)}$.
Then  we can choose $\Gamma_{\bar Z_{s'}}=\Gamma_{\bar Z_{W(s)}|X_{s'}}$ for $s'\in W(s)$ and we see that $\nu_{Z|W(s)}=\tilde\nu_{ZW(s)}$ is $C^{\infty}$
This shows that $\nu_Z$ is $C^{\infty}$ and in particular continous on $S^{an}$.
Hence, to prove the holomorphicity and the horizontality of $\nu_Z$, it is enough by continuity of $\nu_Z$ on $S^{an}$, to prove
it on a Zariski analytic open subset of $S$ since it is dense in $S^{an}$.
Thus, we can restrict to the  Zariski open subset of $S^o\subset S$ over which the families $f_{|\bar Z_i}:\bar Z_i\to S$ are isosingular.

Let $s\in S^o$. There exists a diffeomorphism 
\begin{equation*}
T:(X_{W(s)},D_{W(s)})\xrightarrow{\sim}(X_s,D_s)\times W(s), \; \;  T(x)=(T_s,f),
\end{equation*}
over a sufficialy small open neigborhood $W(s)\subset S$ of $s$ in $S^o$ such that $T$ induces on $Z_i$ trivialisations :
\begin{equation*}
T=T_{|\bar Z_i}:(\bar Z_i,D\cap \bar Z_i)\xrightarrow{\sim}(\bar Z_{is},D\cap\bar Z_i)\times S^o
\end{equation*}
and such that $T^{-1}(x\times S^o)$ are complex subvarieties of $X$. 
We can choose $\Gamma_{\bar Z_{s'}}=T^{-1}(\Gamma_{\bar Z_s}\times s')$ for $s'\in W(s)$
Then, for $u\in\Gamma(W(s),T^{0,1}_{S})$, $\tilde u\in\Gamma(X_{W(s)},T^{0,1}_X)$ a relevement of $u$ of type $(0,1)$
i.e. $d_f(\tilde u)=u$, and $\eta\in\Gamma(X_{W(s)},F^{d-p+1}\mathcal A_{X/S}^{2d-2p+n+1}(\nul D))^{d_{X/S}=0}$,
$(\iota(\tilde u)d\eta)\in\Gamma(X_{W(s)},F^{d-p+1}\mathcal A_X(\nul D))$ is $d_{X/S}$ exact. Hence,
\begin{equation*}
d_u<\nu_Z,[\eta]>(s)=d_u(s'\mapsto\int_{\Gamma_{\bar Z_{s'}}}\pi_X^*\eta\wedge\pi_{(\mathbb P^1)^n}\Omega_{\square^n})(s)
=\int_{\Gamma_{\bar Z_{s}}}\pi_X^*\iota(\tilde u)d\eta\wedge\pi_{(\mathbb P^1)^n}\Omega_{\square^n}=0
\end{equation*}
since the form $(\iota(\tilde u)d\eta)_{|X_s}\in\Gamma(X_s,F^{d-p+1}\mathcal A_{X_s}(\nul D))$ is exact and in $F^{d-p+1}$.
This proves that $\nu_Z$ is holomorphic.
Now let $\omega\in\Gamma(X_{W(s)},F^{d-p+2}\mathcal A^{2d-2p+n+1}_X(\nul D))^{d_{X/S}=0}$ and $u\in\Gamma(W(s),T_S)$.
Then $\iota(\tilde u)d\omega\in\Gamma(X_{W(s)},F^{d-p+1}\mathcal A^{2d-2p+n+1}_X(\nul D))^{d_{X/S}=0}$
and $\nabla_u[\omega]=[\iota(\tilde(u))d\omega]\in\Gamma(W(s),F^{d-p+1}\mathcal H^{2d-2p+n+1}_S(f_{X,D}))$.
Hence, by proposition \ref{manin} (i),
\begin{eqnarray*}
<\nabla_u\tilde\nu_Z,[\omega]>(s)&=&<\nu_Z,\nabla_u[\omega]>(s)-d_u<\nu_Z,\omega>(s) \\
&=&d_u(s'\mapsto\int_{\bar Z_{s'}}\pi_X^*\omega\wedge\pi_{(\mathbb P^1)^n}\Omega_{\square^n})(s)-
\int_{\bar Z_{s}}\pi_X^*\iota(\tilde u)d\omega\wedge\pi_{(\mathbb P^1)^n}\Omega_{\square^n} \\
&=&\int_{\bar Z_{s}}\pi_X^*\iota(\tilde u)d\omega\wedge\pi_{(\mathbb P^1)^n}\Omega_{\square^n}-
\int_{\bar Z_{s}}\pi_X^*\iota(\tilde u)d\omega\wedge\pi_{(\mathbb P^1)^n}\Omega_{\square^n}=0.
\end{eqnarray*}
This proves that $\nu_Z$ is horizontal.
\end{proof}

\begin{defi}
The map  
\begin{equation*}
AJ_{f_U}:\mathcal Z^p(U,n)_{\partial=0}^{pr/X\hom/S}\to\CH^p(U,n)^{\hom/S}\to NF_S(f_U)\subset\Gamma(S,J^{p,2p+n-1}(f_U)), \; \; 
Z\mapsto \nu_Z=[\tilde\nu_Z]
\end{equation*} 
is the relative higher Abel Jacobi map.
The image of $AJ(f_U)$ lies in the subspace $NF_S(f_U)\subset\Gamma(S,J^{p,2p+n-1}(f_U))$ by theorem \ref{normalfunction}
\end{defi}

\begin{prop}
\begin{itemize}
\item[(i)] We have decompositions 
\begin{itemize}
\item $m_{Uk}=(m^0_{Uk},\cdots,m^k_{Uk}):
H^k(U,\mathbb Z)\xrightarrow{\sim}\oplus_{r=0}^k H^{k-r}(S,H^{r}_{\mathbb Z}(f_U))$ 
\item $m_{(X,D)k}=(m^0_{(X,D)k},\cdots,m^k_{(X,D)k}):
H^k(X,D,\mathbb Z)\xrightarrow{\sim}\oplus_{r=0}^k H^{k-r}(S,H^{r}_{\mathbb Z}(f_{X,D}))$
\end{itemize}
\item[(ii)] For $Z\in\mathcal Z^p(U,n)_{\partial=0}^{pr/X\hom/S}$, 
we have $[\nu_Z]=m^1_{U(2p-n)}([T_Z])\in H^1(S,H_{\mathbb Z}^{2p-n-1}(f_U))$.
\end{itemize}
\end{prop}

\begin{proof}
(i): This follows from the fact that the morphisms $f:X\to S$ and $f_{D_J}\to S$ are smooth projectif.
Indeed, assume for simplicity that $D=D_1\subset X$ is smooth. Then we have decompositions
\begin{equation}
\xymatrix{
Rf_*\mathbb Z\ar[r]\ar[d]^{m_X} 
& Rf_{U*}\mathbb Z\ar[r]\ar[d]^{m_U} & Rf_{D*}\mathbb Z[-1]\ar[r]\ar[d]^{m_D} &  Rf_*\mathbb Z[1]\ar[d]^{m_X[1]}\\
\oplus_{r=0}^{2d}H^{r}_{\mathbb Z}(f)[-r]\ar[r] & \oplus_{r=0}^{2d}H^{r}_{\mathbb Z}(f_U)[-r]\ar[r] &
\oplus_{r=0}^{2d-2}H^{r}_{\mathbb Z}(f_D)[-r-1]\ar[r] & \oplus_{r=0}^{2d}H^{r}_{\mathbb Z}(f)[1-r]}
\end{equation}
and the map $m_U$ is the one induced by $m_X$ and $m_D$.
Taking hypercohomology gives the decompositions featuring in the commutative diagram 
whose rows are long exact sequence:
\begin{equation}
\xymatrix{
\cdots\ar[r] & H^k(X,\mathbb Z)\ar[r]\ar[d]^{m_{Xk}} 
& H^k(U,\mathbb Z)\ar[r]\ar[d]^{m_{Uk}} & H^{k-1}(D,\mathbb Z)\ar[r]\ar[d]^{m_{Dk}} & \cdots \\
\cdots\ar[r] & \oplus_{r=0}^kH^{k-r}(S,H^{r}_{\mathbb Z}(f))\ar[r] & \oplus_{r=0}^{k-r}H^r(S,H^{r}_{\mathbb Z}(f_U))\ar[r] &
\oplus_{r=0}^{k-1}H^{k-r-1}(S,H^{r}_{\mathbb Z}(f_D))\ar[r] & \cdots}
\end{equation}

The map $m_{(X,D)}$ is defined similarly : we have decompositions
\begin{equation}
\xymatrix{
Rf_*\mathbb Z\ar[r]\ar[d]^{m_X} 
& Rf_{D*}\mathbb Z\ar[r]\ar[d]^{m_D} & Rf_{X,D*}\mathbb Z[-1]\ar[r]\ar[d]^{m_{(X,D)}} &  Rf_*\mathbb Z[1]\ar[d]^{m_X[1]}\\
\oplus_{r=0}^{2d}H^{r}_{\mathbb Z}(f)[-r]\ar[r] & \oplus_{r=0}^{2d-2}H^{r}_{\mathbb Z}(f_D)[-r]\ar[r] &
\oplus_{r=0}^{2d}H^{r}_{\mathbb Z}(f_{X,D})[-r-1]\ar[r] & \oplus_{r=0}^{2d}H^{r}_{\mathbb Z}(f)[-r+1]}
\end{equation}
and the map $m_{(X,D)}$ is the one induced by $m_X$ and $m_D$.

(ii): Standard.
\end{proof}

Let $\mathcal Z^p(U,n)_{\partial=0}^{pr/X\hom/S}$. 
Denote again by $[\Omega_Z]=[\Omega_Z^{p,p-n}]\in H^{p-n}(X,\Omega_X^p(\log D))$ its class, 
recall that $\Omega_Z=\Omega_Z^{p,p-n}$ is of type $(p,p-n)$.
Denote again by $[\Omega_Z]\in\Gamma(S,R^{p-n}f_*(\Omega_X^p(\log D)))$ its image by
the canonical map $H^{p-n}(X,\Omega_X^p(\log D))\to\Gamma(S,R^{p-n}f_*\Omega_X^p(\log D))$.
Since $[q(\Omega_Z)]=0$, 
\begin{equation*}
[\Omega_Z]\in\ker(\Gamma(S,R^{p-n}f_*\Omega_X^p(\log D))\to\Gamma(S,R^{p-n}f_*\Omega_{X/S}^p(\log D)))
=\Gamma(S,L^1R^{p-n}f_*\Omega_X^p(\log D)).
\end{equation*}
Denote by 
\begin{itemize}
\item $[\Omega_Z/L^2]\ker(\Gamma(S,R^{p-n}f_*(\Omega_X^p(\log D)/L^2))\to\Gamma(S,R^{p-n}f_*\Omega_{X/S}^p(\log D)))$,
the image of $[\Omega_Z]$ by the projection $\Gamma(S,R^{p-n}f_*\Omega_X^p(\log D))\to\Gamma(S,R^{p-n}f_*(\Omega_X^p(\log D)/L^2))$.
\item  $[\Omega_Z]/L^2\in\Gamma(S,\Gr_L^1R^{p-n}f_*\Omega_X^p(\log D))$,
the image of $[\Omega_Z]$ by the projection $\Gamma(S,R^{p-n}f_*\Omega_X^p(\log D))\to\Gamma(S,(R^{p-n}f_*\Omega_X^p(\log D))/L^2)$.
\end{itemize}
We have the following commutative diagram of sheaves on $S^{an}$ :
\begin{equation}\label{phi12r}
\xymatrix{ \, &\Gr_L^1R^{p-n}f_*\Omega_X^p(\log D)\ar@{^{(}->}[ld]^{\psi^L_1}\ar@{^{(}->}[rd]^{\psi^L_2} & \, \\
\Omega_S\otimes\mathcal H_S^{p-1,p-n}(f_U)/\Im(\bar\nabla)\ar^{\bar{r^{\vee}}\sim}[rr] & \, & 
\ker(R^{p-n}f_*(\Omega_X^p(\log D)/L^2)\to R^{p-n}f_*\Omega_{X/S}^p(\log D))}
\end{equation}
where
\begin{itemize}
\item $\bar{r^{\vee}}:\Omega_S\otimes\mathcal H_S^{p-1,p-n}(f_U)/\Im(\bar\nabla)\to R^{p-n}f_*(\Gr_L^1\Omega^p_X(\log D))$ 
is the isomorphism induced by
\begin{itemize}
\item the morphism of sheaves on $S^{an}$
$r^{\vee}:R^{p-n}f_*\Gr_L^1\Omega^p_X(\log D)\to R^{p-n}f_*(\Omega^p_X(\log D)/L^2)$  
(induced in relative cohomoloy by the morphism of sheaves on $X^{an}$ $r^{\vee}:\Gr_L^1\Omega^p(\log D)\to\Omega^p_X(\log D)/L^2$),
\item  the isomorphism of sheaves on $S^{an}$ 
$\phi^{1,p}:R^{p-n}f_*\Gr_L^1\Omega^p_X(\log D)\xrightarrow{\sim}\Omega_S\otimes\mathcal H_S^{p-1,p-n}(f_U)$
(induced in $f$ direct image cohomology by the isomorphism of complexes of sheaves on $X^{an}$
$\phi^{1,p,\bullet}:\Gr^1_L\mathcal A^{p,\bullet}_X(\log D)\xrightarrow{\sim}\mathcal A^{p-1,\bullet}_{X/S}(\log D)\otimes\Omega_S$,
c.f. proposition \ref{IdInProp} and remark \ref{IdInRem})
\end{itemize} 
\item $\psi^2_L:\Gr_L^1R^{p-n}f_*\Omega^p_X(\log D)=E_{\infty}^{1,p-n}\hookrightarrow R^{p-n}f_*(\Gr_L^1\Omega^p_X(\log D))=E_1^{1,p-n}$
is the inclusion of sheaves on $S^{an}$ induced by the spectral sequence associated to the complex $(\Omega^p_X(\log D),L)$ :
for degree  reason no arrow $d_r$, $r\geq 2$ can lead to $E_r^{1,p-n}$. 
We have $\psi^L_2([\Omega_Z/L^2])=[\Omega_Z]/L^2$.
\item  $\psi^L_1:=\bar{r^{\vee}}^{-1}\circ\psi^L_2$ is the inclusion given by composition.
\end{itemize} 
 
The infinitesimal invariant associated to the class $[\Omega_Z]\in H^{p,p-n}(U,\mathbb C)$ is 
\begin{equation*}
\delta[\Omega_Z]:=\psi^L_1([\Omega_Z]/L^2)=\bar{r^{\vee}}^{-1}([\Omega_Z/L^2])
\in\Gamma(S,\Omega_S\otimes\mathcal H_S^{p-1,p-n}(f_U)/\Im(\bar\nabla))
\end{equation*}

\begin{lem}\label{diemterm}
Let $Z=\sum_in_iZ_i\in\mathcal Z^p(U,n)_{\partial=0}^{pr/X\hom/S}$ 
such that $\pi_X(\bar Z_i)\subset X$ is a local complete intersection for all $i$. 
Then, for $s\in S$ and $\gamma\in\Gamma(W(s),R^{d-p+n}f_*L^{d_S-1}\Omega^{d_X-p}_X(\nul D))$,
\begin{equation*}
<[\Omega_Z/L^2],\gamma>_{f_*ev_X}(s)=<[\Omega_Z/L^2](s),\gamma(s)>_{f_*ev_X(s)}
=\sum_i\int_{\bar Z_{is}}\pi_X^*\tilde{\gamma}(s)^{N_i}\wedge\pi_{(\mathbb P^1)^n}\Omega_{\square^n},
\end{equation*}
where, 
\begin{itemize}
\item $[\Omega_Z/L^2](s)\in O_s\otimes_{O_{S,s}}(R^{p-n}f_*(\Omega^{p}_X(\log D)/L^2))_s= H^{p-n}(X_s,(\Omega^{p}_X(\log D)/L^2)_{|X_s})$
\item $\gamma(s)\in O_s\otimes_{O_{S,s}}(R^{d-p+n}f_*L^{d_S-1}\Omega^{d_X-p}_X(\nul D))_s= H^{d-p+n}(X_s,(L^{d_S-1}\Omega^{d_X-p}_X(\nul D))_{X_s})$
\item $\tilde{\gamma}(s)^{N_i}=i_{\pi_X(\bar Z_i)}^*\tilde{\gamma}(s)\in H^{d-p+n}(\pi_X(\bar Z_i),L^{d_S-1}\Omega^{d_X-p}_{Z_i}(\nul D))$
\end{itemize}
are the evaluation in $s$ of the respective sheaves on $S^{an}$ of $O_S$ module,
and $\tilde{\gamma}(s)\in\Gamma(X_s,\Omega_X^{d_X-p}(\nul D)_{|X_s}\otimes_{O_{X_s}}\mathcal A^{0,d-p+n}_{X_s})^{\bar\partial=0}$
 is a closed form such that $[\tilde{\gamma}(s)]=\gamma(s)$.
\end{lem}

\begin{proof}
We can assume that $\pi_{X|Z_i}:Z_i\to\pi_X(Z_i)$ is generically finite, otherwise $\Omega_Z=0$.
It is then a staightforward generalization of the description given in \cite{Voisin} section 19.2.2. and the remark that the description is still correct
in the case the $\pi_X(Z_i)$ are not smooth but only local complete intersection in $X$ :
The class $[\Omega_{Z_i}/L^2](s)\in H^{d-p+n}(X_s,L^{d_S-1}\Omega^{d_X-p}_X(\nul D)_{|X_s})^{\vee}$ is given by the composite
\begin{eqnarray*}
H^{d-p+n}(X_s,L^{d_S-1}\Omega^{d_X-p}_X(\nul D)_{|X_s})\to H^{d-p+n}(X_s,\Omega^{d_X-p}_X(\nul D)_{|X_s}) \\
\to H^{d-p+n}(\pi_X(Z_{is}),\Omega^{d_X-p}_X(\nul D)_{|\pi_X(Z_{is})}) 
\to H^{d-p+n}(\pi_X(Z_{is}),\Omega^{d_X-p}_{\pi_X(Z_{is})}(\nul(D\cap\pi_X(Z_{is})))) \\
\to H^{d-p+n}(\pi_X(Z_{is}),\Omega^{d-p}_{\pi_X(Z_{is})}(\nul(D\cap\pi_X(Z_{is}))))\xrightarrow{ev_{Z_{is}}(\Omega_{Z_{is}})}\mathbb C.
\end{eqnarray*}
Note that $\dim\pi_X(Z)=\dim Z=d_X-p+n$ and $\dim\pi_X(Z_s)=\dim Z_s=d-p+n$.
\end{proof}

We have then one of the main result of this paper :

\begin{thm}\label{main}
Let $Z=\sum_i n_iZ_i\in\mathcal Z^p(U,n)_{\partial=0}^{pr/X\hom/S}$ such that $\pi_X(\bar Z_i)\subset X$ 
is a local complete intersection for all $i$.
Then $\delta\nu_Z=\delta[\Omega_Z]\in\Gamma(S,\Omega_S\otimes\mathcal H_S^{p-1,p-n}(f_U)/\Im(\bar\nabla))$.
\end{thm}

\begin{proof}
For simplicity of notation, denote by $<\cdot,\cdot>=<\cdot,\cdot>_{ev_f}$.
By proposition \ref{manin} (ii), we have to prove that
for all $s\in S$, and all $\mu\otimes u\in\Gamma(W(s),\mathcal H_S^{d-p+1,d-p+n}(f_{X,D}))^{\,^t\nabla=0}$,
where $W(s)\subset S$ is an open neighborhood of $s$ in $S$,
\begin{equation*}
<(\delta\nu_Z)_{|W(s)},\mu\otimes u>(s)=<(\delta[\Omega_Z])_{|W(s)},\mu\otimes u>(s)
\end{equation*}

So, let $s\in S$, and $\mu\otimes u\in\Gamma(W(s),\mathcal H_S^{d-p+1,d-p+n}(f_{X,D}))^{\,^t\bar\nabla=0}$. 
Shrinking $W(s)$ if necessary, 
there exist $\eta\in\Gamma(W(s),f_*\mathcal A^{d-p+1,d-p+n}_{X/S}(\nul D))^{\bar\partial_{X/S}=0}$
such that
\begin{equation*}
[\eta]=\mu\in\Gamma(W(s),\mathcal H_S^{d-p+1,d-p+n}(f_{X,D}))
\end{equation*}
(see definition \ref{Hodgesub}). By corollary \ref{relvtR}(ii),there exist 
$\hat{\eta}\in\Gamma(W(s),f_*F^{d-p+1}\mathcal A^{2d-2p+n+1}_{X/S}(\nul D))^{d_{X/S}=0}$
such that
\begin{equation*}
\overline{[\hat\eta]}=[\hat\eta^{d-p+1,d-p+n}]=[\eta]\in\Gamma(W(s),\mathcal H_S^{d-p+1,d-p+n}(f_{X,D})).
\end{equation*}
By definition, 
\begin{equation*}
\nabla_u[\hat\eta]=[\iota(\tilde u)d\hat\eta]\in\Gamma(W(s),F^{d-p}\mathcal H_S^{2d-2p+n+1}(f_{X,D})), 
\end{equation*}
with $\iota(\tilde u)d\hat\eta\in\Gamma(W(s),f_*F^{d-p}\mathcal A^{2d-2p+n+1}_{X/S}(\nul D))^{d_{X/S}=0}$.
By hypothesis, 
\begin{equation*}
\bar\nabla_u(\mu)=\overline{\nabla_u([\hat\eta])}=\overline{[\iota(\tilde u)d\hat\eta]}=\,^t\bar\nabla(\mu\otimes u)=0
\in\Gamma(W(s),\mathcal H_S^{d-p,d-p+n+1}(f_{X,D})),
\end{equation*}
that is 
\begin{equation*}
\nabla_u[\hat\eta]=[\iota(\tilde u)d\hat\eta]\in\Gamma(W(s),F^{d-p+1}\mathcal H_S^{2d-2p+n+1}(f_{X,D})).
\end{equation*}
Thus, using again the $E_1$ degenerescence of $(f_*\mathcal A_{X/S}(\nul D),F)$ (corollary \ref{relvtR} (i)),
there exist 
\begin{itemize}
\item $\alpha\in\Gamma(W(s),f_*F^{d-p}\mathcal A^{2d-2p+n+1}_{X/S}(\nul D))^{d_{X/S}=0}$,
\item $\beta\in\Gamma(W(s),f_*F^{d-p+1}\mathcal A^{2d-2p+n+1}_{X/S}(\nul D))^{d_{X/S}=0}$
\end{itemize}
such that
\begin{equation}\label{alphabeta}
\iota(\tilde u)d\hat\eta=\beta+d\alpha\in\Gamma(W(s),f_*F^{d-p}\mathcal A^{2d-2p+n+1}_{X/S}(\nul D))^{d_{X/S}=0}.
\end{equation} 

Let us now compute the first term 
$<(\delta\nu_Z)_{|W(s)},\mu\otimes u>(s)=<(\delta\nu_Z)(s),\mu(s)\otimes u(s)>_{ev_{X_s}}$.
We have
\begin{eqnarray*} 
<(\delta\nu_Z)_{|W(s)},\mu\otimes u>(s)&=&<\nabla_u[\tilde\nu_{Z,W}],[\hat\eta]>(s) \\
&=&d_u<[\tilde\nu_{Z,W}],[\hat\eta]>(s)-<[\tilde\nu_{Z,W}],\nabla_u[\hat\eta]>(s) \; \; \mbox{by proposition} \, \ref{manin} \, (i) \\
:&=&  d_u(s'\mapsto\int_{\Gamma_{\bar Z_{s'}}}\pi_X^*\hat\eta\wedge\pi_{(\mathbb P^1)^n}^*\Omega^n_{\square})-
\int_{\Gamma_{\bar Z_{s}}}\pi_X^*\beta\wedge\pi_{(\mathbb P^1)^n}^*\Omega^n_{\square} \\
\end{eqnarray*}
We have 
$d_u(s'\mapsto\int_{\Gamma_{\bar Z_{s'}}}\pi_X^*\hat\eta\wedge\pi_{(\mathbb P^1)^n}^*\Omega^n_{\square})
=\int_{\Gamma_{\bar Z_{s}}}\iota(\tilde u(\square^n))d(\pi_X^*\hat\eta\wedge\pi_{(\mathbb P^1)^n}^*\Omega^n_{\square})
=\int_{\Gamma_{\bar Z_{s}}}\pi_X^*\iota(\tilde u)d\hat\eta\wedge\pi_{(\mathbb P^1)^n}^*\Omega^n_{\square}$
where $\tilde u(\square^n)\in\Gamma(X_{W(s)}\times\square^n,T_{X\times\square^n})$ is a relevement of $\tilde{u}$
hence a relevement of $u$ for $f(\square^n):X\times\square^n\to S$, $f(\square^n)(x,t)=f(x)$, 
since $\iota(\tilde u(\square^n))\pi_{\mathbb P n}^*\ell=0$ for all differential form $\ell\in\Gamma((\mathbb C^*)^n,\mathcal A_{(\mathbb P^1)^n})$ 
(hence in particular $\iota(\tilde u(\square^ n))\pi_{\mathbb P n}^*d\Omega^n_{\square}=0$).
Hence,
\begin{eqnarray*} 
<(\delta\nu_Z)_{|W(s)},\mu\otimes u>(s)
&=&\int_{\Gamma_{\bar Z_{s}}}\pi_X^*(\iota(\tilde u)d\hat\eta-\beta)\wedge\pi_{(\mathbb P^1)^n}^*\Omega^n_{\square} 
=\int_{\Gamma_{\bar Z_{s}}}\pi_X^*d\alpha\wedge\pi_{(\mathbb P^1)^n}^*\Omega^n_{\square} \\
&=&\sum_i n_i\int_{\bar Z_{is}}\pi_X^*\alpha\wedge\pi_{(\mathbb P^1)^n}^*\Omega^n_{\square} 
=\sum_i n_i\int_{\bar Z_{is}}\pi_X^*\alpha^{d-p,d-p+n}\wedge\pi_{(\mathbb P^1)^n}^*\Omega^n_{\square} 
\end{eqnarray*}
where the third equality follows by Stoke formula and the last equality for type reason ($\Omega_{Z_s}$ is of type $(p,p-n)$). 

Let us compute the second term.
Shrinking $W(s)\subset S$ if necessary, there exist, by the exactness of the first row
of the diagramm of sheaves on $S^{an}$ (\ref{GMgrDGMgr}), 
\begin{equation*}
\gamma\in\Gamma(W(s),R^{d-p+n}f_*L^{d_S-1}\Omega^{d_X-p}_X(\nul D)) 
\end{equation*}
such that  $r(\gamma)=\mu\otimes u$.
By commutativity of this diagram (\ref{GMgrDGMgr}), 
\begin{equation*}
<\delta[\Omega_Z]_{|W(s)},\mu\otimes u>=<\delta[\Omega_Z]_{|W(s)},r(\gamma)>=<[\Omega_Z/L^2]_{|W(s)},\gamma>_{f_*ev_X}
\end{equation*}
Hence, by lemma \ref{diemterm},
\begin{eqnarray}\label{inf}
<\delta[\Omega_Z]_{|W(s)},\mu\otimes u>(s)=<\delta[\Omega_Z](s),\mu(s)\otimes u(s)>_{ev_{X_s}}&=&<\delta[\Omega_Z](s),r(s)(\gamma(s))>_{ev_{X_s}} \\
&=&<[\Omega_Z/L^2](s),\gamma(s)>_{f_*ev_X(s)} \\
&=&\sum_in_i\int_{\bar Z_{is}}\pi_X^*\tilde{\gamma}(s)^{N_i}\wedge\pi_{(\mathbb P^1)^n}\Omega_{\square^n}
\end{eqnarray}
Hence, we have to find a form $\xi\in\Gamma(X_s,L^{d_S-1}\mathcal A^{d_X-p,d-p+n}_X(\nul D)_{|X_s})^{\partial_{X_s}=0}$
such that
\begin{equation*}
r(s)[\xi]=\mu(s)\otimes u(s)\in H^{d-p+n}(X_s,\Omega^{d-p}_{X_s}(\nul D)). 
\end{equation*}
Consider the form
\begin{equation*}
\chi:=\hat\eta^{d-p+1,d-p+n}\wedge f^*\iota(u)\kappa+\alpha^{d-p,d-p+n}\wedge f^*\kappa
\in\Gamma(W(s),L^{d_S+1}\mathcal A^{d_X-p,d-p+n}_X(\nul D)). 
\end{equation*}
We have $r(\chi)=\hat\eta^{d-p+1,d-p+n}\otimes u$.
Taking the component of type $(d-p,d-p+n)$ in the relation
\begin{equation*}
(\iota(\tilde u)d\hat\eta)_{|X_s}=\beta_{X_s}+d\alpha_{|X_s}
\end{equation*}
which is the restriction of (\ref{alphabeta}) to $X_s$, we find
that the form
\begin{equation*}
\xi:=\chi_{|X_s}\in\Gamma(X_s,L^{d_S-1}\mathcal A_X^{d_X-p,d-p+n}(\nul D)_{|X_s})^{\bar\partial_{X_s}=0}
\end{equation*}
is closed. 
Moreover, since $r(\chi)=\hat\eta^{d-p+1,d-p+n}\otimes u$, we have
\begin{equation*}
r(s)(\xi)=r(s)(\chi_{|X_s})=\hat\eta^{d-p+1,d-p+n}_{|X_s}\otimes u. 
\end{equation*}
Hence, on cohomology $r(s)([\xi])=[\eta_{|X_s}]\otimes u(s)=\mu(s)\otimes u(s)$. We have the desired form.
Then (\ref{inf}) gives,
\begin{eqnarray*}
<\delta[\Omega_Z]_{|W(s)},\mu\otimes u>(s)=<\delta[\Omega_Z](s),\mu(s)\otimes u(s)>_{ev_{X_s}}&=&<(\delta[\Omega_Z](s),r(s)[\xi]>_{ev_{X_s}} \\
&=&<[\Omega_Z/L^2](s),[\xi]>_{f_*ev_X(s)} \\
&=&\sum_in_i\int_{\bar Z_{is}}\pi_X^*\alpha^{d-p,d-p+n}\wedge\pi_{(\mathbb P^1)^n}\Omega_{\square^n},
\end{eqnarray*}
where the last equality follows again from the fact that 
$\Omega_{Z_{is}}\in\Gamma(X_s,\mathcal D_X^{d-p,d-p+n}(\log D))$ is of type $(d-p,d-p+n)$.

\end{proof}

\begin{rem}
Note that the form $\alpha^{d-p,d-p+n}\in\Gamma(X_{W_s},\mathcal A^{d-p,d-p+n}_{X}(\nul D))$ is not $\bar\partial_{X/S}$
closed, hence not $d_{X/S}$ closed since it is of single type $(d-p,d-p+n)$,
that is $\alpha^{d-p,d-p+n}_{|X_s}$ is not $\bar\partial_{X_s}$ and not $d_{X_s}$ closed. 
But $\pi_X^*\alpha^{d-p,d-p+n}_{|Z^{\reg}_i}$ is $\bar\partial_{Z^{\reg}_i/S}$ closed,
that is $\pi_X^*\alpha^{d-p,d-p+n}_{|Z^{\reg}_{is}}$ is $\bar\partial_{Z^{\reg}_{is}}$ closed.
where, $Z^{\reg}_i\subset Z_i$ the smooth locus of $Z_i$.
Denote by $i_{\bar Z_i}:\bar Z_i\hookrightarrow X\times(\mathbb P^1)^n$ the closed embedding. 
On the other side, the current 
$\Omega_Z=\sum_i n_i \pi_{X*}i_{\bar Z_i*}\Omega_{Z_i}^{onZ^{\reg}_i}\in\Gamma(X,\mathcal D^{p,p-n}_X(\log D))$ 
is $d_{X/S}$ closed, hence  $\bar\partial_{X/S}$ closed since it is of single type $(p,p-n)$, 
that is $\Omega_{Z_s}$ is $d_{X_s}$ closed, since $\partial Z=0$.
But the currents $\Omega_{Z_i}^{onZ^{\reg}_i}\in\Gamma(\bar Z_i,\mathcal D_{Z^{\reg}_i}^{p,p-n}(\log (D\cap Z_i))$ 
are not $\bar\partial_{Z^{\reg}_i/S}$ closed.

\end{rem}

%%%%%%%%%%%%%%%%%%%%%%%%%%%%%%%%%%%%%%%%%%%%%%%%%%%%%%%%%%%%%%%%%%%%%%%%%%%%%%%%%%%%%%%%%%%
\section{Higher Abel Jacobi map for open complete intersection} 
%%%%%%%%%%%%%%%%%%%%%%%%%%%%%%%%%%%%%%%%%%%%%%%%%%%%%%%%%%%%%%%%%%%%%%%%%%%%%%%%%%%%%%%%%%%

Let $Y\in\PSmVar(\mathbb C)$ together with an embedding $Y\in\mathbb P^N$. For $d,e>>0$, the morphisms of $\mathbb C$ vector spaces
\begin{itemize}
\item $\Gamma(\mathbb P^N,O(d))\to\Gamma(Y,O_Y(d))=S_d$,
\item $\Gamma(\mathbb P^N,O(e))\to\Gamma(Y,O_Y(e))=S_e$, and  
\item $\Gamma(\mathbb P^N,O(d))\to\Gamma(Z,O_Z(d))=S_d$ for $Z\subset Y$ such that $Z\in\Gamma(Y,O_Y(e))$,
\end{itemize}
are surjective.
Denote by $p_{d,e}:Y\times S_d\times S_e\to S_d\times S_e$ and  $p_Y:Y\times S_d\times S_e\to Y$
the projections. 
Consider the commutative diagram of families of hypersurface sections of degre $d$ and $e$, whose squares are cartesians : 
\begin{equation}\label{HF}
\xymatrix{
f_D:\mathcal D=\mathcal X\cap\mathcal Z\ar@{^{(}->}[r]^{k_{\mathcal D}}\ar@{^{(}->}[d]^{i_{\mathcal D}} & \mathcal Z\ar@{^{(}->}[d]\ar[rd] \\
f:\mathcal X\ar@{^{(}->}[r]^{i_{\mathcal X}} &  Y\times S_d\times S_e\ar[r]^{p_{d,e}} & S_d\times S_e \\
f_U:\mathcal U=\mathcal X\backslash\mathcal D\ar@{^{(}->}[r]^{i_{\mathcal U}}\ar@{^{(}->}[u]^{j_{\mathcal U}} &  
(Y\times S_d\times S_e)\backslash\mathcal Z\ar@{^{(}->}[u]\ar[ru]}
\end{equation}
Note that $\mathcal X,\mathcal Z,\mathcal D\in\PSmVar(\mathbb C)$, since 
$p_{Y|\mathcal X}:\mathcal X\to Y$, $p_{Y|\mathcal Z}:\mathcal Z\to Y$, $p_{Y|\mathcal D}:\mathcal D\to Y$
are projective bundles and $Y$ is smooth.

For $0\in S_e$, denote by 
$p^0_Y=p_{Y|Y\times S_d\times 0}:Y\times S_d\times 0\to Y$ and 
$p^0_{Y\backslash Z_0}=p_{Y|(Y\backslash Z_0)\times S_d\times 0}:(Y\backslash Z_0)\times S_d\times 0\to Y\backslash Z_0$,
$p^0_d=p_{d,e|Y\times S_d\times 0}:Y\times S_d\times 0\to S_d$,
the projections,
and consider the pullback of the diagram (\ref{HF}) :
\begin{equation}\label{HF0}
\xymatrix{
f^o_D:D=X\cap (Z_0\times S_d)\ar@{^{(}->}[r]^{k_D}\ar@{^{(}->}[d]^{i_D} & Z_0\times S_d\ar@{^{(}->}[d]\ar[rd] \\
f^o:X=\mathcal X_{S_d\times 0}\ar@{^{(}->}[r]^{i_X} &  Y\times S_d\times 0\ar[r]^{p^0_d} & S_d \\
f^o_U:U=X\backslash D\ar@{^{(}->}[r]^{i_U}\ar@{^{(}->}[u]^{j} &  
(Y\backslash Z_0)\times S_d\ar@{^{(}->}[u]\ar[ru]}
\end{equation}
where $Z_0=p^0_Y(\mathcal Z_{S_d\times 0})\subset Y$ so that we have $\mathcal Z_{S_d\times 0}=Z_0\times S_d$.
Then $Y\backslash Z_0$ is an affine variety. For $s\in S_d$, 
consider the correspondence $\Delta(U_s)\subset(Y\backslash Z_0)\times(Y\backslash Z_0)$
which is the diagonal of $U_s$. Since the projection $\Delta(U_s)\to(Y\backslash Z_0)$ is proper
there is a well defined action of this correspondence on cohomology.
We denote by  
\begin{equation}
H^{k}(Y\backslash Z_0,\mathbb C)^0:=\ker(\Delta(U_s)_*)\subset H^{k}(Y\backslash Z_0,\mathbb C)
\end{equation}
the primitive cohomology of $Y\backslash Z_0$, that is kernel of this action.
For $s\in S_d$ such that $U_s\subset Y\backslash Z_0$ is smooth,
we have the equality (by Poincare duality for $U_s$) 
\begin{equation}
H^{k}(Y\backslash Z_0,\mathbb C)^0:=\ker(\Delta(U_s)_*)=\ker(i_{U_s}^*)\subset H^{k}(Y\backslash Z_0,\mathbb C),
\end{equation}
that is the primitive cohomology coincide with the kernel of pullback by the inclusion of an ample smooth hypersurface section.
Since $Y\backslash Z_0$ is affine, $H^{k}(Y\backslash Z_0,\mathbb C)=0$ for $k\geq d_Y+1$ and
$H^{d_Y}(Y\backslash Z_0,\mathbb C)^0=H^{d_Y}(Y\backslash Z_0,\mathbb C)$.

For a morphism $T\to S_d$, we consider the pullback of the diagram (\ref{HF0}) :
\begin{equation}\label{HF0T}
\xymatrix{
f^{T}_D:D_T=X_T\cap (Z_0\times T)\ar@{^{(}->}[r]^{k_{D_T}}\ar@{^{(}->}[d]^{i_{D_T}} & Z_0\times T\ar@{^{(}->}[d]\ar[rd] \\
f^{T}:X_T\ar@{^{(}->}[r]^{i_{X_T}} &  Y\times T\times 0\ar[r]^{p^T} & T \\
f^{T}_U:U_T=X_T\backslash D_T\ar@{^{(}->}[r]^{i_{U_T}}\ar@{^{(}->}[u]^{j_{U_T}} &  
(Y\backslash Z_0)\times T\ar@{^{(}->}[u]\ar[ru],}
\end{equation}
where $X_T=X\times_{S_d} T$, $U_T=U\times_{S_d} T$, $D_T=D\times_{S_d} T$. 

We now give a version of Nori connectness theorem for families of ample open hypersurfaces of $Y\in\PSmVar(\mathbb C)$.

\begin{thm}\label{nori}
Assume $d_Y\geq 4$
Let $0\in S_e$ sufficiently general and $S\subset S_d$ the open subset over which
such that the morphisms $f^0:X\to S_d$ and $f^0_D:D\to S_d$ are smooth projective.
Then, if $d,e>>0$, for all smooth morphism $T\to S_d$ and all $0\leq k\leq d_Y$, 
\begin{itemize}
\item[(i)] $i_{X_T}^*:H^{k-p}(Y\times T,\Omega^p_{Y\times T}(\log(Z_0\times T)))\xrightarrow{\sim} H^{k-p}(X_T,\Omega_{X_T}^p(\log D_T))$
is an isomorphism,
\item[(ii)] $i_{U_T}^*:H^{k}((Y\backslash Z_0)\times T,\mathbb C)\xrightarrow{\sim} H^{k}(U_T,\mathbb C)$
is an isomorphism of mixed hodge structure.
\end{itemize}
\end{thm}

\begin{proof}
(i): Consider the commutative diagram of sheaves on $Y\times T$ :
\begin{equation}
\xymatrix{
0\ar[r] & \Omega^p_{Y\times T}\ar[r]\ar[d]^{i_{X_T}^*} & \Omega^p_{Y\times T}(\log(Z_0\times T))\ar^{Res}[r]\ar[d]^{i_{X_T}^*} &
i_{(Z_0\times T)*}\Omega^{p-1}_{Z_0\times T}\ar[r]\ar[d]^{k_{D_T}^*} & 0 \\
0\ar[r] & i_{X_T*}\Omega^p_{X_T}\ar[r] & i_{X_T*}\Omega^p_{X_T}(\log D_T)\ar^{Res}[r] &
i_{X_T*}i_{D_T*}\Omega^{p-1}_{D_T}\ar[r] & 0}
\end{equation}
whose rows are exact sequences.

It induces in cohomology
\begin{equation}\label{Nori1}
\xymatrix{
\cdots\ar[r]  & H^{k-p}(Y\times T,\Omega^p_{Y\times T})\ar[r]\ar[d]^{i_{X_T}^*}
 & H^{k-p}(Y\times T, \Omega^p_{Y\times T}(\log(Z_0\times T)))\ar^{Res}[r]\ar[d]^{i_{X_T}^*} &
H^{k-p}(Z_0\times T,\Omega^{p-1}_{Z_0\times T})\ar[r]\ar[d]^{k_{D_T}^*} & \cdots \\
\cdots\ar[r]  & H^{k-p}(X_T,\Omega^p_{X_T})\ar[r] & H^{k-p}(X_T,\Omega^p_{X_T}(\log D_T)))\ar^{Res}[r] &
H^{k-p}(D_T,\Omega^{p-1}_{D_T})\ar[r] & \cdots}
\end{equation}
Now,
\begin{itemize}
\item by Nori connectness theorem for the pair $(Y\times S_d,X)$,
since $d_Y\geq 4$ (hence $d_Y< 2d_Y-2$), $d>>0$ and $T\to S_d$ is smooth, 
the map $i_{X_T}^*:H^{k-p}(Y\times T,\Omega^p_{Y\times T})\xrightarrow{\sim} H^{k-p}(X_T,\Omega^p_{X_T})$
is an isomorphism for all $0\leq k\leq d_Y$, 
\item by Nori connectness theorem for the pair $(Z_0\times S_d,D)$, 
since $d_Y\geq 4$ (hence $d_Y-1< 2d_Y-4$), $e>>0$, and $T\to S_d$ is smooth,
the map $k_{D_T}^*:H^{k-p}(Z_0\times T,\Omega^{p-1}_{Z_0\times T})\xrightarrow{\sim} H^{k-p}(D_T,\Omega^{p-1}_{D_T})$
is an isomorphism for all $0\leq k\leq d_Y$. 
\end{itemize}
Hence, by the diagramm (\ref{Nori1})
$i_{X_T}^*:H^{k-p}(Y\times T,\Omega^p_{Y\times T}(\log (Z_0\times T)))\xrightarrow{\sim} H^{k-p}(X_T,\Omega^p_{X_T}(\log D_T))$ 
is an isomorphism for all $0\leq k\leq d_Y$.

(ii): It follows from (i). We can also prove (ii) directly.
Indeed, we have the commutative diagram whose rows are long exact sequences :
\begin{equation}\label{Nori2}
\xymatrix{
\cdots\ar[r]  & H^{k}(Y\times T,\mathbb C)\ar[r]^{j_{(Y\backslash Z_0)\times T}^*}\ar[d]^{i_{X_T}^*}
 & H^{k}((Y\backslash Z_0)\times T,\mathbb C)\ar[r]^{Res}\ar[d]^{i_{U_T}^*} &
H^{k-1}(Z_0\times T,\mathbb C)\ar[r]^{i_{(Z_0\times T)*}}\ar[d]^{k_{D_T}^*} & \cdots \\
\cdots\ar[r]  & H^{d_Y}(X_T,\mathbb C)\ar[r]^{j_{U_T}^*} & H^{k}(U_T,\mathbb C)\ar[r]^{Res} &
H^{k-1}(D_T,\mathbb C)\ar[r]^{i_{D_T*}} & \cdots}
\end{equation}
Now,
\begin{itemize}
\item by Nori connectness theorem for the pair $(Y\times S_d,X)$,
since $d_Y\geq 4$ (hence $d_Y<2d_Y-2$), $d>>0$ and $T\to S_d$ is smooth, 
the map $i_{X_T}^*:H^{k}(Y\times T,\mathbb C)\xrightarrow{\sim} H^{k}(X_T,\mathbb C)$
is an isomorphism of mixed hodge structures for all $0\leq k\leq d_Y$, 
\item by Nori connectness theorem for the pair $(Z_0\times S_d,D)$, 
since $d_Y\geq 4$ (hence $d_Y-1<2d_Y-4$), $e>>0$, and $T\to S_d$ is smooth,
the map $k_{D_T}^*:H^{k-1}(Z_0\times T,\mathbb C)\xrightarrow{\sim} H^{k-1}(D_T,\mathbb C)$
is an isomorphism of mixed hodge structures for all $0\leq k\leq d_Y$. 
\end{itemize}
Hence, by the diagramm (\ref{Nori2})
$i_{U_T}^*:H^{k}((Y\backslash Z_0)\times T,\mathbb C)\xrightarrow{\sim} H^{k}(U_T,\mathbb C)$ an isomorphism
of mixed hodge structures for all $0\leq k\leq d_Y$.
\end{proof}

\textbf{A non vanishing criterion for an ample hypersurface of $Y\backslash Z_0$}

We will prove theorem \ref{cormain}. We begin by a lemma :

\begin{lem}\label{leraysurj}
Let $0\in S_e$ sufficiently general and $S\subset S_d$ the open subset over which
such that the morphisms $f^0:X\to S_d$ and $f^0_D:D\to S_d$ are smooth projective.
The map of filtered complexes of sheaves on $(Y\times S_d)^{an}$ 
\begin{equation*}
i_X^*:(\Omega^{\bullet}_{Y\times S_d}(\log(Z_0\times S_d)),L)\to(i_{X*}\Omega_X^{\bullet}(\log D),L)
\end{equation*}
induces a surjection of sheaves on $S^{an}$ 
\begin{equation*}
i_X^*:L^2R^{d_Y-p}p^0_{d*}\Omega^p_{Y\times S}(\log(Z_0\times S))\to L^2R^{d_Y-p}f^0_*\Omega^d_{X_S}(\log D).
\end{equation*}
\end{lem}

\begin{proof}
By Lefschetz theorem, the restriction morphism
$i_U^*:H^k(Y\backslash Z_0,\mathbb C)\to H^k(U_s,\mathbb C)$
is an isomorphism for $0\leq k<d_Y-1$ and is injective for $k=d_Y-1$ ($Y\backslash (Z_0\cup U_s)$ is a smooth affine variety).
Moreover it is a morphism of mixed hodge structures.
Hence, since the Fr\"olicher filtration is $E_1$ degenerate,
\begin{equation*}
i_X^*:\mathcal H^{l,m}_S(p_{Y\backslash Z_0})=R^mp^0_{d*}\Omega^l_{Y\times S/S}(\log(Z_0\times S))=E^{1,m}_1
\to \mathcal H^{l,m}_S(p_{Y\backslash Z_0})=R^mf^0_*\Omega^l_{X_S/S}(\log D)=E^{1,m}_1
\end{equation*}
is an isomorphism for $0\leq l+m<d_Y-1$ and is injective for $l+m=d_Y-1$.
\end{proof}

\begin{thm}\label{cormain}

Assume $d_Y\geq 4$
Let $0\in S_e$ sufficiently general and $S\subset S_d$ the open subset over which
such that the morphisms $f^0:X\to S_d$ and $f^0_D:D\to S_d$ are smooth projective.
Let $Z\in\mathcal Z^{p}(Y\backslash Z_0,2p-d_Y)^{pr/Y}_{\partial=0}$ such that 
$[\Omega_Z]\neq 0\in H^{d_Y}(Y\backslash Z_0,\mathbb C)$. Then for $s\in S$ general,
$AJ_{U_s}(Z_s):=[R'_{Z_s}]\neq 0\in J^{p,d_Y-1}(U_s)$.

\end{thm}

\begin{proof}
Consider the cycle $\tilde Z=i_U^*p_{Y\backslash Z_0}^{0*}Z\in\mathcal Z^{p}(U_S,2p-d_Y)^{pr/X,\hom/S}_{\partial=0}$.
We want to show that 
\begin{equation*}
\nu_{\tilde Z}\neq 0\in\Gamma(S,J^{p,d_Y-1}(f^0_U))=\Gamma(S,D_{O_S}^{\vee}(F^{d_Y-p}\mathcal H_S^{d_Y-1}(f^0_{X,D}))/H_{d_Y-1,\mathbb Z}(f^0_{X,D}). 
\end{equation*}
Since for all $s\in S$, 
\begin{equation*}
\nu_{\tilde Z}(s)=[R_{Z_s}]=ev_{X_s}(AJ_{U_s}(Z_s))\in F^{d_Y-p}H^{d_Y-1}(X_s,D_s,\mathbb C)^{\vee}/H_{d_Y-1}(X_s,D_s,\mathbb Z),
\end{equation*}
this will give the result because then $V(\nu_{\tilde Z})\subset S$ will be a proper analytic subset, 
even a proper algebraic subset by a result of Brossman, Pearlstein and Schnell.
By theorem \ref{main},
\begin{equation*}
\delta\mu_{\tilde Z}=\delta[\Omega_{\tilde Z}]\in\Gamma(S,\mathcal H^{p-1,d_Y-p}_S(f^0_U)\otimes_{O_S}\Omega_S/\Im\bar\nabla). 
\end{equation*}
Hence, it suffice to show that $\delta[\Omega_{\tilde Z}]\neq 0$.
Since, by the commutativity of (\ref{GMgrDGMgr}), we have, for all $s\in S$, $W(s)\subset S$ an open neighborhood of $s$ in $S$ and 
$\mu\in\Gamma(W(s),\mathcal H^{d_Y-p,p-1}_S(f^0_{X,D})\otimes_{O_S}T_S)^{\bar\nabla=0}$,
\begin{equation*}
<\delta[\Omega_{\tilde Z}]_{|W(s)},\mu>_{ev_f}=<[\Omega_{\tilde Z}/L^2]_{|W(s)},\gamma>_{f_*ev_X} 
\end{equation*}
where $\gamma\in\Gamma(W(s),R^{p-1}f^0_*L^{d_S-1}\Omega^{d_Y-p+d_S}_X(\nul D))$ is such that $r(\gamma)=\mu$,
it suffice to show that $[\Omega_{\tilde Z}/L^2]\neq 0\in\Gamma(S,R^{d_Y-p}f^0_*(\Omega^p_X(\log D)/L^2))$.
Since the map from the Leray spectral sequence of associated to the filtered complex 
$(\Omega^{\bullet}_X(\log D),L)$ (c.f.diagramm (\ref{phi12r}))
\begin{equation*}
\psi^2_L:\Gamma(S,(L^1R^{d_Y-p}f^0_*\Omega^p_X(\log D))/L^2)=E_{\infty}^{1,d_Y-p}
\hookrightarrow\Gamma(S,R^{d_Y-p}f^0_*(\Omega^p_X(\log D)/L^2))^{q=0}=E_1^{1,d_Y-p}
\end{equation*}
is injective, it suffice to show that $[\Omega_{\tilde Z}]/L^2\neq 0\in\Gamma(S,(L^1R^{d_Y-p}f^0_*\Omega^p_X(\log D))/L^2)$.

So, suppose that $[\Omega_{\tilde Z}]/L^2=0$, that is $[\Omega_{\tilde Z}]\in\Gamma(S,L^2R^{d_Y-p}f^0_*\Omega^p_X(\log D))$.
By the lemma \ref{leraysurj}, since $S\subset S_d$ is affine, 
there exist 
\begin{equation*}
\alpha\in\Gamma(S,L^2R^{d_Y-p}p^0_{d*}\Omega^p_{Y\times S}(\log(Z_0\times S)))
\end{equation*}
such that $i_X^*\alpha=[\Omega_{\tilde Z}]$.
Since $S$ is affine, the canonical maps
\begin{itemize}
\item $L^2H^{d_Y-p}(Y\times S,\Omega^p_{Y\times S}(\log(Z_0\times S)))\xrightarrow{\sim}
\Gamma(S,L^2R^{d_Y-p}p^0_{d*}\Omega^p_{Y\times S}(\log(Z_0\times S)))$ and
\item $L^2H^{d_Y-p}(X_S,\Omega^p_X(\log D))\xrightarrow{\sim}\Gamma(S,L^2R^{d_Y-p}f^0_*\Omega^p_X(\log D))$
\end{itemize}
are isomorphisms.
Hence, seeing $\alpha\in L^2H^{d_Y-p}(Y\times S,\Omega^p_{Y\times S}(\log(Z_0\times S)))$,
\begin{equation*}
i_X^*\alpha=[\Omega_{\tilde Z}]=i_X^*p^{0*}_Y[\Omega_Z]\in H^{d_Y-p}(X_S,\Omega_{X_S}^p(\log D)),
\end{equation*}
that is $i_X^*(\alpha-p^{0*}_Y[\Omega_Z])=0\in H^{d_Y-p}(X_S,\Omega_{X_S}^p(\log D)$.
But since $p^{0*}_Y[\Omega_Z]\notin L^2H^{d_Y-p}(Y\times S,\Omega^p_{Y\times S}(\log(Z_0\times S)))$,
$\alpha-p^{0*}_Y[\Omega_Z]\neq 0\in H^{d_Y-p}(Y\times S,\Omega^p_{Y\times S}(\log(Z_0\times S)))$.
But by the theorem \ref{nori} (i), since $S\hookrightarrow S_d$ is smooth,
\begin{equation*}
i_X^*:H^{d_Y-p}(Y\times S,\Omega^p_{Y\times S}(\log(Z_0\times S)))\xrightarrow{\sim} H^{d_Y-p}(X_S,\Omega_{X_S}^p(\log D)),
\end{equation*}
is an isomorphism. We get a contradiction.
\end{proof}

\textbf{The image of the Abel Jacobi map of an ample hypersurface of $Y\backslash Z_0$}

\begin{thm}\label{ImAJUs}
Assume $d_Y\geq 4$.
Let $0\in S_e$ sufficiently general and $S\subset S_d$ the open subset over which
such that the morphisms $f^0:X\to S_d$ and $f^0_D:D\to S_d$ are smooth projective.
Consider the commutative diagram
\begin{equation}\label{ImAJUsdiag}
\xymatrix{
\CH^p(Y\backslash Z_0,2p-d_Y,\mathbb Q)\ar^{i_{U_s}^*}[r]\ar^{\mathcal R(Y,Z_0)}[d] &
\CH^p(U_s,2p-d_Y,\mathbb Q)\ar^{\overline{\mathcal R(X_s,D_s)}}[d] \\
H^{d_Y}_{\mathcal D}(Y,Z_0,\mathbb Q)\ar[r] & H^{d_Y}_{\mathcal D}(X_s,D_s,\mathbb Q)/J^{p,d_Y}(Y\backslash Z_0)_{\mathbb Q}}
\end{equation}
Then for a general point $s\in S$, $\Im(\overline{\mathcal R^p(X_s,D_s)})=\Im(\overline{\mathcal R^p(X_s,D_s)}\circ i_{U_s}^*)$. 
\end{thm}

\begin{proof}
We follow \cite{GrMS}.
Let $s\in S$ a general point and $Z_s=\sum_{i=1}^k n_i Z_{is}\in\mathcal Z^p(U_s,n,\mathbb Q)^{pr/X_s}_{\partial=0}$. 
Then, there exists a branched covering $h:T\to S_d$, $t\in h^{-1}(s)$, and $Z\in\mathcal Z^p(X_T,n,\mathbb Q)$
such that 
\begin{itemize}
\item $h^{-1}(s)=\left\{t,t_1,\cdots,t_r\right\}\subset T_0$, where $T_0\subset T$ is the Zariski open subset such that $h:T_0\to S_d$ is smooth,
\item $\partial Z\in i_{D_T*}\mathcal Z^{p-1}(D_T,n-1,\mathbb Q)$, 
\item $Z\cdot(X_t\times\left\{t\right\})=\bar Z_s$,
\end{itemize}
with $X_T=X\times_{S_d} T$ and $D_T=D\times_{S_d} T$.
For this, consider, for each $1\leq i\leq k$, 
the relative Hilbert scheme $h_i:H_i\to S_d$ of $f(\square^n):\mathcal X\times\square^n\to S_d$,
such that $\bar Z_{is}$ belongs to and $h:H\hookrightarrow H_1\times_{S_d}\cdots\times_{S_d}H_k\to S_d$
defining the condition $\partial G_{is}\in i_{D*}\mathcal Z^{p-1}(D,n)$. 
Note that $H\to S_d$ is surjective since there always exist such a cycle in a fiber and $s\in S_d$ is general.
Take a multisection $T\hookrightarrow H\to S_d$ of $h$ such that
$h^{-1}(s)\cap T\cap\sing(h)=\emptyset$, where $\sing(h)$ is the singular locus of $h$, 
and such that the intersection $h^{-1}(s)\cap T\subset H$ is transversal.  

Denote by $C=\left\{t,t_1,\cdots,t_{r'}\right\}\subset h^{-1}(s)$, with $1\leq r'\leq r$, the subset such that
$Z_{t_i}\subset X_s$ is not included in $D_s$.
By theorem \ref{nori}(ii),
\begin{equation*}
i_{U_T}^*:H_{\mathcal D}^{d_Y}((Y\backslash Z_0)\times T,\mathbb C)\xrightarrow{\sim} H_{\mathcal D}^{d_Y}(U_T,\mathbb C)
\end{equation*}
is an isomorphism and in particular surjectif. Hence, there exist  
$\gamma\in H_{\mathcal D}^{d_Y}((Y\backslash Z_0)\times T,\mathbb C)$ such that
$\mathcal R((Y\backslash Z_0)\times T)(Z_T)=i_{U}^*\gamma$.
Hence, for $t_i,t_j\in C$,
\begin{equation*}
\mathcal R(X_s,D_s)(Z_{t_i})-\mathcal R(X_s,D_s)(Z_{t_j})\in i_{U_s}^*J^{p,d_Y-1}(Y\backslash Z_0)
\end{equation*}
This gives the equality
\begin{equation}\label{Nsurj}
\overline{\mathcal R(X_s,D_s)}(\sum_{t_i\in C}Z_{t_i})=\overline{\mathcal R(X_s,D_s)}(\sum_{t_i\in h^{-1}(s)}Z_{t_i})
=r'\overline{\mathcal R(X_s,D_s)}(\bar Z_s)
\end{equation}
Consider now a pencil $\Lambda_d\subset S_d$ such that $s\in\Lambda_d$,
and $\hat T=h^{-1}(\Lambda_d)\subset T$.
\begin{itemize}
\item In $Y\times\hat T$ we have 
$(X_s\times\hat T).X_{\hat T}=\sum_{i=1}^rX_s\times\left\{t_i\right\}+B(\Lambda_d)\times\hat T$
\item In $(Y\backslash Z_0)\times\hat T$ we have 
$(U_s\times\hat T).U_{\hat T}=\sum_{i=1}^{r'}U_s\times\left\{t_i\right\}+(B(\Lambda_d)\cap(Y\backslash Z_0))\times\hat T$
\end{itemize}
where $X_{\hat T}=X\times_{S_d}{\hat T}$, $U_{\hat T}=U\times_{S_d}{\hat T}$ and 
$B(\Lambda_d)=X_s\cap X_{s'}\subset Y$, $s'\neq s\in S_d$ is the base locus of the pencil.
Consider
\begin{itemize}
\item $Z_{\hat T}=Z\cdot X_{\hat T}\in\mathcal Z^p(X_{\hat T},2p-d_Y,\mathbb Q)$ and 
\item ${Z_{\hat T}}_{|{U_{\hat T}}}:=j^*_{U_{\hat T}}Z_{\hat T}=(j_{U_T}^*Z).U_{\hat T}\in\mathcal Z^p(U_{\hat T},2p-d_Y,\mathbb Q)$
its restriction.
\end{itemize}
We may assume, adding a boundary if necessary, that 
\begin{equation*}
(j_{U_T}^*Z)\cap((B(\Lambda_d)\cap(Y\backslash Z_0))\times\hat T)
:=(j_{U_T}^*Z).(U_s\times\hat T).(U_{s'}\times\hat T)
\in\mathcal Z^{p+3}((Y\backslash Z_0)\times\hat T,2p-d_Y),
\end{equation*}
that is the intersection is a Bloch cycle of the appropriate codimension. 
By the projection formula, we have, denoting $p^{\hat T}_{Y\backslash Z_0}:(Y\backslash Z_0)\times\hat T\to Y\backslash Z_0$
the projection (which is proper since $\hat T$ is projective),
\begin{equation}\label{intUs}
(p^{\hat T}_{Y\backslash Z_0*}{Z_{\hat T}}_{|U_{\hat T}}).U_s=p^{\hat T}_{Y\backslash Z_0*}((j_{U_T}^*Z).U_{\hat T}.(U_s\times\hat T)) 
=\sum_{t_i\in C}j_{U_s}^*Z_{t_i}+(p^{\hat T}_{Y\backslash Z_0*}((j_{U_T}^*Z).(U_{s'}\times\hat T)))\cdot U_s
\end{equation}
Finaly, we obtain,
\begin{eqnarray*}
\overline{\mathcal R(X_s,D_s)}(Z_s)&=&\frac{1}{r'}\overline{\mathcal R(X_s,D_s)}(\sum_{t_i\in C}j_{U_s}^*Z_{t_i}) \mbox{\; by \; (\ref{Nsurj})} \\
&=&\frac{1}{r'}\overline{\mathcal R(X_s,D_s)}\circ i_{U_s}^*(p^{\hat T}_{Y\backslash Z_0*}{Z_{\hat T}}_{|U_{\hat T}}-
p^{\hat T}_{Y\backslash Z_0*}((j_{U_T}^*Z).(U_{s'}\times\hat T)) \mbox{\; by \; (\ref{intUs})}
\end{eqnarray*}
This gives the theorem.
\end{proof}

%%%%%%%%%%%%%%%%%%%%%%%%%%%%%%%%%%%%%%%%%%%%%%%%%%%%%%%%%%%%%%%%%%%%%%%%%%%%%%%%%%%%%%%%%
%%%%%%%%%%%%%%%%%%%%%%%%%%%%%%%%%%%%%%%%%%%%%%%%%%%%%%%%%%%%%%%%%%%%%%%%%%%%%%%%%%%%%%


\begin{thebibliography}{1}
\bibitem{AS} M.Asakura and S.Saito \emph{Generalized Jacobi rings for open complete intersections, Math. Nachr. 2004}
\bibitem{GrMS} M.Green and S.M\"uller-Stach 
\emph{Algebraic cycles on a general complete intersection of high multi-degree of a smooth projective variety. Compositio Mathematica. 1996}
\bibitem{Jansen} U.Jansen \emph{Deligne homology, Hodge-D-conjecture and motives. 1988}
\bibitem{MKerr} M.Kerr \emph{Geometric construction of regulator currents with application to algebraic cycles. Princeton University thesis. 2003}
\bibitem{MKerrLMS} M.Kerr, J.Lewis and S.M\"uller-Stach \emph{The Abel Jacobi map for higher Chow groups. Compositio Mathematica. 2006}
\bibitem{King} J.R.King \emph{Log complexes of currents and the functorial properties of the Abel Jacobi map. Duke mathematical Journal. 1983}
\bibitem{PS} C.Peters and J.Steenbrink, \emph{Mixed Hodge Structures. Volume 52 Springer. 2007}
\bibitem{Voisin} C.Voisin, \emph{Th\'eorie de Hodge et and g\'eometrie algebrique complexe. Cours sp\'ecialis\'e. 2002}

\end{thebibliography}
\end{document}